\documentclass[12pt]{amsart}
\usepackage[utf8]{inputenc}
\usepackage[headings]{fullpage}
\usepackage{amsmath, amssymb, bbm, faktor}
\usepackage[nobysame,alphabetic]{amsrefs}
\usepackage{amsthm}
\usepackage{mathtools}
\usepackage{bbm}
\usepackage{blkarray}
\usepackage[matrix,arrow,curve]{xy}
\usepackage{color}

\usepackage{tikz}
	\usetikzlibrary{decorations.pathreplacing}
	\usetikzlibrary{decorations.pathmorphing}
	\usetikzlibrary{patterns}
\usepackage{caption}

\usepackage{tikz-cd}
\usepackage{graphicx}
\definecolor{darkblue}{rgb}{0,0,0.6}
\usepackage[ocgcolorlinks,colorlinks=true, citecolor=darkblue, filecolor=darkblue, linkcolor=darkblue, urlcolor=darkblue]{hyperref}
\usepackage[capitalize,noabbrev]{cleveref}
\usepackage{comment}
\usepackage{enumitem}
\usepackage{euscript}
\usetikzlibrary{decorations.markings,arrows.meta,decorations.text}
\makeatletter
\DeclareRobustCommand*{\mfaktor}[3][]
{
   { \mathpalette{\mfaktor@impl@}{{#1}{#2}{#3}} }
}
\newcommand*{\mfaktor@impl@}[2]{\mfaktor@impl#1#2}
\newcommand*{\mfaktor@impl}[4]{
   \settoheight{\faktor@zaehlerhoehe}{\ensuremath{#1#2{#3}}}%
   \settoheight{\faktor@nennerhoehe}{\ensuremath{#1#2{#4}}}%
      \raisebox{-0.5\faktor@zaehlerhoehe}{\ensuremath{#1#2{#3}}}%
      \mkern-4mu\diagdown\mkern-5mu%
      \raisebox{0.5\faktor@nennerhoehe}{\ensuremath{#1#2{#4}}}%
}
\makeatother

\makeatletter
\newtheorem*{rep@theorem}{\rep@title}
\newcommand{\newreptheorem}[2]{%
	\newenvironment{rep#1}[1]{%
		\def\rep@title{#2 \ref{##1}}%
		\begin{rep@theorem}}%
		{\end{rep@theorem}}}
\makeatother

\newcommand{\immto}{\looparrowright}
\newcommand{\wt}{\widetilde}
\newcommand{\wh}{\widehat}
\renewcommand{\H}{\mathbb{H}}
\newcommand{\R}{\mathbb{R}}
\newcommand{\C}{\mathbb{C}}
\newcommand{\Z}{\mathbb{Z}}
\newcommand{\N}{\mathbb{N}}
\newcommand{\Q}{\mathbb{Q}}
\newcommand{\CP}{\mathbb{CP}}
\newcommand{\RP}{\mathbb{RP}}
\newcommand{\Op}{\mathcal{O}}
\newcommand{\U}{\mathrm{U} }

\newcommand{\abs}[1]{\left\lvert #1 \right\rvert}

\DeclareMathOperator{\QCA}{QCA}
\DeclareMathOperator{\Sq}{Sq}
\newcommand{\SO}{\mathrm{SO}}
\newcommand{\Spin}{\mathrm{Spin}}
\DeclareMathOperator{\SL}{SL} 

\DeclareMathOperator{\hCh}{hCh}
\DeclareMathOperator{\sCh}{sCh}

\DeclareMathOperator{\End}{End}
\DeclareMathOperator{\supp}{supp}
\DeclareMathOperator{\range}{range}
\newcommand{\lbar}[1]{\overline{#1}}

\DeclareMathOperator{\Ext}{Ext}
\DeclareMathOperator{\Mono}{Mono}

\DeclareMathOperator{\Imm}{Imm}
\DeclareMathOperator{\Hom}{Hom}

\DeclareMathOperator{\Sub}{Sub}
\DeclareMathOperator{\Sur}{Sur}

\newcommand{\proj}{\mathrm{pr}}
\newcommand{\pr}{\prime}
\newcommand{\fr}{\mathrm{fr}}
\newcommand{\id}{\mathrm{id}}

\DeclareMathOperator{\coker}{coker}

\DeclareMathOperator{\Aut}{Aut}

\DeclareMathOperator{\im}{im}
\DeclareMathOperator{\shift}{shift}

\newcommand{\imra}{\looparrowright}
\newcommand{\ra}{\longrightarrow}
\newcommand{\hra}{\hookrightarrow}
\newcommand{\sra}{\twoheadrightarrow}

\newcounter{commentcounter}

\numberwithin{equation}{section}

\newtheorem{thm}[equation]{Theorem}
\newtheorem{lem}[equation]{Lemma}
\newtheorem{cor}[equation]{Corollary}
\newtheorem{prop}[equation]{Proposition}
\theoremstyle{definition}
\newtheorem{defi}[equation]{Definition}
\newtheorem{rem}[equation]{Remark}

\newreptheorem{thm}{Theorem}
\newreptheorem{lem}{Lemma}
\newreptheorem{prop}{Proposition}
\newreptheorem{cor}{Corollary}

\AddToHook{env/thm/begin}{\crefalias{equation}{thm}}
\AddToHook{env/lem/begin}{\crefalias{equation}{lem}}
\AddToHook{env/cor/begin}{\crefalias{equation}{cor}}
\AddToHook{env/prop/begin}{\crefalias{equation}{prop}}
\AddToHook{env/defi/begin}{\crefalias{equation}{defi}}
\AddToHook{env/rem/begin}{\crefalias{equation}{rem}}
\AddToHook{env/ex/begin}{\crefalias{equation}{ex}}

\crefname{thm}{Theorem}{Theorems}
\crefname{lem}{Lemma}{Lemmas}
\crefname{cor}{Corollary}{Corollaries}
\crefname{prop}{Proposition}{Propositions}
\crefname{defi}{Definition}{Definitions}
\crefname{rem}{Remark}{Remarks}
\crefname{ex}{Example}{Examples}

\renewcommand{\phi}{\varphi}

\begin{document}

\title[Immersions of punctured $4$-manifolds]{Immersions of punctured $4$-manifolds: with applications to Quantum Cellular Automata}

\author{Michael Freedman}
\address{
	Michael Freedman \\
	Microsoft Research, Station Q, and
	Department of Mathematics\\
	University of California \\
	Santa Barbara, CA 93106
}
\email{michaelf@microsoft.com}

\author{Daniel Kasprowski}
\address{\hskip-\parindent
	Daniel Kasprowski \\
	School of Mathematical Sciences\\
	University of Southampton \\
	Southampton SO17 1BJ, United Kingdom
}
\email{d.kasprowski@soton.ac.uk}

\author{Matthias Kreck}
\address{\hskip-\parindent
	Matthias Kreck \\
	Hausdorff Center for Mathematics \\
	53115 Bonn, Germany
}
\email{kreck@math.uni-bonn.de}

\author{Peter Teichner}
\address{\hskip-\parindent
	Peter Teichner \\
	Max Planck Institute for Mathematics \\
	53111 Bonn, Germany
}
\email{teichner@mpim-bonn.mpg.de}

\makeatletter
\let\@wraptoccontribs\wraptoccontribs
\makeatother
\contrib[with an appendix by]{Alan W. Reid}
\address{\hskip-\parindent
	Alan W. Reid\\
	Department of Mathematics\\
	Rice University\\
	Houston, TX 77005, USA.\\
	Max-Planck-Insititut f\"ur Mathematik\\
	Vivatsgasse 7\\
	D-53111 Bonn, Germany.}
\email{alan.reid@rice.edu, areid@mpim-bonn.mpg.de}
	\begin{abstract}
	Motivated by applications to pulling back quantum cellular automata from one manifold to another, we study the existence of immersions between certain smooth $4$-manifolds. We show that they lead to a very interesting partial order on closed $4$-manifolds. 
		\end{abstract}
	\maketitle

%\tableofcontents

\section{Introduction}

In this paper manifolds are closed, connected and smooth (ccs), unless otherwise stated. A notable exception is the punctured manifold $M_* := (M \smallsetminus$point) of a ccs manifold $M$.  
Given ccs manifolds $M,N$ of the same dimension, we consider the question whether $M_* $ immerses into $N$. To answer this question, we use a result of Phillips \cite{phillips} who constructed an immersion $M_* \imra N$, in the spirit of Smale--Hirsch theory \cites{smale1,smale2,hirsch2}, from a fibrewise isomorphism of tangent bundles (that is a priori not the differential of an immersion). We give a detailed formulation at the end of the introduction in \cref{thm:phillips}.  Besides some simple observations in dimensions 2 and 3 (see \cref{sec:23}) we study this question in dimension 4. In Section~\ref{sec:general-intro} we give a complete answer in terms of certain invariants. We apply this to exhibit the resulting partial order for $4$-manifolds with cyclic fundamental groups in Section~\ref{sec:cyclic-intro} and - a particularly interesting order - when the fundamental group is $\mathbb Z^4$ in \cref{sec:Z4}. 

An immersion $M_* \imra N$ can be used to construct quantum cellular automata (QCA) on $M$ out of QCAs on $N$. The first author will explain QCAs in the \hyperref[sec:appendix]{appendix} and also discuss how to pull them back. This construction, due to Hastings \cite{hastings13}, motivated our entire study.

\begin{defi} Let $M,N$ be css manifolds of the same dimension.
\begin{enumerate}
	\item Write $M\le N$ if and only if there is an immersion $M_* \imra N$. 
	\item If $M\leq N$ and $N\leq M$ we write $M \bowtie N$ and say that $M,N$ are \emph{immersion equivalent}. 
\end{enumerate}
\end{defi}

The relation $\leq$ is transitive (\cref{lem:open}) and by construction, it becomes a partial order on immersion equivalence classes. Clearly $S^n \leq M$ for all $n$-manifolds $M$.

The Stiefel-Whitney classes of the stable \emph{normal bundle} $\nu (M)$ of $M$ are abbreviated by $w_i(M):= w_i(\nu (M))\in H^i(M;\Z/2)$ as is common in surgery theory (and translates to classes for the tangent bundle $TM$ via $w_1(TM)=w_1(M), w_2(TM) = w_2(M)+w_1(M)^2$). The relevant classes for $4$-manifolds are only $w_1,w_2$ as explained in Remark~\ref{rem:wi}. An immersion $j\colon M\imra N$ gives $j^*(w_i(N) ) = w_i(M)$ and hence any relation among the $w_i(N)$ are mirrored among the $w_i(M)$ via $j$, explaining the 'only if' direction in all items below. 

\begin{prop} \label{prop:easy} 
If $M$ is a $4$-dimensional css manifold then
\begin{enumerate}
\item\label{it:easy1} $M\leq S^4 \iff M$ is spin, i.e.\ $w_1(M)=w_2(M)=0$.

\item\label{it:easy2} $\CP^2 \le M \iff$ $M$ is not almost spin, i.e.\ $w_2(\widetilde M)\neq 0$.

\item\label{it:easy3} $M \le \CP^2 \iff M$ is orientable, i.e.\ $w_1(M)=0$. 

\item\label{it:easy4} $S^1\wt \times S^3\leq M \iff$ $w_1(M)\neq 0$, where $S^1\wt\times S^3$ is the non-trivial $S^3$-bundle over $S^1$.

\item\label{it:easy5} $M\leq S^1\wt\times S^3 \iff$ $w_2(M)=0$ and $w_1(M)$ admits an integral lift.
\end{enumerate}
\end{prop}

In particular, every spin manifold is a minimum for this order, $\CP^2$ is a maximum among orientable $4$-manifolds, and $S^1\wt\times S^3$ is a minimum among non-orientable $4$-manifolds. It is an open question whether there is a maximum among all $4$-manifolds. However, there are further ``local'' maxima and minima for certain classes of $4$-manifolds. For example, $M\leq \RP^4$ if and only if $M$ (or more precisely $\nu (M)$) is pin$^-$, i.e.\ $w_2(M)=(w_1(M))^2$.

Let $c\colon M\to B\pi_1M$ be a map that is the identity on fundamental groups. If $M$ is \emph{almost spin}, i.e.\ $w_2(\widetilde M)=0$, then there is a unique group cohomology class $w_2^{\pi} (M)\in H^2(\pi_1M;\Z/2)$ with $c^*(w_2^\pi (M))=w_2(M)$ since $H^2(\pi;\Z/2)\to H^2(M;\Z/2)\to H^2(\wt M;\Z/2)$ is exact. There is always a unique element $w_1^\pi (M)$ with $c^*(w_1^\pi (M))=w_1(M)$. To have a unified notation, we set $w_2^\pi (M):=\infty$ if $M$ is not almost spin. Recall from \cite{teichnerthesis}*{Theorem~2.1.1 and~2.2.1} that the resulting triple $(\pi_1M, w_1^\pi (M), w_2^\pi (M))$  is equivalent to the \emph{normal 1-type} as introduced in \cite{surgeryandduality} for the purpose of $S^2 \times S^2$-stable classification of $4$-manifolds. 

We let $\pm [M]\in H_4(M;\Z^{w_1(M)})\cong \Z$ be the (unsigned) twisted fundamental class. Note that a homomorphism $w:\pi\to\{\pm 1\}$ induces an automorphism of the group ring $\Z[\pi]$ by sending $g\mapsto w(g)\cdot g$. For a $\Z[\pi]$-module $P$, we denote by $P^w$ the module twisted by this automorphism.
\begin{prop}
	\label{cor:stable}
	Assume that  $\varphi\colon \pi_1M\overset{\cong}{\ra} \pi_1N $ is an isomorphism such that
	\begin{enumerate}
		\item $\varphi^*(w_1^\pi (N)) = w_1^\pi (M)$,
		\item $\varphi^*(w_2^\pi (N) )= w_2^\pi (M) \in H^2(\pi_1M;\Z/2) \cup \{\infty\}$, where we set $\varphi^*(\infty):=\infty$, and
		\item $\varphi_*c_*[M]=\pm c_*[N]\in H_4(\pi_1N;\Z^{w_1^\pi (N)})$.
	\end{enumerate}
	Then $M$ and $N$ are immersion equivalent with $M_*\imra N$ inducing $\phi$ on fundamental groups.
\end{prop}

\begin{defi}\label{def:imm}
 A \emph{$\phi$-immersion} is an immersion $j:M_*\imra N$ inducing $\phi\colon \pi_1M\to\pi_1N $ on fundamental groups, i.e.\ with $\pi_1(j)=\phi$.
\end{defi}

It follows that the question whether $M_1\leq M_2$ only depends on the two quadruples $(\pi_1(M_k), w_1^\pi (M_k), w_2^\pi (M_k), c_*[M_k])$ for $k=1,2$, up to the actions of $\Aut(\pi_1(M_k))$ and up to sign in the last entry. Hence we call
\[
(\pi_1M, w_1^\pi (M), w_2^\pi (M), \pm c_*[M])
\]
the \emph{immersion type} of $M$. Given a finitely presented group $\pi$ every pair $(w_1,w_2)$ can be realised as $(w_1^\pi (M),w_2^\pi (M))$ for some $4$-manifold $M$, see \cref{sec:james}. We will also see that a finite index subgroup of $H_4(\pi;\Z^{w_1})$ is realised as images $c_*[M]$ of $4$-manifolds $M$ with given $(w_1,w_2)$. If $w_2=\infty$, this is the entire group but in general the index is an arbitrary power of~2.

\subsection{$4$-manifolds with cyclic fundamental group} \label{sec:cyclic-intro}

To show that the immersion order is indeed interesting, we now describe it for all manifolds with cyclic fundamental group. For this we introduce some \emph{order graphs}, where we use the notation that an arrow $\begin{tikzcd}[every arrow/.append style={no head,"\blacktriangleright" marking}]{}\ar[r]&{}\end{tikzcd}$ represents $<$ which means \emph{strictly} smaller in the immersion order.
In \cref{sec:cyclic}, we give representatives with minimal Euler characteristic for each immersion equivalence class that will appear below. In particular, the oriented immersion types $M(2^m)$ used below will be represented by rational homology 4-spheres with fundamental group $\Z/2^m$. 

\begin{thm}\label{thm:cyclic-or}
	\begin{enumerate}
		\item Any two orientable $4$-manifolds with fundamental group $\Z/2^m$ that are almost spin but not spin are immersion equivalent. We denote this immersion equivalence class by $M(2^m)$.
		\item Any orientable $4$-manifold $M$ with fundamental group $\Z/2^mk$, $k$ odd, is immersion equivalent to $S^4, M(2^m)$ or $\CP^2$ depending on $w_2^\pi (M)$.
		\item For infinite cyclic fundamental group, there are exactly two immersion equivalence classes represented by $S^1\times S^3$, which is immersion equivalent to $S^4$, and $S^1\times S^3 \#\CP^2$, which is immersion equivalent to $\CP^2$.
	\end{enumerate}
	Furthermore, the immersion equivalence classes of orientable $4$-manifolds with cyclic fundamental group form the following infinite chain.
	\[S^4< M(2)< M(4)<\ldots< M(2^m)< \ldots< \CP^2\]
	In terms of order graphs this look as follows.
	\[\begin{tikzcd}[every arrow/.append style={no head,"\blacktriangleright" marking}]
		S^4\ar[r]&M(2)\ar[r]&M(4)\ar[r]&\ldots \ar[r]&M(2^m)\ar[r]&\ldots\ar[r]&\CP^2
	\end{tikzcd}\]
	%In particular, any orientable $4$-manifold $M$ with fundamental group $\Z/2^mk$, $k$ odd, is immersion equivalent to $S^4, M(2^m)$ or $\CP^2$ depending on $w_2^\pi (M)$. For infinite cyclic fundamental group, there are exactly two immersion equivalence classes represented by $S^1\times S^3$, which is immersion equivalent to $S^4$, and $S^1\times S^3 \#\CP^2$, which is immersion equivalent to $\CP^2$.
\end{thm}

As explained in \cref{rem:realization}, the non-oriented immersion types $N(2^m,1,1)$ used below are represented by rational homology 4-balls with fundamental group $\Z/2^m$, with $\RP^4$ coming up for $m=1$. All cases with $w_2=\infty$ are realized by a connected sum with $\CP^2$.

\begin{thm}\label{thm:cyclic-nor}
	Let $N(2^m,w_2,c)$ denote a non-orientable $4$-manifold with fundamental group $\pi=\Z/2^m$, $w_2\in\{0,1,\infty\}$ and image of the fundamental class $c$ in $H_4(\pi;\Z^{w_1})=\Z/2$ such that $c=0$ or $w_2\neq 0$. Note that in particular no such manifold exists with $c=1$ and $w_2=0$. 
	Then immersion equivalence classes of non-orientable $4$-manifolds with cyclic fundamental group form the following order graph.
 {\scriptsize
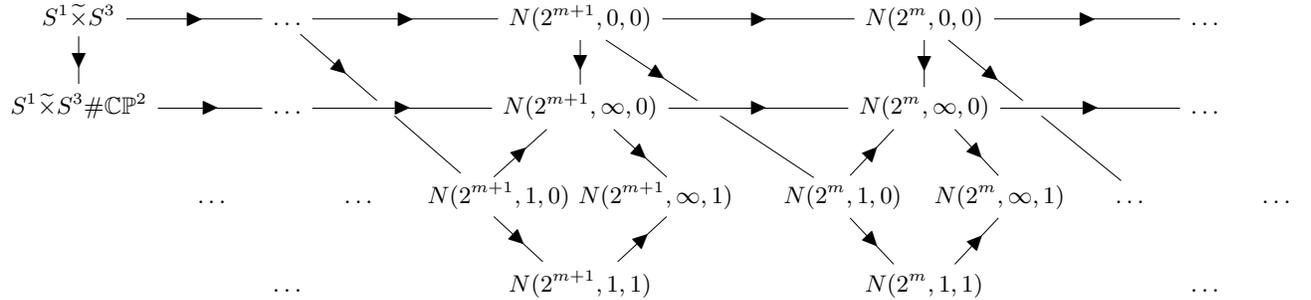
 	\[\begin{tikzcd}[column sep=small,every arrow/.append style={no head,"\blacktriangleright" marking}]
 		S^1\wt\times S^3 \ar[rr]\ar[d]	 &&\dots \ar[ddrr, near start,end anchor={[xshift=4ex]}]\ar[rrr]&&&N(2^{m+1},0,0)\ar[ddrr, near start,end anchor={[xshift=4ex]}]\ar[d]\ar[rrr]&&&N(2^m,0,0)\ar[d]\ar[rrr]\ar[rrdd, near start, end anchor={[yshift=1ex]}]&&&\ldots&\\
 		S^1\wt\times S^3\#\CP^2\ar[rr]&&\dots\ar[rrr, crossing over]&&&N(2^{m+1},\infty,0)\ar[dr,end anchor={[xshift=-3ex,yshift=0.5ex]}]\ar[rrr, crossing over]&&&N(2^m,\infty,0)\ar[dr,end anchor={[xshift=-3ex,yshift=0.5ex]}]\ar[rrr, crossing over]&&&\ldots&\\
 		&\makebox[1em][l]{$\dots$}&&\makebox[1em][r]{$\dots$}&\makebox[1em][l]{$N(2^{m+1},1,0)$}\ar[ur,start anchor={[xshift=3ex,yshift=0.5ex]}]\ar[dr,start anchor={[xshift=3ex,yshift=-0.5ex]}]&&\makebox[1em][r]{$N(2^{m+1},\infty,1)$}&\makebox[1em][l]{$N(2^m,1,0)$}\ar[ur,start anchor={[xshift=3ex,yshift=0.5ex]}]\ar[dr,start anchor={[xshift=3ex,yshift=-0.5ex]}]&&\makebox[1em][r]{$N(2^m,\infty,1)$}&\makebox[1em][l]{$\dots$}&&\makebox[1em][r]{$\dots$}\\
 		&&\dots&&&N(2^{m+1},1,1)\ar[ur,end anchor={[xshift=-3ex,yshift=-0.5ex]}]&&&N(2^m,1,1)\ar[ur,end anchor={[xshift=-3ex,yshift=-0.5ex]}]&&&\ldots&
 	\end{tikzcd}\]}
 	Our order graph ends on the right hand side with the four manifolds $N(2,-,-)$ but repeats infinitely on the left (with increasing $m$) before encountering the manifolds with infinite fundamental group.
\end{thm}

In \cref{sec:cyclic} we also determine the combined immersion order for orientable and non-orientable $4$-manifolds with cyclic fundamental group.

\subsection{$4$-manifolds with free abelian fundamental group}\label{sec:Z4}

Another class of $4$-manifolds representing an interesting suborder is that of orientable manifolds with free abelian fundamental group $\Z^4$. For almost spin manifolds, $w_2^\pi (M)$ lies in  
\[
H^2(\Z^4;\Z/2)/\!\Aut(\Z^4) = \{0, e_1 \cup e_2, e_1 \cup e_2+e_3 \cup e_4\},
\]
where $e_i\in H^1(\Z^4;\Z/2)$ are the four standard pullbacks of the generator in $H^1(\Z;\Z/2)$. 

\begin{thm}
	\label{prop:Z4}
	Let $M, N$ be oriented, almost spin $4$-manifolds with fundamental group $\Z^4$. 
	\begin{enumerate}
		\item If $[w_2^\pi (M)] = e_1 \cup e_2$, then $M\le N$ if and only if $N$ is non-spin.
		\item If $[w_2^\pi (M)] = e_1 \cup e_2+ e_3 \cup e_4$, then $M\le N$ if and only if $[w_2^\pi (N)]=[w_2^\pi (M)]$  and $c_*[M]$ is a multiple of $c_*[N]$ in $H_4(\Z^4)/\!\Aut(\Z^4) \cong \Z/\!\{\pm 1\} = \N_0$.
	\end{enumerate}
Moreover, all immersion types $(\Z^4,0,e_1\cup e_2+e_3\cup e_4,2c)$ are realized.
\end{thm}

This implies that the order $(\N_0,\cdot)$ given by ($m\le n\Leftrightarrow m$ is a multiple of $n$) is realized in the immersion order on orientable $4$-manifolds. It would be interesting to know whether the division order on $\N$ also arises, where $m\le n$ if and only if $m$ divides $n$.

\subsection{$4$-manifolds with arbitrary fundamental groups} \label{sec:general-intro}

The following result follows from obstruction theory (for basics about obstruction theory see \cite{Baues}) and a version of Smale--Hirsch theory by Phillips \cite{phillips} that works for open manifolds of the same dimension. 
\begin{rem}
	We use the \emph{Postnikov 2-type} of $X$. It is a 3-connected map $c_2\colon X\to P_2X$ with target a 3-coconnected CW-complex $P_2X$, i.e.\ $\pi_k(P_2X)$ for $k > 2$. The Postnikov 2-type of a CW-complex can be constructed as an inclusion, by attaching cells of dimension $>3$ to $X$ that kill all homotopy groups of dimension $>2$. Whenever we consider $P_2X$ in the following we also fix a 3-connected map $c_2$.
	In particular $H^i(X;\Z/2)$ and $H^i(P_2X;\Z/2)$ are isomorphic via $c_2^*$ for $i=0,1,2$, uniquely in the following sense.
	
	For any other choice $c_2\colon X \to (P_2X)'$ there exists a homotopy equivalence $\phi\colon P_2X\to (P_2X)'$ such that $\phi\circ c_2\simeq c_2'$ and such a $\phi$ is unique up to homotopy by \cite{Baues}*{Theorem~1.5.8}. 
	 In the following we will abuse notation and denote $(c_2^*)^{-1}(w_i(M))$ again by $w_i(M)\in H^i(P_2M;\Z/2)$.	
\end{rem}

\begin{thm}\label{thm:main-short} For $4$-manifolds $M,N$, we have $M\leq N$ if and only if there is a map $f_2\colon P_2M \to P_2N$ with $f_2^*(w_i(N))=w_i(M), i=1,2$. Moreover, given $f_2$, there exists an immersion $j:M_*\imra N$ such that the induced maps on fundamental groups satisfy $\pi_1(f_2)=\pi_1(j)$. 
\end{thm}
We give more equivalent conditions in \cref{thm:main}.
In the `only if' direction one can use $f_2:=P_2(j)$ because $P_2M_*=P_2M$. %Note that  $w_i(M)\in H^i(P_2M;\Z/2)$ are well defined for $i=1,2$ because $c_2$ induces isomorphisms on the first two (co)homology groups with arbitrary twisted coefficients.

We next study the existence of maps $f_2\colon P_2M \to P_2N$, regardless of their behaviour on the Stiefel-Whitney classes.  We use the fibration sequence $K(\pi_2M,2)\to P_2M \to B\pi_1M$ that is classified by the $k$-invariant $k_M\in H^3(\pi_1M;\pi_2M)$. The $k$-invariant is the obstruction for the existence of a section $B\pi_1M\to P_2M$.
%can be described as follows. It is given by the exact sequence
%$$
%0= H_2(\wt M^{(1)}) \to \pi_2M \cong H_2(\wt M) \to H_2(\wt M, \wt M^{(1)}) \to C_1(M) \to C_0(M)\to \mathbb Z \to 0
%$$ 
%which represents an element in $\Ext_{\Z[\pi_1M]}^3(\mathbb Z, \pi_2M)\cong H^3(\pi_1M; \pi_2M)$, see \cite[Theorem~III.5.3]{McL}.

\begin{thm}\label{thm:post}
There is a map $f_2\colon P_2M\to P_2N$ inducing $\varphi\colon \pi_1M \to \pi_1N$ on fundamental groups if and only if
		 $\varphi^*(k_N)\cap c_*[M] = 0 \in H_1(\pi_1M;\pi_2N_\varphi^{w_1(M)})$.
\end{thm}
Here $P_\phi$ is the $\Z[\pi_1M]$-module structure induced by $\phi$ on a $\Z[\pi_1N]$-module $P$, while for a $\Z[\pi_1M]$-module $Q$, the $\Z[\pi_1M]$-module $Q^{w_1(M)}$ is given by changing the sign of the group action by $w_1(M)$. One important consequence of the above result is that the existence of $f_2$ only depends on $(\pi_1M, c_*[M])$, not on all of $M$. In fact, it also just depends on $(\pi_1N, c_*[N])$ and we spell this out carefully in \cref{thm:shift}. Here we only mention the most striking special case where $c_*[N]=0$. Note that \cref{thm:post} gives a map $P_2N\to P_2N'$ for any $N'$ with the same normal 1-type as $N$ (but possibly non-trivial $c_*[N']$) and one can pre-compose it with a map $P_2M\to P_2N$.

\begin{thm}
	\label{thm:trivialclass}
	Fix a group $\pi'$, a $4$-manifold $M$ and a homomorphism $\phi\colon \pi_1M\to \pi'$. A free resolution $(C_*,d_*)$ of $\Z$ as trivial $\Z[\pi']$-module
	%$C_2\xrightarrow{d_2}C_1\xrightarrow{d_1}C_0\xrightarrow{\epsilon}\Z\to 0$,
gives a, in general non-free, $\Z[\pi_1M]$-resolution $((C_*)_\phi^{w_1(M)},d_*)$ of $\Z^{w_1(M)}$. %\xrightarrow{d_2}(C_1)_\phi^{w_1(M)}\xrightarrow{d_1}(C_0)_\phi^{w_1(M)}\xrightarrow{\epsilon}\Z^{w_1(M)}\to 0$ 
It induces a shift homomorphism
	\[
	H_4(\pi_1M;\Z^{w_1(M)})\xrightarrow{\shift}H_1(\pi_1M;(\ker d_2)_\phi^{w_1(M)}).
	\]
This shift sends $c_*[M]$ to zero if and only if for all $4$-manifolds $N$ with $\pi_1N=\pi'$ and $c_*[N] =0$
there exists a map $f_2:P_2M\to P_2N$ inducing $\phi$ on fundamental groups.
\end{thm}

If $c_*[N]\neq 0$ the vanishing of $\shift(c_*[M])$ is a sufficient but not necessary condition for the existence of $f_2$. We give a more sophisticated if and only if condition  in \cref{thm:shift}. In \cref{sec:general} we will also show how to control the pullbacks of Stiefel-Whitney classes under $f_2$ and we will derive the following corollaries.

\begin{cor}
	\label{cor:stable2}
	If $N$ is not almost spin, $H^1(\pi_1N;\Z[\pi_1N])=0$ and $\phi\colon \pi_1M\to \pi_1N$ is an isomorphism with $\varphi^*w_1^\pi( N)=w_1^\pi (M)$, then a $\phi$-immersion exists if and only if $\varphi_*c_*[M]$ is a multiple of $c_*[N]$.
\end{cor}
Note that the assumption $H^1(\pi;\Z[\pi])=0$ is satisfied if $\pi$ is finite or has one end.

 \begin{cor}
	\label{cor:multiple}
	Let $M, N$ be almost spin and consider a homomorphism $\varphi\colon \pi_1M\to \pi_1N$ satisfying $\varphi^*(w_i^\pi N)=w_i^\pi M$ for $i=1,2$ and inducing a monomorphism \newline $H_1(\pi_1M ; \pi_2N_\varphi^{w_1^\pi (M)}) \to H_1(\pi_1N;\pi_2N^{w_1^\pi (N)})$.
	\begin{enumerate}
		\item\label{it:cor:multiple1} A $\varphi$-immersion exists if $H_4(\pi_1N;\Z^{w_1^\pi (N)})=0$. 
		\item\label{it:cor:multiple2}  If $H^1(\pi_1N;\Z[\pi_1N])=0$, a $\varphi$-immersion exists if and only if $\varphi_*c_*[M]$ is a multiple of $c_*[N]$ in $H_4(\pi_1N;\Z^{w_1^\pi (N)})$. 
	\end{enumerate}
\end{cor}

In \cref{sec:general} we will prove a generalization of this result, \cref{cor:multiple2}, implying that a $\phi$-immersion $M_*\imra N$ does not force $\varphi_*c_*[M]$ to be a multiple of $c_*[N]$ if $H^1(\pi_1N;\Z[\pi_1N])\neq 0$.

We finish the introduction with the announced detailed formulation of Phillips' result. For manifolds $M, N$, we write $\Sub(M,N)$ and $\Imm(M,N)$ for the spaces of submersions and immersions, respectively. Furthermore, we denote by $\Sur(TM,TN)$ and $\Mono(TM,TN)$ the spaces of vector bundle homomorphism that are fibrewise surjective or injective, respectively. 
\begin{thm}[{\cite[Thm.~A]{phillips}}]
	\label{thm:phillips}
	If $M$ is open, the differential induces a weak homotopy equivalence $\Sub(M,N)\to \Sur(TM,TN)$.
\end{thm}
Since a codimension zero submersion is an immersion, we obtain the following corollary.
\begin{cor}
	\label{cor:phillips}
	If $M$ is open and $\dim N=\dim M$, the differential induces a weak homotopy equivalence $\Imm(M,N)\to \Mono(TM,TN)$.
\end{cor}
By the Smale--Hirsch theorem, the corollary is true for $\dim M<\dim N$ without the assumption that $M$ is open. The special case that an open, parallelizable $n$-manifold can be immersed in $\R^n$ is due to Poenaru \cite[Theorem 5]{poenaru} and Hirsch \cite[Theorem 4.7]{Hirsch}. 

\subsection*{Outline}
In \cref{sec:23}, we study the case of $2$ and $3$-dimensional manifolds. In \cref{sec:background}, we recall the necessary background on immersions. In particular, we give statements that are equivalent to the existence of an immersion in \cref{thm:main}, which includes the statement from \cref{thm:main-short}. We also prove \cref{prop:easy}. The remaining results from \cref{sec:general-intro} as well as \cref{cor:stable} are shown in \cref{sec:general}. In \cref{sec:james}, we give an overview over the existence of manifolds representing certain immersion types and prove \cref{prop:Z4}. Manifolds with cyclic fundamental groups are considered in \cref{sec:cyclic}, where we prove \cref{thm:cyclic-or,thm:cyclic-nor}.

\subsection*{Acknowledgements}
We thank the referee for their detailed and thoughtful report.

Much of this research was conducted at the Max Planck Institute for Mathematics.
DK was partially supported by the Deutsche Forschungsgemeinschaft (DFG, German Research Foundation) under Germany's Excellence Strategy - GZ 2047/1, Projekt-ID 390685813.

\section{The immersion order for lower dimensional manifolds}\label{sec:23}

In dimension 1 the circle $S^1$ is the only closed connected manifold.
\begin{prop} \label{prop:dim2}
There are two immersion equivalence classes of surfaces, $S^2 < \RP^2$.
\end{prop}

\begin{proof}
	To ``see'' this answer, remove a point from an orientable surface $M$. Then $M_*$ is the interior of a 2-disk with 1-handles trivially attached (such that the result is orientable) and one can draw an immersion into $\mathbb R^2$ as in \cref{fig:torus}. 

\def\ringa{(-1,0) circle (2) (-1,0) circle (3)}
\def\ringb{(1,0) circle (2) (1,0) circle (3)}
\begin{figure}[h] \label{fig:torus}
    \begin{tikzpicture}[scale=0.50]
        \draw \ringa;
        \begin{scope}[even odd rule]
            \draw[fill=white] \ringb;
        \end{scope}
        \begin{scope}[even odd rule]
            \clip (-1.25,-2.83) rectangle (1.25,-1.72);
            \fill[fill=white] \ringa;
        \end{scope}
    \end{tikzpicture}
    \caption{An immersion of $T^2_*$ into the plane.}
\end{figure}

In particular, all orientable surfaces are immersion equivalent. If one removes a point from a non-orientable surface the resulting manifold is the interior of a 2-disk with trivial 1-handles attached plus a single non-trivial 1-handle. One can draw an immersion of this handlebody into a M\"obius strip and so one obtains an immersion into $\mathbb{RP}^2$. Conversely, $\RP^2_*$ is a M\"obius strip and thus embeds into every non-orientable surface. Thus all non-orientable surfaces are immersion equivalent. This implies \cref{prop:dim2} because if a codimension zero manifold immerses into an orientable manifold, it is itself orientable. 
\end{proof}

The situation in dimension 3 is more complex. We first show that the tangent bundle is determined by $w_1$ and $w_2$. This is well-known but we give a proof sketch for convenience of the reader.
\begin{lem}
	\label{lem:w1w2-det}
	Let $X$ be a CW-complex of dimension at most $3$. Then every vector bundle of dimension at least $3$ over $X$ is determined by its first two Stiefel--Whitney classes. 
\end{lem}
\begin{proof}
	Since $SO(3)\simeq \RP^3$, we have that 
	\[\pi_3(BO(3))\cong \pi_3(BSO(3))\cong \pi_2(SO(3))=0.\]
	The map $BO(n)\to BO(n+1)$ is $n$-connected and hence $\pi_3(BO(n))=0$ for $n\geq 3$. Thus the map $BO(n)\xrightarrow{w_1\times w_2}K(\Z/2,1)\times K(\Z/2,2)$ is $4$-connected for $n\geq 3$. It follows that for a CW-complex of dimension at most $3$ a map to $BO(n)$ with $n\geq 3$ is determined up to homotopy by $w_1$ and $w_2$ as claimed.
\end{proof}
Since the Wu class $v_k(TM)$ represents $\Sq^k$, the second Wu class $v_2(TM)$ vanishes for every $3$-manifold $M$ by degree reasons. Since $\Sq(v(TM))=w(TM)$ and $w(TM)w(\nu M)=1$, we have $v_1(TM)=w_1(TM)=w_1(M)$ and $w_2(M)=w_2(TM)+w_1(TM)^2=v_2(TM)=0$. Hence orientable 3-manifolds have trivial normal (and tangent) bundle by \cref{lem:w1w2-det}. Recall that $w_1(M)$ comes from a unique class $w_1^\pi (M)\in H^1(\pi_1M;\mathbb Z/2)$.

\begin{prop} \label{prop:dim3} In dimension 3, $M\leq N$ if and only if there is a group homomorphism $\phi\colon \pi_1M \to \pi_1N$ with $\phi^*(w_1^\pi (N)) = w_1^\pi (M)$. In particular,
\begin{enumerate}
\item $S^3$ is the minimum and $\mathbb {RP}^2\times S^1$ is the maximum for the immersion order,
\item $S^3 \bowtie N$ if and only if $N$ is orientable,
\item $\mathbb {RP}^2\times S^1\bowtie N$ if and only if  there exists $x \in \pi_1N$ with $x^2=1$ and $w_1^\pi (N)(x)\neq 0$, 
\item $S^1\wt\times S^2\leq N$ if and only if  $N$ is non-orientable, and 
\item $M\leq S^1\wt\times S^2$ if and only if  $w_1^\pi (M)$ admits an integral lift.
\end{enumerate}
\end{prop}

\begin{proof}
	We first show (2), i.e.\ that all orientable manifolds are immersion equivalent.	For orientable $M,N$ the tangent bundles are trivial. Thus the constant map $M\to N$ pulls back the tangent bundle of $N$ to the tangent bundle of $M$ and by \cref{cor:phillips} its restriction to $M_*$ is homotopic to an immersion. 

For non-orientable 3-manifolds, $w_1(M)$ comes into play.  If $j\colon M_*\imra N$ then the induced map $\phi:=j_*\colon \pi_1M \to \pi_1N$  has the property: $\phi^*(w_1^\pi (N)) = w_1^\pi (M)$. If in turn we have a map $\phi\colon \pi_1M \to \pi_1N$ with this property then there is a map $M_* \to B\pi_1N$ pulling $w_1^\pi (N )$ back to $w_1(M_*)$. The obstructions for lifting this map to $N$ lie in $H^{k+1}(M_*;\pi_k(N))$ for $k\geq 2$. All these groups are trivial since $M_*$ has the homotopy type of a 2-complex. So there is a map $M_*\to N$ pulling back $w_1$ (and $w_2$ since $w_2(M)=w_2(N)=0$). But 3-dimensional vector bundles over $M_*$  are classified by $w_1$ and $w_2$ by \cref{lem:w1w2-det}. So by \cref{thm:phillips} the map is homotopic to an immersion. This shows the main statement of the proposition.

(1) and (3)-(5) follow immediately. For example, if $M$ is a non-orientable  3-manifold then $\phi\colon \pi_1(M)\xrightarrow{w_1}\Z/2\hra \Z/2\times \Z\cong \pi_1(\RP^2\times S^1)$ satisfies $\phi^*w_1^\pi(\RP^2\times S^1)=w_1^\pi(M)$ and hence $M\leq \RP^2\times S^1$. This shows that $\mathbb {RP}^2 \times S^1$ is a maximum for 3-manifolds. The details of (3)-(5) are left to the reader.
\end{proof}

\section{Background on immersions}
\label{sec:background}

From now on we study only $4$-manifolds. 
Let $(C_*M, d_*^M)$ be the chain complex of the universal cover of $M$ induced by a handle decomposition of $M$. The $\Z[\pi_1M]$-modules $C_kM$ are free on one generator for each $k$-handle and the boundary maps $d_k^M\colon C_kM\to C_{k-1}M$ are an algebraic image of the attaching maps of the $k$-handles to (the boundary of) the $(k-1)$-skeleton. We can assume that our ccs manifolds have only one $0$-handle.
%
%\begin{rem}\label{rem:SW}
%	One can think of $w_1(M)$ as a 1-cocycle $C_1M\to \Z/2$. It measures whether a chosen orientation on the $0$-handle extends to a given $1$-handle.
%	The central extension
%	\[
%	1 \ra \{\pm 1\} \ra Pin^+(n) \ra O(n) \ra 1
%	\]
%	is classified by the universal Stiefel-Whitney class $w_2\in H^2(BO(n);\Z/2)$.
%	Therefore, $w_2(\xi)\in H^2(X;\Z/2)$ is the obstruction to existence of a pin$^+$-structure on the vector bundle $\xi$ over $X$. Equivalently a $\text {pin}^+$-structure on $\xi$ is a Spin-structure on $\xi \oplus 3 \det \xi$, where this vector bundle is equipped with the tautological orientation. For a manifold a $\text {pin}^+$-structure is such a structure on the normal bundle. So such a structure exists on the 1-skeleton $M^{(1)}$, consisting of all $0$- and $1$-handles and $w_2(M)$ is represented by the 2-cocycle $\wh w_2(M)\colon C_2M \to \Z/2$ measuring whether this pin$^+$-structure extends to a given $2$-handle. A second pin$^+$-structure on $M^{(1)}$ differs by a class $x$ in $H^1(M^{(1)};\Z/2) = \Hom_{\pi_1M}(C_1M,\Z/2)$ that changes the 2-cocycle $\wh w_2(M)$ by the coboundary of $x$. 
%\end{rem}

\begin{rem}\label{rem:wi}
Since $v_k(TM)$ represents $\Sq^k$, we have $v_3(TM)=0$ by degree reasons, and
\[w_1(TM)w_2(TM)=v_1(TM)w_2(TM)=\Sq^1(w_2(TM)).\]
By the Wu formula, $\Sq^1(w_2(TM))=w_3(TM)+w_1(TM)w_2(TM)$. Hence $w_3(TM)=0$.
%
%From $w(TM)w(\nu M)=1$ we see that \[w_1(M)=w_1(TM), w_2(M)=w_2(TM)+w_1(M)^2, \quad\text{and}\quad w_3(M)=w_3(TM)+w_1(TM)^3.\] By the Wu formula $\Sq^1(w_2(TM))=w_3(TM)+w_1(TM)w_2(TM)=w_3(M)+w_1(M)w_2(M)$. Hence
%\[w_3(M)=\Sq^1(w_2(TM))+w_1(M)w_2(M)=\Sq^1(w_2(M))+w_1(M)w_2(M).\]
%Using that the Wu class $v_3(TM)=0$, and $w(TM)=\Sq(v(TM))$, one can also see that $w_1(TM)w_2(TM)=0$ and thus $w_1(M)w_2(M)=w_1(M)^3$.
Since $w(TM)w(\nu M)=1$ and $H^i(M_*;\Z/2)=0$ for $i>3$ for every $4$-manifold $M$, it follows that only the Stiefel--Whitney classes $w_1,w_2$ are relevant. %for vector bundles over 3-dimensional complexes like $M_*$ of a 4-manifold. 
\end{rem}

\begin{thm}\label{thm:main} 
The existence of the following types of data are equivalent:
	\begin{enumerate}
		\item\label{it:main-i} An immersion $j\colon M_*\imra N$.
		\item\label{it:main-ii}  A map $f\colon M_*\to N$ together with an isomorphism of vector bundles $TM_* \cong f^*TN$ (not necessarily given by the differential $Tf$). 
		\item\label{it:main-iii}  A map $f\colon M_*\to N$ with $f^*(w_i(N))=w_i(M)$ for $i=1,2$.
		\item\label{it:main-iv}  A map $f_2\colon P_2M \to P_2N$ with $f_2^*(w_i(N))=w_i(M)$ for  $i=1,2$. 
		\item\label{it:main-v}  A  homomorphism $\varphi\colon \pi_1M \to \pi_1N$ with $\varphi^*(w_1^\pi (N))=w_1^\pi (M)$, together with a $\phi$-linear map $h\colon C_2M/\im(d_3^M) \to C_2N/\im(d_3^N)$ and a choice of $C_1(\phi)\colon C_1M\to C_1N$ such that $d_2^N \circ h= C_1(\varphi)\circ d_2^M$ and $h^*(\wh w_2(N))$ is a 2-cocycle representative of $w_2(M)$.
	\end{enumerate}
	One can arrange that the maps induced on $\pi_1$ by $j, f, f_2$ are the same and agree with $\phi$. 
	\end{thm}
	Here a \emph{$\phi$-linear map} $P\to Q$, for $P$ a $\Z[\pi_1M]$-module and $Q$ a $\Z[\pi_1N ]$-module, is a $\Z[\pi_1M]$-linear map $P\to Q_\phi$. The choices for the $\phi$-linear map $C_1(\phi):C_1M \to C_1N$ come about as follows: $C_1M$ is determined by a choice of generators $g_i$ of $\pi_1M$ and similarly for $C_1N$. Then $C_1(\phi)$ is determined by expressing the group elements $\phi(g_i)$ as words in the chosen generators of $\pi_1N$. 

\begin{proof}[Proof of \cref{thm:main}]
	(1) and (2) are equivalent by \cref{cor:phillips}. 
	
	Since $M_*$ is homotopy equivalent to a 3-dimensional complex, by \cref{lem:w1w2-det} two $4$-dimensional vector bundles over $M_*$ are isomorphic if and only if $w_1, w_2$ agree. Using that $w_1(TM)=w_1(M)$ and $w_2(TM)=w_2(M)+w_1(M)^2$, this shows that (2) and (3) are equivalent. 
	
	We can build the Postnikov $2$-types $P_2M$ and $P_2N$ by attaching cells of dimension $4$ and higher to $M_*$ and $N$, respectively. Thus by cellular approximation, every map $f_2\colon P_2M\to P_2N$ is homotopic to a map that sends the $3$-skeleton $M_*$ to $N$. This shows that (4) implies (3). The other direction follows from the functoriality of $P_2$ and similar functorialities show that (5) is implied by any of the previous statements.
	
	To show that (5) implies (4), we first observe that $h$ can be extended to a $\phi$-linear chain map $h_n\colon C_n(M_*) \to C_nN$ as follows.  
	For $n=0,1$ this follows from our assumption on the existence of $C_1(\phi)$. The map $C_2(M_*)\to C_2(M_*)/\im(d_3^M)\xrightarrow{h}C_2(N)/\im(d_3^N)$ admits a lift $h_2$ along $C_2(N)\to C_2(N)/\im(d_3^N)$ since $C_2(M_*)$ is free. Since $C_2(M_*)\xrightarrow{h_2}C_2(N)\to C_2(N)/\im(d_3^N)$ factors through $C_2(M_*)/\im(d_3^M)$ and $C_3(M_*)$ is free, $h\circ d_3^M$ admits a lift $h_3$ along $d_3^N$.
	
	Composing with $c_2\colon N\to P_2N$ and observing that for $n\leq 3$ the chain complexes of $M_* \subset M \subset P_2M$ agree (because the cells do), we need to realize a chain map $C_n(M_*)\to C_n(P_2N)$ by a continuous map $f\colon M_*\to P_2N$ that is already given on the 1-skeleton $M^{(1)}$ by a choice of $(B\phi)^{(1)}$. Choose an extension $(B\phi)^{(2)}$ on the 2-cells and compare the resulting map $C_2(\phi)\colon C_2M\to C_2N$ with $h_2$.
		Since both commute with $C_1(\phi)$, their difference is a map
		\[
		h_2 - C_2(\phi) : C_2M \to \ker d_2^N \cong \pi_2N^{(2)}.
		\]
		We can change our choice of $(B\phi)^{(2)}: M^{(2)}\to N^{(2)}$ on the 2-cells of $M$ as follows: Pinch off a 2-sphere from a 2-cell $e\subset M$ and map it to $h_2(e) - C_2(\phi)(e)$. This new version of $C_2(\phi)$ agrees with $h_2$ by construction. 
	 The obstruction for extending $(B\phi)^{(2)}$ to a 3-cell $e\colon D^3\to M_*$ is given by $[h_2(d_3^M(e))] \in \pi_2N \cong H_2(C_*(P_2N))$. Since $h_2(d_3^M(e))=d_3^N(h_3(e))$, this element vanishes. Hence there is a map $f\colon M_*\to P_2N$ as claimed.
	 % because $h_2$ extends to $h_3$ on 3-chains and hence any boundary in $M_*$ becomes a boundary in $P_2N$.
	 
	 By construction of $f\colon M_*\to P_2N$, we can read off its induced map $f^*w_2(N)$ by using the chain map $h_n$ as claimed.
\end{proof}

The following is a useful observation concerning immersions into $N$ compared to $N_*$. It implies that the immersion order is indeed transitive.

\begin{lem}
	\label{lem:open}
	If there is an immersion $M_*\imra N$, then there also is an immersion $M_*\imra N_*$.
\end{lem}
\begin{proof}
	An immersion $f\colon M_*\immto N$ induces an isomorphism $TM_*\cong f^*TN$. Since $M_*$ is homotopy equivalent to a 3-dimensional complex, $f$ is homotopic to a map $g$ which misses one point. Then
	\[TM_* \cong f^*TN \cong g^*TN\cong g^*TN_*,\]
	and hence, by \cref{cor:phillips}, there is an immersion $M_*\imra N_*$.
\end{proof}

\begin{proof}[Proof of \cref{prop:easy}]
	We will use the equivalence of (1) and (3) in \cref{thm:main}.	
	
	\eqref{it:easy1} Since $S^4_*\cong \R^4$, $S^4\leq M$ holds for all $4$-manifolds $M$. Since $w_iS^4=0$, a map $f\colon M_*\to S^4$ with $f^*w_iS^4=w_i(M)$ for $i=1,2$ exists if and only if $M$ is spin.
	
	\eqref{it:easy2} Since $\CP^2_*\simeq S^2$ and $w_2(\CP^2)\neq 0$, we have $\CP^2\leq M$ if and only if there exists an element of $\pi_2(M)$ on which $w_2$ is non-trivial. This is the case if and only if $M$ is not almost spin.
	
	\eqref{it:easy3}  Since $\CP^2$ is orientable, $M\leq \CP^2$ implies that $M$ is orientable. For the converse, note that every orientable $4$-manifold admits an integral lift $x\in H^2(M;\Z)$ of $w_2(M)$ by unpublished work of Vogt and the fourth author, see \cite[Remark~5.7.5]{GS} where their proof is given. The class $x$ induces a map $f\colon M\to \CP^2$. By definition of $x$ we have $f^*w_2(\CP^2)=w_2(M)$ and thus $M\leq \CP^2$.
	
	\eqref{it:easy4} Since $w_1(S^1\wt\times S^3)\neq 0$, $S^1\wt\times S^3\leq M$ implies that $w_1(M)\neq 0$. Now let $w_1(M)\neq 0$. Then there is a map $g\colon S^1\to M$ such that $g^*w_1\neq 0$. Define $f\colon S^1\wt\times S^3\to M$ as the composition of the projection to $S^1$ with $g$. Then $f^*w_1(M)=w_1(S^1\wt\times S^3)$ and $f^*w_2(M)=0=w_2(S^1\wt\times S^3)$. Hence there exists an immersion $(S^1\wt\times S^3)_*\imra M$.
	
	\eqref{it:easy5} 
	If $M\leq S^1\wt\times S^3$, then there exists a map $f\colon M_*\to S^1\wt\times S^3$ with $f^*(w_i(S^1\wt\times S^3))=w_i(M)$ by \cref{thm:main}. Since $\pi_1(S^1\wt\times S^3)\cong \Z$ and $w_2(S^1\wt\times S^3)=0$ it follows that $w_1(M)$ admits an integral lift and $w_2(M)=0$. 
	
	Now assume that $w_2(M)=0$ and that $w_1(M)$ admits an integral lift represented by $h\colon M\to S^1$ and $w_2(M)=0$. Let $g\colon S^1\to S^1\wt\times S^3$ be a map that induces an isomorphism on fundamental groups. Let $f:=g\circ h$. Then 
	\[f^*w_1(S^1\wt\times S^3)=h^*g^*w_1(S^1\wt\times S^3)=w_1(M),\]
	and
	\[f^*w_2(S^1\wt\times S^3)=f^*0=0=w_2(M).\] 
	By \cref{thm:main} there exists an immersion $M_*\imra S^1\wt\times S^3$ and hence $M\leq S^1\wt\times S^3$ as claimed.
\end{proof}

\section{Results for general fundamental group}\label{sec:general}
 In this section we prove the results from \cref{sec:general-intro} (and \cref{cor:stable}) as well as additional facts relevant to our computations. We start with a result on the existence of maps between Postnikov towers, regardless of their behaviour on Stiefel-Whitney classes. The equivalence of (1) and (3) is \cref{thm:post}, the equivalence of (1), (2) and (4) actually holds for all spaces, not just 4-manifolds. 
\begin{thm}\label{thm:post-general}
	For a homomorphism $\varphi\colon \pi_1M \to \pi_1N$ the following are equivalent:
	\begin{enumerate}
		\item\label{it:post-general1} There is a continuous map $f_2\colon P_2M\to P_2N$ inducing $\varphi$ on fundamental groups. 
		\item\label{it:post-general2} There is a $\Z[\pi_1M]$-linear map $\varphi_2\colon \pi_2M \to \pi_2N_\varphi$ such that $\varphi^*(k_N) = (\varphi_2)_*(k_M)$.
		\item\label{it:post-general31} $c_*[M] \cap \varphi^*(k_N) = 0 \in H_1(\pi_1M;\pi_2N_\varphi^{w_1})$.
		\item\label{it:post-general32} $c^*\varphi^*(k_N) = 0 \in H^3(M;\pi_2N_\varphi)$.	
		\item\label{it:post-general4} There is an $h\colon C_2M/\im(d_3^M) \to C_2N_\varphi/\im(d_3^N)$ such that $d_2^N \circ h= C_1(\varphi)\circ d_2^M$.	
	\end{enumerate}
\end{thm}
\begin{proof}
	Let $q\colon P_2M\to B\pi_1M$ induce the identity on fundamental groups. 
	The obstruction for a lift $f_2\colon P_2M\to P_2N$ of $q$ inducing $\varphi$ on fundamental groups is $q^*\phi^*k_N$.
	
	\eqref{it:post-general1}$\Rightarrow$\eqref{it:post-general2}: The fibration $K(\pi_2M,2)\to P_2M\to B\pi_1M$ induces an exact sequence \[H^2(K(\pi_2M,2);\pi_2N)\to H^3(B\pi_1M;\pi_2N)\xrightarrow{q^*} H^3(P_2M;\pi_2N)\]
	as can be seen from the Serre spectral sequence. Thus if $q^*\phi^*k_N=0$, there exists an element $\phi_2\in H^2(K(\pi_2M,2);\pi_2N)$, which we can interpret as a $\Z[\pi_1M]$-linear map $\varphi_2\colon \pi_2M \to \pi_2N_\varphi$. It now follows from the following commutative diagram that $\phi^*(k_N)=(\phi_2)_*(k_M)$.
	
	\[\begin{tikzcd}
		\id\in H^2(K(\pi_2M,2);\pi_2M)\ar[r,"(\phi_2)_*"]\ar[d]
		&H^2(K(\pi_2M,2);\pi_2N)\ni \phi_2\ar[d]\\
		k_M\in H^3(B\pi_1M;\pi_2M)\ar[r,"(\phi_2)_*"]&H^3(B\pi_1N ;\pi_2N)\ni \phi^*(k_N)
	\end{tikzcd}\]
	
	\eqref{it:post-general2}$\Rightarrow$\eqref{it:post-general1}:  A $\Z[\pi_1M]$-linear map $\varphi_2\colon \pi_2M \to \pi_2N_\varphi$ such that $\varphi^*(k_N) = (\varphi_2)_*(k_M)$ gives
	\[q^*\phi^*(k_N)=q^*((\phi_2)_*(k_M))=(\phi_2)_*(q^*(k_M))=(\phi_2)_*(0)=0.\]
	Hence there exists a map $f_2$ as claimed.
	
	\eqref{it:post-general1}$\Leftrightarrow$\eqref{it:post-general32}: Since $P_2N$ is 3-coconnected, a map $f_2\colon P_2M\to P_2N$ inducing $\phi$ on fundamental groups exists if and only if there is a map $f\colon M\to P_2N$ inducing $\phi$ on fundamental groups. Now let $c\colon M\to B\pi_1M$ be the identity on fundamental groups.	
	The map $\phi$ induces a map $\phi\circ c\colon M\to B\pi_1N$ and the obstruction for lifting this map to a map $f\colon M\to P_2N$ is $c^*\phi^*(k_N)$. This shows that \eqref{it:post-general1} and \eqref{it:post-general32} are equivalent. 
	
	\eqref{it:post-general31}$\Leftrightarrow$\eqref{it:post-general32} holds by Poincar\'e duality.

\eqref{it:post-general1}$\Rightarrow$\eqref{it:post-general4}: Choose models for $P_2M$ and $P_2N$ that agree with $M$ and $N$ on the 3-skeletons, respectively. Then we choose $h$ to be the chain map induced by $f$.

\eqref{it:post-general4}$\Rightarrow$(\eqref{it:post-general1}: This is the same as \eqref{it:main-v}$\Rightarrow$\eqref{it:main-iv} of \cref{thm:main} without the condition on Stiefel--Whitney classes and the same proof holds.
\end{proof}

\begin{prop}\label{prop:stable1new}
Let $f_2\colon P_2M\to P_2N$ induce $\phi$ on fundamental groups and assume that $\phi^*w_1^\pi (N)=w_1^\pi (M)$. Then the following are equivalent:
	\begin{enumerate}
		\item\label{it:stable1} There exists an immersion $M_*\imra  N$ inducing $\phi$ on fundamental groups.
		\item\label{it:stable2} $w_2(M)+f_2^*w_2 N$ lifts to $H^2(M;\pi_2N_\phi)$ along the $\Z[\pi_1N]$-linear map $w_2\wt N\colon \pi_2N\to \Z/2$.
		\item\label{it:stable3} For every $f'_2\colon P_2M\to P_2N$ inducing $\phi$ on fundamental groups, the sum $w_2(M)+(f'_2)^*w_2 N$ lifts to $H^2(M;\pi_2N_\phi)$ along $w_2\wt N\colon \pi_2N\to \Z/2$.
	\end{enumerate}
\end{prop}
\begin{proof}
	The statement follows from obstruction theory which implies that $H^2(M;\pi_2N_\phi)$ acts simply transitively on the set $[P_2M,P_2N]_\phi$ of homotopy classes of maps inducing $\phi$ on fundamental groups. We will give an explicit chain level version of this  action to show that it indeed changes the pullback of $w_2(N)$ as claimed. For this, choose models for $P_2M$ and $P_2N$ that agree with $M$ and $N$ on the $3$-skeleton, respectively. 
	
	\eqref{it:stable1}$\Rightarrow$\eqref{it:stable3}: If there exists an immersion $M_*\imra  N$, then the induced map $f_2\colon P_2M\to P_2N$ satisfies $f_2^*w_2(N)=w_1(M)$. It remains to show that for any other map $f_2'\colon P_2M\to P_2N$ inducing $\phi$ on fundamental groups, $f_2^*w_2(N)+(f_2')^*w_2(N)$ admits a lift to $H^2(M;\pi_2N_\phi)$ along $w_2\wt N\colon \pi_2N\to \Z/2$. For $+\leq 3$, let $h_*,h_*'\colon C_*M\to C_*N$ be the maps of chain complexes induced by $f_2$ and $f_2'$, respectively. Let $w\colon C_2N\to \Z/2$ be a cocycle representing $w_2(N)$. Since $w$ is a cocycle, it induces a map $C_2N/\im d_3^N\to \Z/2$. Restricting this map to $\ker d_2^N/\im d_3^N=\pi_2N$ yields $w_2\wt N$. Since $f_2$ and $f_2'$ induce the same map on fundamental group, we can assume that $h_1=h_1'$. Thus $h_2-h_2'$ can be viewed as a map $C_2M\to \ker d_2^N$. It follows that $f_2^*w_2(N)+(f_2')^*w_2(N)$, which is represented by $w\circ (h_2-h_2')$, admits a lift to $\pi_2N$ as claimed.
	
	\eqref{it:stable3}$\Rightarrow$\eqref{it:stable2} is obvious.
	
	\eqref{it:stable2}$\Rightarrow$\eqref{it:stable1}: For $*\leq3$, let $h_*\colon C_*M\to C_*N$ be the map of chain complexes induced by $f_2$. Let $h'\colon \coker d_3^M\to\pi_2N$ be a $\phi$-linear map representing the given lift.
	We choose lifts $h_2'\colon C_2M\to \ker d_2^N$ and $h_3'\colon C_3M\to C_3N$ given by the following diagram as explained below.
	\[\begin{tikzcd}
		C_3M\ar[r,dashed,"h_3'"]\ar[d,"d_3^M"]&C_3N\ar[d,"d_3^N"]\\
		C_2M\ar[r,dashed,"h_2'"]\ar[d,two heads]&\ker d_2^N\ar[d,two heads]\\
		\coker d_3^M\ar[r,"h'"]&\ker d_2^N/\im d_3^N
	\end{tikzcd}\]
	Here $h_2'$ is any lift along the map $\ker d_2^N\to \ker d_2^N/\im d_3^N$. Since the left vertical composition is trivial, $h_2'\circ d_3^M$ has image in $\im d_3^N$ and hence there exists a lift $h_3'\colon C_3M\to C_3N$.
	
	 Since $h_2'$ has image in the kernel of $d_2^N$, the following diagram commutes.
	\[\begin{tikzcd}
		C_3M\ar[r,"h_3+h_3'"]\ar[d,"d_3^M"]&C_3N\ar[d,"d_3^N"]\\
		C_2M\ar[r,"h_2+h_2'"]\ar[d,"d_2^M"]&C_2N\ar[d,"d_2^N"]\\
		C_1M\ar[r,"h_1"]&C_1N
	\end{tikzcd}\]
	The bottom part can be realized as a map of 2-complexes by starting with $f_2$ and then changing the map on the 2-cells corresponding to $h_2'$. Commutativity of the diagram and the fact that $\pi_2(N)$ is isomorphic to $H_2(N;\Z[\pi] )$ implies that it can be extended to a map $M_*\to N_*$. Any such map can be extended a map $P_2M\to P_2N$ since $P_2N$ is 3-coconnected. Call the obtained map $f_2'$. As in the first part of the proof, $(f_2')^*w_2(N)$ differs from $f_2^*w_2(N)$ by $w\circ h_2'$, where $w\colon C_2N\to \Z/2$ is a cocycle representing $w_2(N)$.
	By construction of $h_2'$, the map $w\circ h_2'$ agrees with
	\[C_2^M\to \coker d_3^M\xrightarrow{h'}\pi_2(N)\xrightarrow{w_2(\wt N)}\Z/2.\]
	By assumption on $h'$, this represents the class $w_2(M)+f_2^*w_2 N$.
	Thus $(f_2')^*w_2(N)=f_2^*w_2(N)+w_2(M)+f_2^*w_2 N=w_2(M)$. By \cref{thm:main} \eqref{it:main-iv}$\Rightarrow$\eqref{it:main-i}, there is an immersion $M_*\to N$ inducing $\phi$ on fundamental groups.
\end{proof}

\begin{cor}\label{prop:stable1}
	Let $M, N$ be almost spin and fix a homomorphism $\phi\colon \pi_1M\to \pi_1N$. Then there exists an immersion $M_*\imra  N\#\CP^2$ inducing $\phi$ on fundamental groups if and only if there exists a map $f_2\colon P_2M\to P_2N$ inducing $\phi$ on fundamental groups such that $\phi^*w_1^\pi (N)=w_1^\pi (M)$ and $w_2 M+c^*\phi^*w_2^\pi (N)$ admits a lift to $H^2(M;\Z[\pi_1 N]_\phi)$.
\end{cor}
\begin{proof}
	We have maps $p\colon P_2(N\#\CP^2)\to P_2N$ and $i\colon P_2N\to P_2(N\vee \CP^2)\simeq P_2(N\#\CP^2)$ inducing the identity on fundamental groups such that $i^*w_2(N\#\CP^2)=w_2(N)$. % and $p^*w_2(N)+w_2(N\#\CP^2)$ admits a lift to $H^2(N\#\CP^2;\Z[\pi_1 N])$. 
	An immersion $M_*\imra N\#\CP^2$ induces a map $P_2(M)\to P_2(N\#\CP^2)$ inducing $\phi$ on fundamental groups. Composing this with $p$, we obtain a map $f_2\colon P_2(M)\to P_2(N)$ inducing $\phi$ on fundamental groups.
	
	Since $N$ is almost spin, we have
	\[w_2(\wt{N\#\CP^2})\colon \pi_2(N\#\CP^2)\cong \pi_2N\oplus\Z[\pi_1N]\xrightarrow{\proj_2}\Z[\pi_1N]\xrightarrow{\epsilon} \Z/2,\]
	where $\proj_2$ is the projection onto the second summand and $\epsilon$ is the mod 2 augmentation map.
	A lift of $x\in H^2(M;\Z/2)$ to $H^2(M;\pi_2(N\#\CP^2)_\phi)$ along $w_2(\wt{N\#\CP^2})$ gives a lift to $H^2(M;\Z[\pi_1 N]_\phi)$ by applying $\proj_2$. Conversely, a lift to $H^2(M;\Z[\pi_1 N]_\phi)$ induces a lift along $w_2(\wt{N\#\CP^2})$ by composing with the map
	\[\Z[\pi_1 N]_\phi\xrightarrow{(0,\id)}\pi_2N_\phi\oplus\Z[\pi_1N]_\phi\cong \pi_2(N\#\CP^2)_\phi.\]
	
	Note that $f_2^*w_2(N\#\CP^2)=w_2(N)$. Let $c_N\colon P_2(N)\to B\pi_1(N)$ be a map that induces the identity on fundamental groups. Then $c_N^*w_2^\pi(N)=w_2(N)$ by definition of $w_2^\pi(N)$ and $c_n\circ f_2\simeq \phi\circ c$. Hence
	\[(i\circ f_2)^*w_2(N\#\CP^2)=f_2^*w_2(N)=(c_n\circ f_2)^*w_2^\pi(N)=c^*\phi^*w_2^\pi(N).\]
	
	The proof so far shows that $w_2(M)+(i\circ f_2)^*w_2(N\#\CP^2)$ admits a lift along $w_2(\wt{N\#\CP^2})$ if and only if $w_2 M+c^*\phi^*w_2^\pi (N)$ admits a lift to $H^2(M;\Z[\pi_1 N]_\phi)$.
	
	Hence applying the equivalence (1)$\Leftrightarrow$(3) of \cref{prop:stable1new} to $i\circ f_2$ shows the statement.
\end{proof}

The following result uses in spirit an argument going back to at least C.T.C.\ Wall \cite{wall}*{Theorem~1,~Corollary~1}, who showed that $N\# S^2 \times S^2 \cong N\# S^2 \wt\times S^2$  if and only if $N$ is not almost spin by utilizing an immersed sphere in $N$ with odd normal Euler number to guide an ambient isotopy from one surgery on a trivial circle in $N$ to the other.  

\begin{lem}\label{lem:stable}
	If $N$ is not almost spin, then the following are equivalent 
	\begin{enumerate}
		\item $M\le N$,
		\item $M \#\CP^2 \le N$, and
		\item $M\leq N\#\CP^2$.
	\end{enumerate}
	Moreover, this holds with control of the fundamental group in the sense that for a given homomorphism $\phi:\pi_1M\to \pi_1N$, if one item comes from a $\phi$-immersion then so do the other two. 
\end{lem}

\begin{proof}
	Since $M\leq M\#\CP^2$ and $N\leq N\#\CP^2$, inducing the identity on fundamental groups, the implications (1) $\Rightarrow$ (3) and (2) $\Rightarrow$ (1) are immediate. It remains to show that (3) $\Rightarrow$ (2).
	
	An immersion $j:M_*\imra N\#\CP^2$, together with the inclusion $\iota\colon \CP^2_*\to N\#\CP^2$ induce
	\[g\colon P_2(M\#\CP^2)\simeq P_2(M_*\vee \CP^2_*) \ra P_2(N\#\CP^2).\]
	Under the isomorphism $H^i(M\#\CP^2;\Z/2)\cong H^i(M;\Z/2)\oplus H^i(\CP^2;\Z/2)$ for $i=1,2$, we have $g^*(w_i(N\#\CP^2))=(j^*(w_i(N\#\CP^2)),\iota^*(w_i(N\#\CP^2)))=(w_i(M),w_i\CP^2)$. By \cref{thm:main} \eqref{it:main-iv}$\Rightarrow$\eqref{it:main-i} this implies $M\#\CP^2\leq N\#\CP^2$ given by an immersion that induces $\pi_1(j)$ on fundamental groups. 
	
	It remains to show $N\#\CP^2\leq N$ inducing the identity on fundamental groups. Since $N$ is not almost spin, there exists a map $f\colon S^2\to N$ such that $f^*w_2(N)\neq 0$. Since $\CP^2_*\simeq S^2$, we can consider the map
	\[P_2(N\#\CP^2)\simeq P_2(N\vee \CP^2_*)\xrightarrow{P_2(\id\vee f)}P_2(N).\]
	As above, this map pulls back the Stiefel--Whitney classes correctly and thus leads to an immersion $(N\#\CP^2)_*\imra N$ as claimed.
\end{proof}

\begin{proof}[Proof of \cref{cor:stable}]
	We first consider the case where $M$ and $N$ are almost spin. Assume that $\phi_*	c_*[M]=c_*[N]$. 
	By \cref{thm:post}, we have $k_N\cap c_*[N]=0$ since $N$ always admits a map to its Postnikov 2-type that is the identity on fundamental groups. Since $\phi_*c_*[M]=c_*[N]$, we get $k_N\cap \phi_*c_*[M]=0$. Since $\phi$ is an isomorphism, this implies that $\phi^*(k_N)\cap c_*[M]=0\in H_1(\pi_1M;\pi_2N^{w_1}_\phi)$ by naturality of the cap product. By \cref{thm:post}, this implies that there exists 
	a map $f_2\colon P_2M\to P_2N$ inducing $\phi$ on $\pi_1$. By assumption on $\phi$ we can apply \cref{thm:main} \eqref{it:main-iv}$\Rightarrow$\eqref{it:main-i} to see that there is an immersion $M_*\to N$ inducing $\phi$ on fundamental group.
		
	Now let $M$ and $N$ be not almost spin. Using Kreck's modified surgery \cite{surgeryandduality}, see also \cite[Theorem~1.2]{KPT}, the manifolds $M$ and $N$ are $\CP^2$-stably isomorphic since $\phi_*c_*[M]=c_*[N]$. By \cref{lem:stable} this implies that $M\bowtie N$ with control of the fundamental groups.
\end{proof}
\begin{lem}
	\label{lem:imageofP}
	Let $P$ be a 3-coconnected space with fundamental group $\pi$, let $w\colon \pi\to \Z/2$ be a homomorphism, let $x\in H_4(\pi;\Z^w)$ and let $c\colon P\to B\pi$ be the classifying map. Then $k_P\cap x=0$ if and only if $x$ is in the image of $c_*\colon H_4(P;\Z^w)\to H_4(\pi;\Z^w)$.
\end{lem}
\begin{proof}
	There exists a $4$-manifold $M$ with a map $f\colon M\to B\pi$ such that $x=f_*[M]$ and $f^*w=w_1(M)$. This follows from the Atiyah-Hirzebruch spectral sequence for  (twisted) oriented bordism. 
Using surgery, we can assume that $f$ is an isomorphism on $\pi_1$. Then $f$ can be lifted to a map $\wh f\colon M\to P$ if and only if $f^*k_P\in H^3(M;\pi_2(P))$ vanishes. By Poincar\'e duality, this holds if and only if $f^*k_P\cap [M]$ vanishes. Since $f$ is an isomorphism on $\pi_1$, this holds if and only if $0=k_P\cap f_*[M]=k_P\cap x\in H_1(\pi;\pi_2(N)^w)$. 
	
	If a lift $\wh f$ exists, then $x=f_*[M]$ is in the image of $c_*\colon H_4(P;\Z^w)\to H_4(\pi;\Z^w)$.
	
	For the converse, let $x'\in H_4(P;\Z^w)$ be a preimage of $x$. As before, there exists a $4$-manifold $M'$ and a map $f'\colon M\to P$ with $f'_*[M']=x'$, $(f')^*c^*w=w_1(M)$ and such that $f'$ is an isomorphism on $\pi_1$. Since $c\circ f'\colon M'\to B\pi$ lifts to $P$, \[0=k_P\cap (c\circ f')_*[M']=k_P\cap c_*(x')=k_P\cap x\] as above. 
\end{proof}

\begin{lem}
	\label{lem:image-P2}
	Let $N$ be a closed 4-manifold with orientation character $(\pi,w)$.
	Let $H_4(\pi;\Z^w)=0$ or $H^1(\pi;\Z^w)=0$. Then the image of $H_4(P_2(N);\Z^w)$ in $H_4(\pi;\Z^w)$ is generated by the image of $[N]\in H_4(N;\Z^w)$.
\end{lem}
\begin{proof}
	This is obvious if $H_4(\pi;\Z^w)=0$. Hence we assume $H^1(\pi;\Z^w)=0$.
	
	Since $H_4(N;Z^w)\cong \Z$, the subgroup generated by the image of $[N]$ is the image of $H_4(N;Z^w)$.
	
	Consider the Leray-Serre spectral sequence applied to the fibrations $\wt{N}\to N\to B\pi_1N$ and $\wt{P_2(N)}\to P_2(N)\to B\pi_1N$ to obtain the following diagram, where the maps on the right are differentials on the $E^3$-page of the spectral sequence.
	\[\begin{tikzcd}
		H_4(N;\Z^w)\ar[d]\ar[r]&H_4(\pi_1N;\Z^w)\ar[d,"="]\ar[r]&H_1(\pi_1(N);\pi_2(N)^w)\ar[d,"="]\\
		H_4(P_2(N);\Z^w)\ar[r]&H_4(\pi_1N;\Z^w)\ar[r]&H_1(\pi_1(N);\pi_2(N)^w)
	\end{tikzcd}\]
	Since $H^1(\pi_1N;\Z[\pi_1N])=0$, the top row is exact by \cite[Theorem~4.1]{KPT}. The bottom row is always exact by \cite[Lemma~$8^{\text{bis}}.27$]{McCleary}. It follows that the images of $H_4(N;\Z^w)$ and $H_4(P_2N;\Z^{w})$ agree as claimed. 
\end{proof}

\begin{proof}[Proof of \cref{cor:multiple}]
	We abbreviate $w:=w_1^\pi (N)$. By \cref{thm:post} a $\phi$-immersion exists if and only if $\phi^*(k_N)\cap c_*[M]=0\in H_1(\pi_1M;\pi_2N^{w}_\phi)$. Since $\phi_*\colon H_1(\pi_1M ; \pi_2N_\varphi^{w}) \to H_1(\pi_1N;\pi_2N^{w})$ is injective by assumption, this holds if and only if \[k_N\cap \phi_*c_*[M]=\phi_*(\phi^*(k_N)\cap c_*[M])=0\in H_1(\pi_1N;\pi_2N^{w}).\] By \cref{lem:imageofP}, $k_N\cap \phi_*c_*[M]$ is trivial if and only if $\phi_*c_*[M]$ is in the image of $H_4(P_2N;\Z^{w})\to H_4(\pi_1N;\Z^{w})$. %It remains to show that the image is the subgroup generated by the $c_*[N]$. 
	
	This is always the case for $H_4(\pi_1N;\Z^{w})=0$ showing \eqref{it:cor:multiple1}. By \cref{lem:image-P2}, the image of $H_4(P_2N;\Z^{w})\to H_4(\pi_1N;\Z^{w})$ is the subgroup generated by the image of $[N]$. This shows \eqref{it:cor:multiple2}.
\end{proof}

The following is a generalization of \cref{cor:multiple} to free products in the orientable case. Recall that a free product decomposition $G\cong G_1\ast\ldots\ast G_k$ induces an isomorphism
$H_4(G;\Z)\cong\bigoplus_{i=1}^kH_4(G_i;\Z)$ given by the sum of projections $p_i\colon H_4(G;\Z)\to H_4(G_i;\Z)$.

\begin{cor}\label{cor:multiple2}
			Let $M,N$ be orientable and almost spin. Assume that $\varphi\colon \pi_1M\to \pi_1N$ satisfies $\varphi^*(w_2^\pi (N))=w_2^\pi (M)$ and induces a monomorphism 
	\[
	\phi_*: H_1(\pi_1M ; \pi_2N_\varphi) \rightarrowtail H_1(\pi_1N;\pi_2N).
	\]
	 If $\pi_1N\cong G_1\ast\ldots\ast G_k$ and each $G_i$ satisfies $H_4(G_i;\Z)=0$ or $H^1(G_i;\Z[G_i])=0$ then a $\varphi$-immersion exists if and only if $p_i(\varphi_*c_*[M])$ is a multiple of $p_i(c_*[N])$ in $H_4(G_i;\Z)$ for each $i=1,\dots,k$.
\end{cor}
Using that the multiples for different $i$ can vary, this shows that a $\phi$-immersion $M_*\imra N$ does not imply that $\varphi_*c_*[M]$ is a multiple of $c_*[N]$.
\begin{proof}
	As in the proof of \cref{cor:multiple}, we have to show that the image of $H_4(P_2N;\Z)\to H_4(\pi_1N;\Z)\cong \bigoplus_{i=1}^k H_4(G_i;\Z)$ is the subgroup generated by the $p_i(c_*[N])$.
	
	Since the images of $H_4(P_2(N);\Z)$ and $H_4(P_2(N\vee S^2);\Z)$ in $H_4(\pi_1N ;\Z)$ agree, the image only depends on the stable homeomorphism type of $N$. By \cite[Theorem~1]{Hillman95} and \cite{KLT}, $N$ is stably homeomorphic to a connected sum $N_1\#\ldots\#N_k$ with $\pi_1(N_i)\cong G_i$.  We have $P_2(N_1\#\ldots\#N_k)\simeq P_2(N_1\vee\ldots\vee N_k)$ and the composition $P_2(N_1\vee\ldots\vee N_k)\to B\pi_1N\to G_i$ factors through $P_2(N_i)$. On the other hand we have the composition $P_2(N_i)\to P_2(N_1\vee\ldots\vee N_k)\to B\pi_1N$. This implies that the image of $H_4(P_2(N);\Z)$ in $H_4(\pi_1N;\Z)$ is the direct sum of the images of $H_4(P_2(N_i);\Z)$ in $H_4(G_i;\Z)$. We thus have reduced the proof to showing that the image of $H_4(P_2(N_i);\Z)$ in $H_4(G_i;\Z)$ is the subgroup generated by $u_*[N_i]$, where $u\colon N_i\to BG_i$ is the classifying map. This again follows from \cref{lem:image-P2}.
\end{proof}

\begin{lem}\label{lem:1skeletonsum}
	Fix a $4$-manifold $N$ and an integer $k\in \N$. Then there exists a $4$-manifold $M$ with fundamental group $\pi_1N$ such that $w_i^\pi M=w_i^\pi N$ for $i=1,2, c_*[M]=k\cdot c_*[N]$ and an immersion $M_*\imra N$ inducing the identity on $\pi_1$.
\end{lem}
\begin{proof}
	We consider $\id_N$ as an element of $\Omega_4(\nu (N))$, the bordism group of pairs $(M,\wt\nu)$ consisting of a $4$-manifolds $M$ and a lift $\wt\nu$ of $\nu (M)\colon M\to BO$ along $\nu (N)$. Similiarly we can view $\bigsqcup_{i=1}^k N$ as an element in $\Omega_4(\nu (N))$. We can perform surgeries to turn $\bigsqcup_{i=1}^k N\to N$ into a 2-connected map $f\colon M\to N$. By definition of $\Omega_4(\nu (N))$, the map $f$ pulls back $w_i(N)$ to $w_i(M)$ and we use $f$ to identify $\pi_1M$ with $\pi_1N$. By \cref{thm:main} \eqref{it:main-iii}$\Rightarrow$\eqref{it:main-i}, there exists an immersion $M_*\imra N$ that is the identity on $\pi_1$.
\end{proof}
	
\begin{proof}[Proof of \cref{cor:stable2}]
	Let $\varphi_*c_*[M]$ be a multiple of $c_*[N]$. Since $M\leq M\#\CP^2$ we can assume that $M$ is not almost spin. By \cref{lem:1skeletonsum} we can assume that $\phi_*c_*[M]=c_*[N]$. Then a $\phi$-immersion $M_*\imra N$ exists by \cref{cor:stable}.
	
	For the converse assume that there exists a $\phi$-immersion $M_*\imra N$. As in the proof of \cref{cor:multiple}, we can apply \cref{lem:imageofP} to show that this implies that $\phi_*c_*[M]$ is in the image of $H_4(P_2N;\Z^{w})\to H_4(\pi_1N;\Z^{w})$. Since $H^1(\pi_1N;\Z[\pi_1N])=0$ this image is generated by $c_*[N]$ by \cref{lem:image-P2}. Hence $\phi_*c_*[M]$ is a multiple of $c_*[N]$. 
\end{proof}

We end this section by proving \cref{thm:shift} below which is a generalization of \cref{thm:trivialclass}. 
For its proof we will need the following lemma. 
\begin{lem}
	\label{lem:k-inv}
	Let $Y$ be a space with fundamental group $\pi'$ and let $(C_*Y,d_*^Y)$ denote its  $\Z[\pi']$-chain complexes. Let $X$ be a space together with a map $f\colon X\to B\pi'$. The 5-term exact sequence
	 \begin{equation}\label{eq:shift}
		0\to \ker d^Y_2\ra C_2Y\xrightarrow{d_2^Y}C_1Y\xrightarrow{d_1^Y}C_0Y\xrightarrow{\epsilon}\Z\to 0.
	\end{equation}
can be split into 3 short exact sequences and the composition of the corresponding Bockstein homomorphisms give homomorphisms $\shift\colon H^i(X;\Z)\xrightarrow{\shift} H^{i+3}(X;(\ker d^Y_2)_f)$. Here $(\ker d^Y_2)_f$ denotes $\ker d^Y_2$ viewed as a $\pi_1(X)$-module via the map $f$. Let $p\colon \ker d^Y_2\to \pi_2(Y)$ be the projection. Let $[\epsilon]\in H^0(X;\Z)\cong\Z$ be the canonical generator.

Then \[p_*(\shift([\epsilon]))=f^*k_Y\in H^3(X;\pi_2(Y)),\]
where $k_Y\in H^3(\pi';\pi_2(Y))$ is the first $k$-invariant of $Y$.
\end{lem}
\begin{proof}
	By naturality of $p_*$ and $\shift$, it suffices to consider the case $X=B\pi'$ and $f=\id$. 
	
	Let $(C_*,d_*)$ be a free resolution of $\Z$ as a $\Z[\pi']$-module starting with $(C_*Y,d_*^Y)$ in degrees at most $2$.
	The Bockstein homomorphisms are given by picking lifts, thus the diagram
	\[
\begin{tikzcd}
		&C_3\ar[d,"d_3"]\ar[r,"d_3"]\ar[dl,"d_3"']&C_2\ar[d,"d_2"]\ar[r,"d_2"]\ar[dl,"\id"']&C_1\ar[d,"d_1"]\ar[dl,"\id"']\ar[r,"d_1"]&C_0\ar[d,"\epsilon"]\ar[dl,"\id"']\\
	\ker d_2\ar[r]&C_2\ar[r,"d_2"]&C_1\ar[r,"d_1"]&C_0\ar[r,"\epsilon"]&\Z
\end{tikzcd}
	\]
	shows that $\shift([\epsilon])\in H^3(\pi';\ker d_2)$ is represented by $d_3\colon C_3\to \ker d_2$. Hence $p_*(\shift([\epsilon]))$ is represented by $d_3\colon C_3\to \ker d_2/\im d_3^Y$. Using the definition from \cite{EilenbergMacLane}*{(4.1)}, this class represents the $k$-invariant of $Y$.
\end{proof}

Set $\pi:=\pi_1M$, $\pi':=\pi_1N$ and let $(C_*M,d_*^M)$ and $(C_*N,d_*^N)$ denote the $\Z[\pi]$- and $\Z[\pi']$-chain complexes of $M$ and $N$ for some handle decompositions, respectively, computing the homology of universal coverings. $P_2M$ can be constructed from $M$ by attaching cells of dimensions $\geq 4$ and hence we may assume that the cellular $\Z[\pi]$-chain complex $(C_*(P_2M),d_*^{P_2M})$
agrees with that of $M$ in degrees $\leq 3$. By the Hurewicz theorem, $0=\pi_3(P_2M)\to H_3(P_2M;\Z[\pi])$ is onto, so $H_3(P_2M;\Z[\pi])=0$ and we get a free resolution 
\[
C_4(P_2M)\xrightarrow{d_4^{P_2M}}C_3M\xrightarrow{d_3^M}C_2M\ra \coker d_3^M\to 0
\]
of the $\Z[\pi]$-module $\coker d_3^M$, cf.~\cite[p.~91]{KPT}. It follows that a $\phi$-linear map 
\[
h\colon \coker d_3^M\ra (\coker d_3^N)_\phi
\]
 can be lifted by a triple $(h_4,h_3,h_2), h_i\colon C_iM\to C_iN$. In particular, $h$ induces a well defined map $h_A^3\colon H^3(P_2N;A)\to H^3(P_2M;A_\phi)$ for every $\Z[\pi']$-module $A$. 

We will mainly use the module $A:= \ker d^N_2$.
% that fits into the 5-term exact sequence
% \begin{equation}\label{eq:shift}
%	0\to \ker d^N_2\ra C_2N\xrightarrow{d_2^N}C_1N\xrightarrow{d_1^N}C_0N\xrightarrow{\epsilon}\Z\to 0.
%\end{equation}
%This sequence can be split into 3 short exact sequences and the
% composition of the corresponding Bockstein homomorphisms gives homomorphisms $H^i(X;\Z)\xrightarrow{\shift} H^{i+3}(X;A)$ for any space with fundamental group $\pi'$. Applying this to $X=B\pi'$ and $i=0$, we see that the
% generator $[\epsilon]\in H^0(B\pi';\Z)\cong\Z$ is sent by the shift to the $k$-invariant $k_{N^{(2)}}$ of the 2-skeleton $N^{(2)}$ of $N$, noting that $A \cong \pi_2N^{(2)}$. 
 Let $q\colon P_2M\to B\pi$ and $q'\colon P_2N\to B\pi'$ induce the identity on $\pi_1$.
%  then by naturality the generator $[\epsilon]\in H^0(P_2N;\Z)$ is shifted to
%\[
%\shift([\epsilon]) = (q')^*k_{N^{(2)}}\in H^3(P_2N;A).
%\]
Assume there exists an $h\colon \coker d_3^M\to (\coker d_3^N)_\phi$ such that the induced map $h^3_A$  makes the upper square in the following diagram commute:
\begin{equation}
	\label{eq:diagram}
	\begin{tikzcd}
		H^0(P_2N;\Z)\ar[d,"\shift"]\ar[r,"\cong"]&H^0(P_2M;\Z)\ar[d,"\shift"]\\
		H^3(P_2N;A)\ar[d,"p_*"]\ar[r,"h_{A}^3"]&H^3(P_2M;A_\phi)\ar[d,"p_*"]\\
		H^3(P_2N;\pi_2N)\ar[r,"h_{\pi_2N}^3"]&H^3(P_2M;\pi_2N_\phi)
	\end{tikzcd}
\end{equation}
Here the lower vertical maps are induced by the projection $p\colon A\sra A/\im d_3^N \cong \pi_2N$ and hence the lower diagram commutes.

\begin{lem}
	\label{lem:existence-h-1}
	 A map $h\colon \coker d_3^M\to (\coker d_3^N)_\phi$ such that $h^3_A$ makes Diagram \eqref{eq:diagram} commute exists if and only if there exists a map $P_2M\to P_2N$ that induces $\phi$ on fundamental groups.  
\end{lem}
\begin{proof}
First assume that $h$ exists. The composition $p_*\circ \shift$ sends the generator $[\epsilon]\in H^0(P_2N;\Z)$ to $(q')^*k_N\in H^3(P_2N;\pi_2N)$ by \cref{lem:k-inv}. Since this the obstruction for lifting along $P_2N\to B\pi'$ and the identity provides such a lift, $(q')^*k_N=0$.
% which vanishes since $q'$ factors through $P_2N$ via the identity. 
Moreover, $[\epsilon]\in H^0(P_2M;\Z)$ is sent to $q^*\phi^*k_N\in H^3(P_2M;\pi_2N)$ by \cref{lem:k-inv}, the only obstruction to lifting $(\phi\circ q)\colon P_2M\to B\pi'$ along $P_2N\to B\pi'$. The commutativity of diagram \eqref{eq:diagram} shows that
\[
q^*\phi^*k_N=h_A^3(q')^*k_N=h_{\pi_2N}^3(0)=0
\]
and hence there exists a map $P_2M\to P_2N$ inducing $\phi$ on fundamental groups. 

Conversely, a map $f\colon P_2M\to P_2N$ induces a chain map $h_*\colon C_*(P_2M)\to C_*(P_2N)_\phi$ and thus maps $h\colon \coker d_3^M\to (\coker d_3^N)_\phi$ and $h^3_A$. For this choice, diagram \eqref{eq:diagram}  commutes by naturality of the shift map.
\end{proof}

We will now phrase the existence of a map $h$ as above in terms of the fundamental classes of $M$ and $N$.
Let $c\colon M\to B\pi$ and $c'\colon N\to B\pi'$ be maps inducing the identity on fundamental groups. We have the following commutative diagram:
\begin{equation}
	\label{eq:diagram2}
	\begin{tikzcd}
	H^0(P_2N;\Z)\ar[d,"\shift"]\ar[r,"\cong"]&H^0(N;\Z)\ar[d,"\shift"]\ar[r,"\cong","PD"']&H_4(N;\Z^{w_1N})\ar[d,"\shift"]\ar[r,"c'_*"]&H_4(\pi';\Z^{w_1N})\ar[d,"\shift","\cong"']\\
	H^3(P_2N;A)\ar[r,"\cong"]&	H^3(N;A)\ar[r,"\cong","PD"']&H_1(N;A^{w_1N})\ar[r,"\cong","c'_*"']&H_1(\pi';A^{w_1N})
\end{tikzcd}
\end{equation}
Here the middle square commutes because capping is natural in the coefficients, the map $H_3(P_2N;A)\to H^3(N;A)$ is an isomorphism by \cite[Lemma~5.3]{KPT}, and $c'_*\colon H_1(N;A^{w_1N})\to H_1(\pi';A^{w_1N})$ is an isomorphism since $c'$ is $2$-connected. It follows that the lower composition $\psi'\colon H^3(P_2N;A)\to H_1(\pi';A^{w_1N})$ is an isomorphism. 

Let $\psi\colon H^3(P_2M;A_\phi)\to H_1(\pi;A_\phi^{w_1M})$ be the corresponding composition for $M$ instead of $N$. 
Since $M\to P_2M$ is $3$-connected, the map $H^3(P_2M;A_\phi)\to H^3(M;A_\phi)$ is injective but in general not an isomorphism. Hence the map $\psi$ is injective. 

\begin{lem}
	\label{lem:existence-h-2}
	 For a map $h\colon \coker d_3^M\to (\coker d_3^N)_\phi$, the induced map $h^3_A$ makes Diagram \eqref{eq:diagram} commute if and only if\[(\psi\circ h_A^3\circ (\psi')^{-1}\circ \shift)(c'_*[N])=\shift(c_*[M])\in H_1(\pi_1M;(A_\phi)^{w_1(M)}).
	 \]
\end{lem}
\begin{proof}
Since Diagram \eqref{eq:diagram2} commutes, the map $\psi'\circ\shift\colon H^0(P_2N;\Z)\to H_1(\pi';A^{w_1N})$ sends $1\in H^0(P_2N;\Z)$ to $\shift(c_*'[N])$. Similarly, the map $\psi$ sends $1\in H^0(P_2M;\Z)$ to $\shift(c_*[M])$.
	
Consider the following diagram.
\begin{equation*}
	\begin{tikzcd}
		H^0(P_2N;\Z)\ar[d,"\shift"]\ar[rr,"\cong"]&&H^0(P_2M;\Z)\ar[d,"\shift"]\\
		H^3(P_2N;A)\ar[d,"\psi'"]\ar[rr,"h_{A}^3"]&&H^3(P_2M;A_\phi)\ar[d,"\psi"]\\
		H_1(\pi;A^{w_1N})\ar[rr,"\psi\circ h_A^3\circ (\psi')^{-1}"]&&H_1(\pi;A_\phi^{w_1M})
	\end{tikzcd}
\end{equation*}
The bottom square of this diagram commutes by definition. Hence the top square of this diagram commutes if and only if the outer rectangle commutes. Since $\psi'(\shift(1))=\shift(c_*'[N])$ and $\psi(\shift(1))=\shift(c_*[M])$, the top square and thus diagram \eqref{eq:diagram} commutes if and only if $\psi\circ h_A^3\circ (\psi')^{-1}\circ \shift$ sends $c'_*[N]$ to $\shift(c_*[M])$ as claimed.
\end{proof}

Combining \cref{lem:existence-h-1,lem:existence-h-2}, we obtain the following result.
\begin{thm}
\label{thm:shift}
	Let $M,N$ be $4$-manifolds and let $\varphi\colon \pi_1M\to \pi_1N$ be a homomorphism. There is a map $f_2\colon P_2M\to P_2N$ inducing $\phi$ on fundamental groups if and only if there is a $\pi_1M$-linear homomorphism $h\colon \coker(d_3^M)\ra \coker(d_3^N)_\varphi$ such that
	\[(\psi\circ h_A^3\circ (\psi')^{-1}\circ \shift)(c'_*[N])=\shift(c_*[M])\in H_1(\pi_1M;(A_\phi)^{w_1(M)}).
	\]
\end{thm}
Thus we have shown that the existence of $f_2$ only depends on the fundamental classes $c'_*[N], c_*[M]$ in an explicit way. Note that we do not require $\phi^*w_1^\pi (N)=w_1^\pi (M)$, so an induced map $H_4(\phi)$ cannot appear in our statement.

\begin{proof}[Proof of \cref{thm:trivialclass}]
	If $c'_*[N]=0$, then $(\psi\circ h_A^3\circ (\psi')^{-1}\circ \shift)(c'_*[N])=0$. By \cref{thm:shift} there is a map $f_2\colon P_2M\to P_2N$ inducing $\phi$ on fundamental groups if and only if $\shift(c_*[M])=0$. This holds for the shift coming from  \eqref{eq:shift} if and only if it holds for the shift for any start of a free resolution of $\Z$ as a $\Z[\pi']$-module. Hence \cref{thm:trivialclass} follows.
\end{proof}

\section{Realization of fundamental classes and the proof of \cref{prop:Z4}}
\label{sec:james}
To answer the question which immersion types $(\pi,w_1,w_2,c)$ are realized, we only have to compute which values of $c\in H_4(\pi;\Z^{w_1})$  are realized for a given normal 1-type $(\pi,w_1,w_2)$ because all normal 1-types are realized by \emph{doubles} of 4-dimensional thickenings. These are in fact the boundaries of 5-dimensional thickenings of 2-complexes with normal data given by $w_1,w_2$, implying that their fundamental classes have trivial image $c$. For such 5-manifolds, transversality implies that the diffeomorphism type only depends on the presentation of the fundamental group (and $w_1,w_2$). 
See \cite{thickenings} for more details.
If $w_2^\pi \neq\infty$, consider the pullback

\[\begin{tikzcd}[column sep=2cm]
	B\ar[d,"\xi"]\ar[r,"u"]&B\pi\ar[d,"w_1^\pi\times w_2^\pi"]\\
	BO\ar[r,"w_1(\gamma)\times w_2(\gamma)"]&K(\Z/2,1)\times K(\Z/2,2)
\end{tikzcd}\]
where $\gamma$ is the universal bundle over $BO$. If $w_2^\pi=\infty$, then consider the following pullback:
\[\begin{tikzcd}
	B\ar[d,"\xi"]\ar[r,"u"]&B\pi\ar[d,"w_1^\pi"]\\
	BO\ar[r,"w_1(\gamma)"]&K(\Z/2,1)
\end{tikzcd}\]
In both cases we can assume that $\xi$ is a fibration and we denote by $\Omega_4(\xi)$ the 4-dimensional bordism group over $\xi$, i.e.\ the bordism group of manifolds with a lift of the normal bundle along $\xi$. 
The map $u\colon B\to B\pi$ induces a map $\Omega_4(\xi)\to H_4(\pi;\Z^{w_1})$ given by sending $[M\xrightarrow{f}B]$ to $u_*f_*[M]$. Given $[M\xrightarrow{f}B]\in \Omega_4(\xi)$ we can apply surgery to obtain a manifold $M'$ bordant to $M$ over $\xi$ with immersion type $(\pi,w_1,w_2,u_*f_*[M])$. Conversely, a manifold $M$ with immersion type $(\pi,w_1,w_2,c)$ admits a lift of its normal bundle along $\xi$. Hence we have to compute the image of  $\Omega_4(\xi)\to H_4(\pi;\Z^{w_1})$.
For this we apply the James spectral sequence \citelist{\cite{teichnerthesis}*{Theorem~3.1.1}\cite{teichner-signature}*{Section~II}} with $E^2$-page
\[E^2_{p,q}=H_p(\pi;\Omega_q^{\Spin})\Longrightarrow \Omega_{p+q}(\xi)\quad\text{for}~ w_2\neq \infty\]
respectively
\[E^2_{p,q}=H_p(\pi;\Omega_q^{\SO})\Longrightarrow \Omega_{p+q}(\xi)\quad\text{for}~ w_2=\infty.\]
Note that the coefficients in the spectral sequence are twisted via $w_1$ and that the edge homomorphism is exactly our map $\Omega_4(\xi)\to H_4(\pi;\Z^{w_1})$. As a consequence, its image agrees with $E_{4,0}^\infty$, which is the iterated kernel of all differentials going out of any $E_{4,0}^r$. 

\begin{lem}
	If $w_2=\infty$, $\Omega_4(\xi)\to H_4(\pi;\Z^{w_1})$ is surjective. Hence all immersion types $(\pi,w_1,\infty,c)$ can be realized.
\end{lem}
\begin{proof}
	Since $\Omega_q^{\SO}=0$ for $q=1,2,3$, there are no non-trivial differentials out of $E^2_{4,0}\cong H_4(\pi;\Z^{w_1})$. Hence $\Omega_4(\xi)\to H_4(\pi;\Z^{w_1})$ is surjective as claimed.
\end{proof}
For $w_2\neq \infty$, we have $\Omega_1^{\Spin}=\Omega_2^{\Spin}=\Z/2$. Let $\Sq^2_{w_1,w_2}$ denote the homomorphism $H^p(\pi;\Z/2)\to H^{p+2}(\pi;\Z/2)$ sending $x$ to $\Sq^2(x)+\Sq^1(x)\cup w_1+x\cup w_2$. The following lemma is \cite[Proposition~1]{teichner-signature} in the case $w_1=0$.  A more detailed proof can be found in \cite[Section~5.1]{OP-MCG}. The proof in the case $w_1\neq 0$ is similar as we briefly explain. In the proof of \cite[Lemma~5.3]{OP-MCG}, $w_1(v)=0$ is used such that $w_1(v)\Sq^1(x)=0$ in the highlighted equation. Without this, in the statement of the lemma $\Sq_{w_2}$ has to be replaced by $\Sq_{w_1,w_2}$. Carrying this change through the rest of the proof gives the desired result.

\begin{lem}
	\label{lem:differentials}
	The $d_2$-differential $H_4(\pi;\Z^{w_1})\cong E^2_{4,0}\to E^2_{2,1}\cong H_2(\pi;\Z/2)$ is given by reduction mod 2 composed with the dual of $\Sq^2_{w_1,w_2}$. 
	
	The $d_2$-differential $H_3(\pi;\Z/2)\cong E^2_{3,1}\to E^2_{1,2}\cong H_1(\pi;\Z/2)$ is given by the dual of $\Sq^2_{w_1,w_2}$.
\end{lem}
For example, if $\pi=\Z/2k, w_1\neq 0, w_2\neq 0,\infty$ this lemma implies that there exists a 4-manifold whose fundamental class has non-trivial image in $H_4(\pi;\Z^{w_1})$ because all differentials vanish at this term and hence the edge homomorphism is surjective. In Remark~\ref{rem:realization}, we will explicitly construct a rational homology $4$-ball in this stable diffeomorphism class.

Recall that for almost spin manifolds with fundamental group $\Z^4$, $w_2^\pi (M)$ lies in  
\[
H^2(\Z^4;\Z/2)/\!\Aut(\Z^4) = \{0, e_1 \cup e_2, e_1 \cup e_2+e_3 \cup e_4\},
\]
where $e_i\in H^1(\Z^4;\Z/2)$ are the four standard pullbacks of the generator in $H^1(\Z;\Z/2)$.

\begin{proof}[Proof of \cref{prop:Z4}]
	(1) By \cref{prop:easy}, we know that $M\leq N$ and $N$ spin implies that $M$ is spin. This shows one direction. For the other let $[w_2^\pi (M)] = e_1 \cup e_2$ and let $N$ be non-spin. Then there is a map $\phi\colon \pi_1M\xrightarrow{\phi_1}\Z^2\xrightarrow{\phi_2}\pi_1N$ with $\phi^*w_2^\pi (N)=w_2^\pi (M)$. The map $\phi_1$ is realized by a map $f_1\colon M\to T^2$ and the map $\phi_2$ is realized by a map $f_2\colon T^2\to N$. Hence there exists a map $f=f_1\circ f_2$ from $M$ to $N$ with $f^*w_2(N)=w_2(M)$. Thus $M\leq N$ by \cref{thm:main} \eqref{it:main-iii}$\Rightarrow$\eqref{it:main-i}.
	
	(2) Assume there exists an immersion $M_*\to N$ that induces $\phi\colon \Z^4\to \Z^4$ on fundamental groups. Then \[e_1\cup e_2+e_3\cup e_4=w_2^\pi(M)=\phi^*w_2^\pi(N).\]
	Then $w_2^\pi(N)=e_1\cup e_2+e_3\cup e_4\in H^2(\Z^4;\Z)/\Aut(\Z^4)$ and $\phi$ is the inclusion of a finite index subgroup. The latter is true since otherwise $\phi$ factors through $\Z^3$ but $H^2(\Z^3;\Z)/\Aut(\Z^3)$ consists only of two elements and it would follow that $\phi^*x=e_1\cup e_2\in H^2(\Z^4;\Z)/\Aut(\Z^4)$. This shows part of the 'only if'-direction.
	
	To finish the proof of (2), we now assume that $\phi\colon \Z^4\to \Z^4$ is the inclusion of a finite index subgroup that pulls back $w_2^\pi(N)$ to $w_2^\pi(M)$, and show that an immersion $M_*\to N$ inducing $\phi$ on fundamental groups exists if and only if $c_*[M]$ is a multiple of $c_*[N]$. This completes the proof of the 'only if'-direction of (2) and shows the 'if'-direction by taking $\phi$ to be the isomorphism sending $w_2^\pi(N)$ to $w_2^\pi(M)$.
	
	Let $N'$ be the finite covering of $N$ corresponding to $\im\phi$. 
	We then have a commutative diagram
	\[\begin{tikzcd}
		N'\ar[r]\ar[d,"c'"]&N\ar[d,"c"]\\
		T^4\ar[r,"{\phi_*}"]&T^4
	\end{tikzcd}\]
	where $c'$ induces an isomorphism on fundamental groups. The degree of the coverings $N'\to N$ and $T^4\xrightarrow{\phi_*}T^4$ both are $[\Z^4:\im \phi_*]$. Thus the degrees of $c$ and $c'$ agree, that is, $c_*[N']=c_*[N]\in H_4(\Z^4;\Z)$. There is an immersion $M_*\imra N$ inducing $\phi$ on fundamental groups if and only if there is an immersion $M_*\imra N'$ that induces the isomorphism determined by $c'$ on fundamental groups. By \cref{cor:multiple}, such an immersion exists if and only if $c_*[M]$ is a multiple of $c'_*[N']=c_*[N]$ as claimed. This shows (2). 
	
	It remains to show the realization statement.
	Let $\xi\colon B\to BSO$ be the fibration determined by $w_2=e_1\cup e_2+e_3\cup e_4$. We compute the image of $\Omega_4(\xi)$ in $H_4(\Z^4;\Z)$ using the James spectral sequence. Note that $\Sq^2\colon H^2(\Z^4;\Z/2)\to H^4(\Z^4;\Z/2)$, which is the cup-square on $H^2(T^4;\Z/2)$, is trivial. By \cref{lem:differentials}, the $d_2$-differential $\Z\cong H_4(\Z^4;\Z)\to H_2(\Z^4;\Z/2)$ is thus given by the reduction mod 2 composed with the dual of $-\cup w_2$. This map is non-trival and hence the kernel is $2\Z\leq \Z\cong H_4(\Z^4;\Z)$. The $d_2$-differential $H_3(\Z^4;\Z/2)\to H_1(\Z^4;\Z/2)$ is again dual to $-\cup w_2$ which is an isomorphism for $w_2=e_1\cup e_2+e_3\cup e_4$. Hence the terms $E^3_{1,2}\cong H_1(\Z^4;\Z/2)/\im d_2=0$ and $E^4_{0,3}\cong E^2_{0,3}\cong H_3(\Z^4;0)=0$ are trivial and there are no higher non-trivial differentials with domain $E^3_{4,0}\cong 2\Z\leq H_4(\Z^4;\Z)$. Thus all immersion types $(\Z^4,0,e_1\cup e_2+e_3\cup e_4,2c)$ can be realized.
\end{proof}

\section{The immersion order on $4$-manifolds with cyclic fundamental group}
\label{sec:cyclic}
We start with the following lemma about the realization of immersion types for cyclic fundamental groups.
\begin{lem}
	\label{lem:image-fund-class-cyclic}
	Let $M$ be a $4$-manifold with fundamental group $\Z/2k$ and $w_2(M)=0$. Then $0=c_*[M]\in H_4(\Z/2k;\Z^{w_1})$.
\end{lem}
\begin{proof}
	If $w_1=0$, then we have $H_4(\Z/2k;\Z)=0$ and the statement is obvious. If $w_1\neq 0$ and $w_2=0$, we can compute the image of $\Omega_4(\xi)$ in $H_4(\Z/2k;\Z^{w_1})\cong \Z/2$, where $\xi$ is the fibration associated to $w_1,w_2$, as at the beginning of \cref{sec:james}. The $d_2$-differential $H_4(\Z/2k;\Z^{w_1})\to H_2(\Z/2k;\Z/2)$ in the James spectral sequence is the composition of reduction mod 2 and $\Sq^2_{w_1}$ by \cref{lem:differentials}. $\Sq^1$ is trivial on $H^2(\Z/2k;\Z/2)$ because it's a reduction of an integral class, while $\Sq^2\colon H^2(\Z/2k;\Z/2)\to H^4(\Z/2k;\Z/2)$ is an isomorphism and also the reduction mod 2 is an isomorphism, it follows that the $d_2$ differential is non-trivial. Hence it has trivial kernel and thus the image of $\Omega_4(\xi)$ in $H_4(\Z/2k;\Z^{w_1})$ is trivial. Hence $c_*[M]=0$ as claimed.
\end{proof}

\begin{proof}[Proof of \cref{thm:cyclic-or}]
	By \cref{prop:easy}, all spin manifolds are immersion equivalent and all not almost spin manifolds are immersion equivalent. In the orientable case, $H_4(C;\Z)=0$ for every cyclic group $C$. Hence remains to show that $M(2^mk)\bowtie M(2^m)$ for $k$ odd and that $M(2^m)\leq M(2^n)$ if and only if $n\leq m$.
	
	By \cref{thm:post-general} \eqref{it:post-general31}$\Rightarrow$\eqref{it:post-general1}, for every homomorphism $\phi\colon \Z/l_1\to \Z/l_2$ there is a map $P_2M(l_1)\to P_2M(l_2)$ that induces $\phi$ on fundamental groups. The statement now follows since a map $\phi\colon \Z/l_1\to \Z/l_2$ that is non-trivial on $H^2(-;\Z/2)$ exists if $l_2=l_1k$ with $l_2$ even and if $l_1=l_2k$ with $k$ odd, but such a map does not exist if $l_1=l_2k$ with $k$ even. This is elementary since the cohomology group is either trivial or $\Z/2$, given by  extensions of a cyclic group by $\Z/2$. 
\end{proof}

\begin{rem}\label{rem:realization-oriented}
An explicit representative of $M(n)$ with fundamental group $\Z/n$ is given by a thickening of a $2$-complex, with one $d$-cell in each dimension $d=0,1,2$. The normal data are given by $w_1=0, w_2\neq 0$ and can be realized for example by a handle decomposition (with one d-handle for each index $d=0,1,2$) of a $5$-manifold. It has $M(n)$ as its boundary, or equivalently, we thicken to a $4$-manifold and take $M(n)$ to be the double of this $4$-manifold (with boundary).  
\end{rem}

In \cite{KPT} the category $\hCh_2(\pi)$ was considered whose objects are free $\Z[\pi] $-chain complexes $C_*$ concentrated in non-negative degrees such that $H_n(C_*) = 0$ for $n\neq 0,2$, together with a fixed identification $H_0(C_*)\cong\Z$, and whose morphisms are chain maps that induce the identity on $H_0$, considered up to chain homotopy. Let $(\Z[\pi] )^n[2]$ denote the chain complex given by the free $\Z[\pi] $-module with a rank~$n$ basis, and concentrated in degree 2. We call two chain complexes $C_*,C'_*$ stably isomorphic if there exists $p,q$ such that $C_*\oplus (\Z[\pi] )^p[2]$ and $C'_*\oplus (\Z[\pi] )^q[2]$ are isomorphic in $\hCh_2(\pi)$, meaning that they are chain homotopy equivalent. We denote the set of stable isomorphism classes by $\sCh_2(\pi)$. To a connected CW-complex $Y$ with fundamental group $\pi$ we can associate an element $\mathsf{J}(Y)$ in $\sCh_2(\pi)$ by taking the cellular $\Z[\pi] $-chain complex of $Y$ and replacing it in degrees $>3$ by a free resolution of $\ker(d_3^Y)$. Note that $\mathsf{J}(Y)=\mathsf{J}(P_2(Y))$.

\begin{prop}\label{prop:stable2type}
	Let $M$ be a non-orientable $4$-manifold with fundamental group $\Z/2k$. If $c_*[M]\in H_4(\Z/2k;\Z^w)\cong \Z/2$ is non-trivial, then $\mathsf{J}(M)\in \sCh_2(\Z/2k)$ is represented by the $\Z[\Z/2k]$-chain complex $X$ given by
	\[\Z[\Z/2k]\xrightarrow{\mathsf{N}_w}\Z[\Z/2k]\xrightarrow{\mathsf{N}}\Z[\Z/2k]\xrightarrow{1-a}\Z[\Z/2k],\]
	in degrees $\leq 3$ and continues with a free resolution of $\ker \mathsf{N}_w$. Here $a$ is a generator of $\Z/2k$, $\mathsf{N}=\sum_{i=0}^{2k-1}a^i$ is the norm element and $\mathsf{N}_w:=(1-a)\sum_{i=0}^ka^{2i}$ is the twisted norm element.
\end{prop}
\begin{proof}
	The start of a free resolution of $\coker \mathsf{N}_w$ is given by
	\[
\Z[\Z/2k]\xrightarrow{1+a}\Z[\Z/2k]\xrightarrow{\mathsf{N}_w}\Z[\Z/2k]\ra \coker \mathsf{N}_w\to 0.
	\]
	Hence an elements of $\Ext_{\Z[\Z/2k]}^1(\coker \mathsf{N}_w,\ker \mathsf{N})$ are represented by maps $\alpha\colon \Z[\Z/2k]\to \ker \mathsf{N}$ such that $\alpha\circ (1+a)=0$. The class represented by $\alpha$ is trivial if and only if $\alpha$ factors as $\Z[\Z/2k]\xrightarrow{\mathsf{N}_w}\Z[\Z/2k]\xrightarrow{\alpha'}\ker \mathsf{N}$.
		
	Since $\mathsf{N}\cdot \mathsf{N}_w = \mathsf{N}\cdot (1-a)\sum_{i=0}^ka^{2i} = 0$, the map $\Z[\Z/2k]\xrightarrow{\mathsf{N}_w}\ker \mathsf{N}$ determines an element $x\in \Ext_{\Z[\Z/2k]}^1(\coker \mathsf{N}_w,\ker \mathsf{N})$. We first show that $x$ is non-trivial. The map $\Z[\Z/2k]\xrightarrow{\mathsf{N}_w}\ker \mathsf{N}$ factors through $\Z[\Z/2k]\xrightarrow{\mathsf{N}_w}\Z[\Z/2k]\xrightarrow{\alpha'}\ker \mathsf{N}$ if and only if there exists a map $f\colon \Z[\Z/2k]\to \ker \mathsf{N}$ that restricts to the identity on $\ker\mathsf{N}$. Assume that $f$ exists. Then we can choose the following lift of the identity on $\Z$:
	\[\begin{tikzcd}
		\ldots\ar[r]&\Z[\Z/2k]\ar[r,"\mathsf{N}"]\ar[d,"0"]&\Z[\Z/2k]\ar[r,"1-a"]\ar[d,"0"]&\Z[\Z/2k]\ar[r,"\mathsf{N}"]\ar[d,"\id-f"]&\Z[\Z/2k]\ar[r,"1-a"]\ar[d,"\id"]&\Z[\Z/2k]\ar[r,"\epsilon"]\ar[d,"\id"]&\Z\ar[d,"\id"]\\
		\ldots\ar[r]&\Z[\Z/2k]\ar[r,"\mathsf{N}"]&\Z[\Z/2k]\ar[r,"1-a"]&\Z[\Z/2k]\ar[r,"\mathsf{N}"]&\Z[\Z/2k]\ar[r,"1-a"]&\Z[\Z/2k]\ar[r,"\epsilon"]&\Z
	\end{tikzcd}\]
	This implies $H_n(\Z/2k;\Z)=0$ for $n>2$ and thus gives a contradiction.
	
	 Consider the map $\Theta\colon \Ext_{\Z[\Z/2k]}^1(\coker \mathsf{N}_w,\ker \mathsf{N})\to \sCh_2(\Z/2k)$ from \cite[(6.4)]{KPT}. By \cite[Propositions~1.9 and~7.1]{KPT}, the composition \[H_4(\Z/2k;\Z^w)\xrightarrow{\cong} \Ext_{\Z[\Z/2k]}^1(\coker \mathsf{N}_w,\ker \mathsf{N})\xrightarrow{\Theta} \sCh_2(\Z/2k)\] sends $c_*[M]$ to $\mathsf{J}(M)$. Since $c_*[M]$ and $x$ are the unique non-trivial element, we have $\Theta(x)=\mathsf{J}(M)$. It remains to show that $\Theta(x)=X$. 
	
	The map $\Theta$ from \cite[(6.4)]{KPT} sends $x$ to the class represented by the chain complex given by
	\[\Z[\Z/2k]\xrightarrow{(\mathsf{N}_w,\mathsf{N}_w)}\Z[\Z/2k]^2\xrightarrow{\mathsf{N}+0}\Z[\Z/2k]\xrightarrow{1-a}\Z[\Z/2k]\]
	together with a free resolution of $\ker(\mathsf{N}_w,\mathsf{N}_w)$ in higher degrees. 
	
	By a base change in degree $2$, the above chain complex representing $\Theta(x)$ is isomorphic to the chain complex
	\[\Z[\Z/2k]\xrightarrow{(\mathsf{N}_w,0)}\Z[\Z/2k]^2\xrightarrow{\mathsf{N}+0}\Z[\Z/2k]\xrightarrow{1-a}\Z[\Z/2k]\]
	together with a free resolution of $\ker(\mathsf{N}_w,0)\cong \ker(\mathsf{N}_w,\mathsf{N}_w)$ in higher degrees. This chain complex is a stabilization of $X$ and hence $\Theta(x)=X$ as needed. 
%	
%	We now show that $x$ is non-trivial. $x$ is trivial if and only if the map $\Z[\Z/2k]\xrightarrow{\ker(\mathsf{N}_w,\mathsf{N}_w)_w}\ker \mathsf{N}$ factors through $\Z[\Z/2k]\xrightarrow{\ker(\mathsf{N}_w,\mathsf{N}_w)_w}\Z[\Z/2k]$ and this is the case if and only if there exists a map $f\colon \Z[\Z/2k]\to \ker \mathsf{N}$ that restricts to the identity on $\ker\mathsf{N}$. Assume that $f$ exists. Then we can choose the following lift of the identity on $\Z$:
%	\[\begin{tikzcd}
%		\ldots\ar[r]&\Z[\Z/2k]\ar[r,"\mathsf{N}"]\ar[d,"0"]&\Z[\Z/2k]\ar[r,"1-a"]\ar[d,"0"]&\Z[\Z/2k]\ar[r,"\mathsf{N}"]\ar[d,"\id-f"]&\Z[\Z/2k]\ar[r,"1-a"]\ar[d,"\id"]&\Z[\Z/2k]\ar[r,"\epsilon"]\ar[d,"\id"]&\Z\ar[d,"\id"]\\
%		\ldots\ar[r]&\Z[\Z/2k]\ar[r,"\mathsf{N}"]&\Z[\Z/2k]\ar[r,"1-a"]&\Z[\Z/2k]\ar[r,"\mathsf{N}"]&\Z[\Z/2k]\ar[r,"1-a"]&\Z[\Z/2k]\ar[r,"\epsilon"]&\Z
%	\end{tikzcd}\]
%	This implies $H_n(\Z/2k;\Z)=0$ for $n>2$ and thus gives a contradiction.
\end{proof}

As in \cref{thm:cyclic-nor}, let $N(k,w_2,c)$ denote the immersion equivalence class represented by non-orientable $4$-manifolds with fundamental group $\pi=\Z/k$, $w_2\in\Z/2\cup\{\infty\}$ and image of the fundamental class $c$ in $H_4(\pi;\Z^{w_1})=\Z/2$. 

\begin{rem}\label{rem:realization}
	Most immersion classes $N(k,w_2,c)$ are easy to realize. If $c=0$ they come from doubles as in the oriented case. For $w_2=\infty$ we can take connected sum of $\CP^2$ with a $4$-manifold with $w_2\neq 0,\infty$. By \cref{lem:image-fund-class-cyclic}, $N(k,0,1)$ does not exist. Hence the only remaining case is $N(2k,1,1)$ which can be realized by a $4$-manifold $R(2k)$ with fundamental group $\Z/2k$ given as follows: 
	
	Start with the easiest 4-dimensional thickening $X(k)$ with one 0- and one oriented 1-handle, where the attaching circle of the 2-handle goes $k$ times around the fundamental group generator of $S^1 \times S^2$  in the easiest possible way, up to isotopy, but has framing $k-1$. It is well known, see for example \cite[Section~3]{Lekili} for a contact approach, that the boundary of $X(k)$ is the lens space $L(k^2, k-1)$.
	
	The double cover of $L(2k^2,k-1)$ is $L(k^2,k-1)$. This can be seen from the definition of $L(p,q)$ as the quotient of $S^3\subseteq \mathbb{C}^2$ by the action $(z_1,z_2)\mapsto (e^{2\pi i/p}z_1,e^{2\pi iq/p}z_2)$. Hence there is a free involution on the lens space $L(k^2, k-1)$, with quotient $L(2k^2,k-1)$, and we define $R(2k)$ as the quotient space of $X(k)$ by this free involution on its boundary. It is easy to see that this is a closed $4$-manifold with fundamental group $\Z/2k$ and Euler characteristic~$1$. Thus it must be non-orientable and cannot stably be a double (which have even Euler characteristic). Hence it must represent $N(2k,1, 1)$.
	
	Note that $X(1)$ is a 4-ball because its 1- and 2-handles cancel. As a consequence, $R(2)$ is diffeomorphic to $\RP^4$, obtained by identifying antipodal points on the boundary of the 4-ball. The $R(2k)$ are thus analogous of $\RP^4$, but with larger cyclic fundamental group.
\end{rem}

\cref{thm:cyclic-nor} is a consequence of the following two propositions.

\begin{prop}
	\label{prop:n-odd}
We have $N(k\cdot m,w_2,c)\bowtie N(k,w_2,c)$ if $m$ is odd. 
\end{prop}
\begin{proof}
	First consider $w_2\neq \infty$. Note that the projection $\Z/km\to \Z/k$ and the inclusion $\Z/k\to \Z/km$ pull back $w_1$ and $w_2$ correctly. Hence it suffices to show that these are realized as maps between the Postnikov 2-types. This follows from \cref{thm:post-general} \eqref{it:post-general31}$\Rightarrow$\eqref{it:post-general1} if $c=0$. For $c=1$ first note that the $m$-fold cover of $N(k\cdot m,w_2,1)$ is immersion equivalent to $N(k,w_2,1)$. This can be seen as follows. Let $\wh N$ be the $m$-fold cover of $N(k\cdot m,w_2,1)$. Then $\pi_1(\wh N)\cong \Z/k$. Since $m$ is odd, the covering map $\wh N\to N(k\cdot m,w_2,1)$ induces isomorphisms on $H^1(-;\Z/2), H^2(-;\Z/2)$ and $H_4(-;\Z^{w_1})$. Hence $\wh N\bowtie N(k,w_2,1)$ as claimed. 
	In particular $N(k,w_2,1)$ immerses into $N(k\cdot m,w_2,1)$ as claimed.
	
	To obtain an immersion $N(k\cdot m,w_2,1)_*\to N(k,w_2,1)$ we can use the model from \cref{prop:stable2type}. Let $X(km)$ denote the model of $\mathsf{J}(N(k\cdot m,w_2,1))$ and let $X(k)$ denote the model of $\mathsf{J}(N(k,w_2,1))$. Let $p\colon \Z[\Z/km]\to \Z[\Z/k]$ be the map induced by the projection $\Z/km\to \Z/k$ and consider the following chain map $X(km)\to X(k)$:
	\[\begin{tikzcd}
		\Z[\Z/km]\ar[r,"\mathsf{N}_w(km)"]\ar[d,"m^2p"]&\Z[\Z/km]\ar[r,"\mathsf{N}(km)"]\ar[d,"mp"]&\Z[\Z/km]\ar[r,"1-a"]\ar[d,"p"]&\Z[\Z/km]\ar[d,"p"]\\
		\Z[\Z/k]\ar[r,"\mathsf{N}_w(k)"]&\Z[\Z/k]\ar[r,"\mathsf{N}(k)"]&\Z[\Z/k]\ar[r,"1-a"]&\Z[\Z/k]
	\end{tikzcd}\]
	Here $\mathsf{N}_w(mk),\mathsf{N}_w(k)$, $\mathsf{N}(mk)$ and $\mathsf{N}(k)$ refer to the (twisted) norm element for $\Z/mk$ and $\Z/k$, respectively, while $a$ denotes the generator in both groups. 	
	
	By \cref{thm:post-general} the existence of such a chain map implies that there is map $P_2(N(k\cdot m,w_2,1))\to P_2(N(k,w_2,1))$ that induces $p$ on fundamental groups. Hence there is an immersion $N(k\cdot m,w_2,1)_*\to N(k,w_2,1)$ as claimed.
	
	For $w_2=\infty$, we have $N(k,w_2,c)\bowtie N(k,0,c)\#\CP^2$. Hence this case follows from the case $w_2\neq\infty$.
\end{proof}

\begin{prop}\label{prop:n-structure}
	For $n,k>0$ we have the following.
\begin{enumerate}
	\item\label{it:n-struc-i} $N(2^k,w_2,c)\leq N(2^n,0,0) \iff w_2 =0$ and $k\geq n$. 
	\item\label{it:n-struc-ii} $N(2^k,w_2,c)\leq N(2^n,1,0) \iff$ ($w_2=0$ and $k>n$) or  $N(2^k,w_2,c)=N(2^n,1,0)$.
	\item\label{it:n-struc-iii} $N(2^k,w_2,c)\leq N(2^n,1,1) \iff$ ($w_2=0$ and $k>n$) or ($w_2=1$ and $k=n$).
	\item\label{it:n-struc-iv} $N(2^k,w_2,c)\leq N(2^n,\infty,0) \iff$ $c=0$ and $k\geq n$.
	\item\label{it:n-struc-v} $N(2^k,w_2,c)\leq N(2^n,\infty,1) \iff k\geq n$ and ($k=n$ or $c=0$).
\end{enumerate}	
\end{prop}
\begin{proof}
	First note that every map on fundamental groups is induced by a map of the Postnikov 2-types if $c=0$ by \cref{thm:post-general} \eqref{it:post-general31}$\Rightarrow$\eqref{it:post-general1}. We will also use that $c=0$ if $w_2=0$ by \cref{lem:image-fund-class-cyclic}.
	
	\eqref{it:n-struc-i}  For a homomorphism $\phi\colon \Z/2^k\to\Z/2^n$ the induced map on $H^1(-;\Z/2)$ is non-trivial, and hence an isomorphism, if and only if $\phi$ is surjective. Let $f\colon N(2^k,w_2,c)_*\to N(2^n,0,0)$ be an immersion that induces $\phi$ on fundamental groups. Then $\phi^*$ is non-trivial on $H^1(-;\Z/2)$ since $N(2^k,w_2,c)$ is non-orientable. Thus $\phi$ is surjective and $k\geq n$. Furthermore $w_2=\phi^*0=0$. This shows the only if direction of \eqref{it:n-struc-i}. Conversely, any map $f\colon P_2(N(2^k,0,0))\to P_2(N(2^n,0,0))$ that induces a surjection on fundamental groups pulls back the first two Stiefel--Whitney classes of $N(2^n,0,0)$ to those of $N(2^k,0,0)$. Hence $N(2^k,0,0)\leq N(2^n,0,0)$ by \cref{thm:main} \eqref{it:main-iv} $\Rightarrow$ \eqref{it:main-i}.
	
	\eqref{it:n-struc-ii} The argument is similar to that of \eqref{it:n-struc-i}. Again any immersion $N(2^k,w_2,c)_*\to N(2^n,1,0)$ has to be surjective on fundamental groups. Since $\phi\colon \Z/2^k\to\Z/2^n$ is an isomorphism on $H^2(-;\Z/2)$ if and only if it is injective, we see that $w_2=\phi^*1=0$ if $k>n$ and $w_2=\phi^*1=1$ if $k=n$. In the latter case, the map on fundamental groups has to be an isomorphism. But $N(2^n,1,c)\leq N(2^n,1,0)$ implies $c=0$ by \cref{cor:multiple}~\eqref{it:cor:multiple2} since $H^1(\Z/2^n;\Z[\Z/2^n])=0$. This shows the only if direction of \eqref{it:n-struc-ii}. For the converse we have to show that $N(2^k,0,)\leq N(2^n,1,0)$. This follows as in \eqref{it:n-struc-i} using again that $\phi^*w_2=0$ since $\phi$ is not injective.
	
	\eqref{it:n-struc-iii} The only if direction is the same as in \eqref{it:n-struc-ii}. For the converse, first note that $N(2^k,0,0)\leq N(2^n,1,0)$ for $k<n$ by \eqref{it:n-struc-ii}. Hence we only have to show that $N(2^n,1,0)\leq N(2^n,1,1)$. This holds by \cref{cor:multiple}~\eqref{it:cor:multiple2} since $H^1(\Z/2^n;\Z[\Z/2^n])=0$.
	
	%(3)	The direction $\Rightarrow$ follows as in (2). For the converse note that $c$ can only be nontrivial if $w_2$ is non-trivial in which case the map has to be an isomorphism on fundamental groups. We have $N(2^n,1,0)\leq N(2^n,1,1)$ by \cref{cor:multiple}.
	
	\eqref{it:n-struc-iv} Assume $N(2^k,w_2,c)\leq N(2^n,\infty,0)$. As before this implies $k\geq n$. For the only if direction it remains to show that $c=0$. Since $c=1$ implies that $w_2$ is non-trivial, it suffices to show that $N(2^k,1,1)_*$ does not immerse into $N(2^n,\infty,0)\bowtie N(2^n,0,0)\# \CP^2$. That $N(2^k,\infty,1)_*$ does not immerse into $N(2^n,\infty,0)$ then follows from \eqref{it:n-struc-v}. By \cref{prop:stable1} this is equivalent to showing that $w_2$ does not admit a lift to $H^2(N(2^k,1,1);\Z[\Z/2^n]_\phi)$ coefficients, where $\phi\colon \Z/2^k\to\Z/2^n$ is the projection. The map $\Z[\Z/2^n]_\phi\to \Z/2$ factors through $\Z[\Z/2]$ with the action given by the orientation character. Hence it suffices to show that $w_2$ does not lift to $H^2(N(2^k,1,1);\Z[\Z/2]^{w_1})\cong H^2(N(2^k,1,1);\Z[\Z/2])$. 
	
	This can be computed using the model $X$ of $\mathsf{J}(N(2^k,1,1))$ from \cref{prop:stable2type} since a lift exists if and only if it exists stably. We have $H^2(\Hom_{\Z[\Z/2^k]}(X,\Z[\Z/2]))\cong\Z/2^{k-1}$, $H^2(\Hom_{\Z[\Z/2^k]}(X,\Z))\cong\Z/2^k$ and $H^2(\Hom_{\Z[\Z/2^k]}(X,\Z/2))\cong\Z/2$. Since every composition $\Z/2^{k-1}\to\Z/2^k\to\Z/2$ is trivial, $w_2$ cannot have a lift to $H^2(N(2^k,1,1);\Z[\Z/2])$ as claimed.
	
	For the converse, note that if $k\geq n$, then $N(2^k,0,0)\leq N(2^n,0,0)\leq N(2^n,0,0)\#\CP^2\bowtie N(2^n,\infty,0)$ by \eqref{it:n-struc-i}. This implies $N(2^k,\infty,0)\bowtie N(2^k,0,0)\#\CP^2\leq N(2^n,0,0)\#\CP^2\bowtie N(2^n,\infty,0)$. Lastly, $N(2^k,1,0)\leq N(2^k,1,0)\#\CP^2\bowtie N(2^k,\infty,0)\leq N(2^n,\infty,0)$. This shows \eqref{it:n-struc-iv}.
	
	\eqref{it:n-struc-v} Assume  $N(2^k,w_2,c)\leq N(2^n,\infty,1)$. As before this implies $k\geq n$. If $k>n$, then as in \eqref{it:n-struc-iv} $c=1$ can only happen if $w_2$ admits a lift to $H^2(N(2^k,1,1);\Z[\Z/2])$ which does not exist. This shows the only if direction.
	
	For the converse, first note that $N(2^n,\infty,0)\leq N(2^n,\infty,1)$ by \cref{cor:multiple}~\eqref{it:cor:multiple2} since $H^1(\Z/2^n;\Z[\Z/2^n])=0$. Thus if $c=0$, then $N(2^k,w_2,0)\leq N(2^n,\infty,0)$ by \eqref{it:n-struc-iv} and hence also $N(2^k,w_2,0)\leq N(2^n,\infty,1)$. If $k=n$ and $c=1$, $w_2$ is non-trivial and hence $N(2^k,w_2,1)$ is $N(2^n,1,1)$ or $N(2^n,\infty,1)$. Since $N(2^n,1,1)\leq N(2^n,1,1)\#\CP^2\bowtie N(2^n,\infty,1)$ this shows \eqref{it:n-struc-v}.
\end{proof}

\begin{proof}[Proof of \cref{thm:cyclic-nor}]
	By \cref{lem:image-fund-class-cyclic} and \cref{rem:realization}, a manifold $N(k,w_2,c)$ exists if and only if $c = 0$ or $w_2 \neq0$. By \cref{prop:n-odd}, every immersion class can be represented by a manifold of the form $N(2^n,w_2,c)$. The structure of the order graph can be seen from \cref{prop:n-structure}.
\end{proof}

We know also compare the orientable and the non-orientable case. For this the following proposition suffices.
\begin{prop}
	$M(2^k)\leq N(2^n,w_2,c)$ if and only if ($k>n$ and $w_2=1$) or $w_2=\infty$.
\end{prop}
\begin{proof}
	If an immersion $f\colon M(2^k)\to N(2^n,w_2,c)$ exists, then $f^*w_2(N)=w_2(M)$ which is non-trivial. Hence for an immersion to exist, $w_2$ must be non-trivial.
	
	If $w_2=1$, the map $\phi\colon \Z/2^k\to \Z/2^n$ on fundamental groups must be injective but not surjective so that $w_1$ is pulled back trivially while $w_2$ is pulled back non-trivially. This implies $k>n$.
	
	If $w_2=\infty$, then $M(2^k)\leq \CP^2\leq N(2^n,\infty,c)$. Now assume $k>n$ and $w_2=1$, then we only have to show that there exists a map $P_2M(2^k)\to P_2(2^n,1,c)$ that induces an inclusion $\Z/2^k\to \Z/2^n$. But since $H_4(\Z/2^k;\Z)=0$, such a map exists by \cref{thm:post-general} \eqref{it:post-general31}$\Rightarrow$\eqref{it:post-general1}.
\end{proof}

To visualize this, we give the order graph containing exactly one manifold for each immersion equivalence class of $4$-manifolds with fundamental group $1$, $\Z/2$ and $\Z/4$.
	{\footnotesize
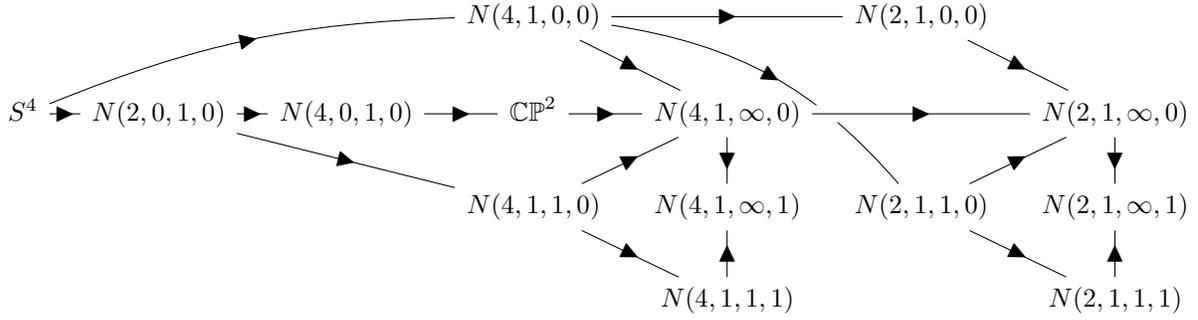
		\[\begin{tikzcd}[column sep=small,every arrow/.append style={no head,"\blacktriangleright" marking}]
			&&&N(4,0,0)\ar[rr]\ar[dr]\ar[ddrr, bend left =20]&&N(2,0,0)\ar[dr]&\\
			S^4\ar[r]\ar[urrr, bend left=10]&M(2)\ar[drr]\ar[r]&M(4)\ar[r]&\CP^2\ar[r]&N(4,\infty,0)\ar[d]\ar[rr, crossing over]&&N(2,\infty,0)\ar[d]\\
			&&&N(4,1,0)\ar[ur]\ar[dr]&N(4,\infty,1)&N(2,1,0)\ar[ur]\ar[dr]&N(2,\infty,1)\\
			&&&&N(4,1,1)\ar[u]&&N(2,1,1)\ar[u]
		\end{tikzcd}\]}
	
\newpage
	
\section{Appendix on QCA, a motivation for the partial order}
\label{sec:appendix}

\subsection{Definitions}
Classically, a cellular automaton is a local rule for updating states at some collection of sites, often the points of a lattice or a grid of points on a fixed manifold $M$. They are well known through Conway's game ``life'' which happens to be Turing complete \cite{wiki}. In the quantum case a quantum cellular automata (QCA) is again a local update rule, which can be expressed as acting on either (quantum) states or local operators. 

The operator point of view is preferable since infinite tensor products of Hilbert spaces are not well defined unless each Hilbert space is equipped with a unit norm \emph{vacuum vector}\footnote{Our (finite dimensional) operator algebras have the virtue of being Hilbert spaces with a distinguished element of norm 1, the multiplicative unit, and so fall into this favored category.}. In this case a completed colimit tensor product over the lattice of finite subsets can be used. Since the focus of this paper is compact manifolds the reader can imagine all definitions restricted to the case where we tensor over a finite index set; but we give the definitions in general.
We use the SWAP maps 
\[
\sigma_{H_1,H_2}: H_{1} \otimes H_{2} \cong H_{2} \otimes H_{1}, \quad h_1 \otimes h_2 \mapsto h_2 \otimes h_1
\]
 to turn complex Hilbert spaces (with morphisms given by \emph{all} linear maps) into a \emph{symmetric monoidal category} under the tensor product operation $\otimes$. This implies that there is a well defined tensor product $\otimes_{i\in \wedge} H_i$ for any finite set $\wedge$, even if it is not ordered, given by the colimit over all orderings of $\wedge$. For example, if $\wedge=\{\ast,\bullet\}$ is an \emph{unordered} set with two elements, then this colimit is simply given by
 \[
 \otimes_{i\in \wedge} H_i := \left( (H_\ast \otimes H_\bullet) \oplus (H_\bullet \otimes H_\ast)\right) / \langle (h\otimes k,0)- (0, k\otimes h)\, \forall h\in H_\ast, k\in H_\bullet \rangle 
 \]
 Scalar multiplication gives isomorphisms $\C\otimes H \cong H\cong H\otimes \C$ and we will identify these Hilbert spaces.
For all mathematical purposes, we can replace the large category of Hilbert spaces by any equivalent symmetric monoidal subcategory. 
 The smallest model for the category of finite dimensional (f.d.)\ Hilbert spaces would have exactly one object $\C^n$ for  each $n\in\N_0$, where the tensor product is defined as $\C^m\otimes \C^n := \C^{m\cdot n}$ and the SWAP maps become non-trivial involutions $\sigma_{m,n}: \C^{m\cdot n} \cong \C^{n\cdot m}=\C^{m\cdot n}$. Whenever needed, we assume that in our category of Hilbert spaces, $\C$ is the only 1-dimensional object. 

QCAs live on a metric space which we later take to be a Riemannian manifold $M$. In the end we define a covariant functor $Q$ from the category of Riemannian manifolds $\mathsf{Riem}$ (with boundary) and smooth proper maps as morphisms, to the category of abelian groups.\footnote{A continuous map is \emph{proper} if inverse images of compact sets are compact, i.e.\ if it preserves the set of sequences that diverge to $\infty$. Properness is automatic if the domain is compact.} Our functor takes disjoint union to direct product, so is most interesting on connected spaces (and vanishes on discrete ones). 

The theory of thickenings and their spines gives an equivalence of categories $\mathsf{Riem} \simeq \mathsf{Simp}$, where the latter consists of  locally finite simplicial complexes $X$ (with at most countably many simplices), together with continuous proper maps as morphisms. 
This allows us to switch back and forth between these categories, in particular $Q(X)$ is defined. Note that the dimension of a manifold thickening of $X$ is not well defined but can be assumed to be just above $2\dim(X)$. 
Later we will explain a modification $\wt{Q}(M)$ whose isomorphism class has the potential to depend on the smooth or PL structure of $M$, not just on the proper homotopy type.
\begin{defi}\label{def:qca}
	A \emph{geometric system} GS on a topological space $X$ is a collection 
	\[
	h=\{H_x, x\in X\}
	\]
	 of f.d.\ Hilbert spaces $H_x\neq \{0\}$, labelled by the points of $X$. The \emph{sites} $J(h)$ of $h$ are the points $x\in X$ for which $\dim H_x \neq 1$. We require that the sites $J(h)$ form a locally finite subset of $X$ and we refer to $H_x$ as the \emph{degrees of freedom (dof)} at the site $x$. 
\end{defi}

In the small model of the category of f.d.\ Hilbert spaces $\{\C^n, n\in\N_0\}$, a geometric system is simply given by a map $h: X \to \N,\, H_x:=\C^{h(x)}$, so that $h^{-1}(\N \smallsetminus \{1\})$ is locally finite.  In this model the geometric systems on $X$ form a set that one can topologize in a way that allows two sites ${x_1}$ and ${x_2}$ to collide at $x\in X$ along a continuous path, as long as the corresponding dofs satisfy $H_x=H_{x_1} \otimes H_{x_2}$. Running such a path in the opposite direction in the space of geometric systems, a site ${x}$ may split into sites ${x'_1}$ and ${x'_2}$,  provided $H_{x}=H_{x'_1} \otimes H_{x'_2}$. Note that $x_i\neq x'_i$ is possible in a concatenation of paths.

\begin{rem} \label{rem:tensor h}
There is a tensor product of geometric systems on $X$ given by the formula
\[
h \otimes h' = \{H_x, x\in X\} \otimes \{H'_x, x\in X\} := \{H_x\otimes H'_x, x\in X\}.
\]
The sites $J(h\otimes h') = J(h) \cup J(h')$ compose by the union of sets and in the small model $h,h': X\to \N$ we simply get the pointwise multiplication of dimensions.  

We can also push forward geometric systems along a proper map $f:X\to X'$ by setting $f_*(h):= \{H'_{x'}, x'\in X'\}$, where properness of $f$ makes $H'_{x'} := \otimes_{x| f(x)=x'} H_x$ into a finite tensor product. Later we will discuss pull backs of GS along certain types of immersions.
\end{rem}

We review some basic facts about endomorphism algebras of complex Hilbert spaces. For a f.d.\ $H$, it is a standard fact \cites{sinclair08,rordam21} that all automorphisms of $\End(H)$ are inner: For any algebra automorphism $\alpha: \End(H) \ra \End(H)$ , there is an invertible operator $\U : H \ra H$ so that for all $\Op \in \End(H)$, $\alpha(\Op) = \U ^{-1}\Op\U $. Such a $U$ is unique up to phase and if $\alpha$ is a $\ast$-automorphisms then $U$ can be chosen to be unitary. 

If $\wedge$ is finite and $H \coloneqq \otimes_{i \in \wedge} H_i$, then the tensor product of operators induces an algebra isomorphism $\End(H) \cong \otimes_{i \in \wedge} \End(H_i)$. 
Wedderburn theory tells us that among finite dimensional $\ast$-algebras being simple and being $\End(H)$ for some f.d.\ $H$, are equivalent conditions. It also tells us that all inclusions of one f.d.\ simple $\ast$-algebra into another, $A \subset E$, are \emph{trivial}, meaning there is another simple $\ast$-algebra $B \subset E$ with $A \cap B = \{\text{scalars}\}$ and $A \otimes B \cong A \cdot B = E$, the isomorphism canonical. $B$ is the commutant $A'$ of $A$ and $A$ is the commutant $B'$ of $B$. 

The algebra we discuss follows quickly from the facts above and the ability to do routine computations in $\End(H)$ for f.d.\ $H$. With this in mind, we remind the reader of convenient additive bases for $\End(H_i)$ and $\End(H)$, called the generalized-Pauli-word basis. It is helpful to expand the operators, below, $\Op$, $\mathcal{E}_i$, and $\mathcal{E}$ in this basis.

For a qubit $\C^2$, $\End(\C^2)$ has a linear basis of four operators $1 \coloneqq \begin{vmatrix} 1 & 0 \\ 0 & 1 \end{vmatrix}$, $X \coloneqq \begin{vmatrix} 0 & 1 \\ 1 & 0 \end{vmatrix}$, $Y \coloneqq \begin{vmatrix} 0 & -i \\ i & 0 \end{vmatrix}$, $Z \coloneqq \begin{vmatrix} 1 & 0 \\ 0 & -1 \end{vmatrix}$, the \emph{Pauli matrices}. For a qubit $\C^n$, $\End(\C^n)$ is similarly spanned by $n^2$ generalized Pauli matrices $\{C^iP^j\}$, $1 \leq i,j \leq n$, where $C$ and $P$ are \emph{clock} and \emph{phase} respectively:
\[
	C = \begin{vmatrix}
		0 & 1 & & & \\
		& 0 & 1 & & \\
		& & \ddots & 1 \\
		1 & & & & 0
	\end{vmatrix},\ P = \begin{vmatrix}
		1 & & & \\
		& e^{2\pi i/n} & & \\
		& & \ddots & \\
		& & & e^{2\pi i(n-1)/n}
	\end{vmatrix}
\]

Generalized Pauli words are tensor products of $\abs{\wedge}$ factors with a (generalized) Pauli in each slot. Such words are typically written without   $\otimes$ symbols and they are a basis for $\End(\otimes_{i \in \wedge} H_i)$.
Convention: When an operator in $\End(H)$ is described as a tensor product of only some factors $i \in \wedge$, we implicitly tensor with the identity on all remaining factors.

\begin{defi} Given a set $\{H_j,j\in J\}$ of f.d.\ Hilbert spaces, define 
\[
\End_J := \operatorname{colim}_{\wedge \subset J} \End(\otimes_{i \in \wedge} H_i),
\]
 where $\wedge$ runs through all finite subsets of $J$. We equip each endomorphism algebra $\End(H_j)$ with the inner product $\langle A, B \rangle = \frac{1}{\dim H_j} \operatorname{tr}(AB^\dagger)$ and extend this to finite tensor products. This normalization has the virtue that $A$ and $A \otimes \id_{\text{another system}}$ have equal norms. $\End_J$ may now be completed to a von Neumann algebra, $\widehat{\End}_J$, which is the GNS construction on our normalized tracial state.

For $\Op \in \widehat{\End}_J$, the \emph{support} of $\Op$, $\supp(\Op) \subset J$, is the set of all $j \in J$ such that there is an $\epsilon \in \End(H_j)$ with $[\Op,\epsilon] \neq 0$. Equivalently, $\supp(\Op)$ is
	\begin{itemize}
		\item The unique smallest $S \subset J$ so that $\Op \in \widehat{\End}_S$, and
		\item The complement of the unique largest subset $S \subset J$ so that for all $\mathcal{E} \in \widehat{\End}_S$, $[\Op,\mathcal{E}] = 0$.
	\end{itemize}
Finally, a \emph{pre-QCA} is a $\ast$-automorphism $\alpha$ of $\widehat{\End}_J$, the easiest example being $\alpha=$ identity.
\end{defi}
For a geometric system $h=\{H_x, x\in X\}$ we let $J(h)\subset X$ be the sites of $h$, i.e.\ those labels $x\in X$ for which $\dim H_x\neq 1$. We note that there is a canonical isomorphism $\End_X \cong \End_{J(h)}$, similarly for the completitions, because 1-dimensional Hilbert spaces do not contribute to the tensor product. Then the support $\supp(\Op)\subseteq J(h)$ of an operator $\Op \in \widehat{\End}_X$ is a locally finite subset of $X$. 

\begin{defi}
Given a geometric system GS $h$ on a metric space $X$, a pre-QCA $\alpha$ on $h$ is a $\ast$-automorphism of $\widehat{\End}_X$. $\alpha$ is said to be a \emph{QCA of $\range\leq r$} if and only if for all $\Op \in \widehat{\End}_X$ we have
\[
\supp(\alpha(\Op)) \subset \EuScript{N}_r(\supp(\Op)).
\]
Here $\EuScript{N}_r$ is the $r$-neighborhood of the sites in $X$ on which ``$\Op$ has support.'' Using the generalized Pauli-word-basis, one shows that $\range(\alpha)$ may be computed simply by looking at the action of $\alpha$ on single site operators. When $X$ is noncompact it is useful to refine the definition of $\range(\alpha)$ to a function $r: X \ra [0,\infty]$. We say $\range(\alpha) \leq r$ if for any operator $\Op$ on $H_x$ at site $x$ we have $\supp(\alpha(\Op)) \subset \EuScript{N}_{r(x)}(\supp(\Op))$. That is, the function $r$ must majorize the spread of support.

In the rest of the paper, we will sometimes abuse notation and just write $\alpha$ or $(h,\alpha)$ for the triple $(X,h,\alpha)$ if the context allows it.
\end{defi}

\begin{rem}\label{rem:tensor alpha}
The tensor product of geometric systems on $X$ from Remark~\ref{rem:tensor h} extends to QCA's by
\[
(h,\alpha) \otimes (h',\alpha') := (h \otimes h', \alpha \otimes \alpha')
\]
where we use the natural isomorphism $\End(H_x \otimes H'_x) \cong \End(H_x)\otimes \End(H'_x)$ of $\ast$-algebras. It is easy to see that $\range(\alpha \otimes \alpha') \leq \range(\alpha)+\range(\alpha')$.
\end{rem}

\begin{defi}
	We say $\alpha \in \operatorname{QCA}(X)$ if $\alpha$ is an \emph{admissible} QCA on a geometric system $(X,h)$. Where we call $\alpha$ \emph{admissible}, if there is an \emph{admissible deformation} $\delta$ (see Definition \ref{def:admissible}) of $\alpha$ to a QCA $\alpha^\pr$ of arbitrarily small range. For $X$ noncompact, arbitrarily small means less than any chosen majorant function. So QCA$(X)$ may be thought of as the ``smallest range'' QCAs on $X$, but to fully define QCA$(X)$ one must wait until Definition \ref{def:admissible}, and in fact the statement of Theorem \ref{thm:wrinkling}; we need to develop the concept of a deformation and its admissibility.
\end{defi}

\begin{defi}
	$Q(X)$ is the set of equivalence classes within QCA$(X)$, where $\alpha$ and $\alpha^\pr$ are equivalent if there is an admissible deformation $\delta$ connecting them. See Definition \ref{def:admissible} for admissibility of $\delta$.
\end{defi}

\begin{defi}
	A gate of \emph{radius} $\rho$ of a geometric system $(X,h)$ is a unitary transformation $\U_\rho(x)$ acting on a tensor product $\otimes_{y\in B_\rho(x)} H_y$, where $B_\rho(x)$ is the ball of radius $\rho$ about some point $x \in X$. Since sites $J(h)$ are locally finite, $J(h) \cap B_\rho(x)$ is finite and therefore the tensor product is over a finite index set.
\end{defi}

\begin{defi} 	
	A \emph{quantum circuit} is a composition of gates $\U_i =\U_\rho(x_i)$,  where the balls $B_\rho(x_i)$ are all disjoint. The conjugation action of the $\U_i$ on $\widehat{\End_X}$ gives a special kind of QCA for which we say that it has \emph{radius} $\rho$. One is often interested in quantum circuits with some bound on $\rho$ and the number of balls involved. Less restrictive is the notion of a finite depth quantum circuit (FDQC) where gate radius $\rho$ is (usually assumed) fixed and the number of balls finite, but not bounded (making FDQC closed under composition). Notice that range spreads out like a light-cone when a FDQC is applied. FDQC is only a useful concept when $X$ is noncompact.
\end{defi}

\begin{defi}
	Vitally important are ancillae. Adding ancillae to a QCA $(h,\alpha)$ on a metric space $X$ is simply given by the tensor product $(h,\alpha) \otimes (h',\alpha')$ from Remark~\ref{rem:tensor alpha}, applied to the case where $\alpha'=$~identity (but where $h'$ is arbitrary). 	
 
 Very generally, we'll allow the addition of ancillae and its inverse operation: removal of ancillae. Both a mathematical and physical motivation will be given shortly.
\end{defi}

\begin{defi}\label{def:admissible}
A deformation $\delta=\{\alpha_1,\alpha_2,\ldots,\alpha_n\}$ from $\alpha_1$ to $\alpha_n$ is a finite sequence of operations where at each step from $\alpha_i$ to $\alpha_{i+1}$ either a collection of ancillae is added or removed, or a quantum circuit is applied. We call $\delta$ \emph{admissable} if 
\[
\range(\delta):=\max_{i=1,\ldots,n}\range(\alpha_i) < \epsilon_0(X),
\]
with $\epsilon_0(X)$ defined in the wrinkling threshold Theorem \ref{thm:wrinkling} below.
\end{defi}

	Wrinkling is a technique, much like $h$-principle methods [EM02] for constructing admissible deformations $\delta$ and therefore admissible QCA $\alpha$. It is an algebraic analog of folding and wrinkling space; it will be explained shortly, and culminates in Theorem \ref{thm:wrinkling} below. The paired concepts (admissibility, wrinkling) lead to $Q$'s functorial properties.

\subsection{Background, motivation, and informal arguments}
In traditional developments \cite{gross12}, one treats a noncompact $X$, like the integers or real line. There, the simple, unquantified condition of \emph{finite range} is imposed on $\alpha \in \QCA(X)$. The resulting set  $\QCA^\fr(X)$ forms a group\footnote{Point (2) below verifies $\range(\alpha^{-1}) = \operatorname{range}(\alpha)$ so the inverse is also a QCA.} under composition. This group has a well-studied subgroup, the finite depth quantum circuits\footnote{A FDQC is a composition of finitely many \emph{rounds}, each of which applies \emph{gates} $\operatorname{U}_j$ to pairwise disjoint collections of $n_j = O(1)$ local Hilbert spaces $\{H_{i_1}, \dots, H_{i_{n_j}}\}$. A gate $\operatorname{U}_j$ is any unitary acting on $H_{i_1} \otimes \cdots \otimes H_{i_{n_j}}$.} (FDQC). In \cite{fhh} it was shown that: 1) FDQC$(X) \triangleleft \QCA^\fr(X)$ is normal and $Q(X) \coloneqq \lbar{\QCA}^\fr(X) \slash \lbar{\operatorname{FDQC}}(X)$ is abelian. The over-bar notation (not in the original references) means that the systems for both group and subgroup are stabilized by allowing \emph{ancillae}. Even lacking our formal framework, \cite{gross12} succeeded in a full calculation of $Q(\R)$; it is a countably generated abelian group, generated by left (or right) shifts of Hilbert spaces of various finite dimensions. Even more remarkable is the physics-level proof in \cite{fedhh} that $Q(T^3)$ contains 2 and 4 torsion, $T^3$ the 3-torus. The argument is grounded in algebraic $K$-theory but relies on an understanding of thermal currents in conformal field theories (CFT) not yet mathematically rigorous. In \cite{fhh20} use a covariant, local-implantation, to produce related (torsion) QCAs on any 3-manifold $M$, compact or noncompact.

Ancilla play a role similar to hyperbolic forms, $\begin{vmatrix} 0 & 1 \\ 1 & 0 \end{vmatrix}$, in $L$-theory, or trivial bundles in $K$-theory. ``Allowing ancilla'' means that at any time the domain of the site labeling function $s$ may be extended and the QCA $\alpha$ similarly extended to act as identity on the new Hilbert spaces (or sites). Conversely, the inverse operation is also permitted; if $\alpha$ acts trivially on a Hilbert space factor it may be removed from domain$(s)$. The physical motivation for ancilla is that any physical system can (or cannot) be considered to contain additional, initially passive, degrees of freedom (dof) which may later interact. For example, in a metal there will be \emph{valence band} electrons not relevant to transport. However, a change in chemical potential may bring some of these into the \emph{conduction band}. Adding and removing ancilla amounts to varying the potential.

In this section we codify, in mathematical terms, developments (chiefly \emph{wrinkling}) from \cite{fhh} to \cite{fh20}, which allow the definition of $Q(X)$ to be made. At the level of physical rigor, one could skip wrinkling, and ask the reader simply to rely on \emph{separation of scale}. If one thinks of a ``lattice'' of say $10^{10} \times 10^{10}$ sites, perhaps on a torus $T^2$, and typical circuits having depth $\approx 10$, one can indeed form very long compositions before the scales mix and the forward light cone begins to wrap around the space. A mathematician will not be entirely satisfied; she may say ``if FDQC form a group, why should I be limited to taking compositions of length merely one billion?'' Wrinkling can address this question by establishing a range threshold $\epsilon_0 > 0$ on a given metric space $X$ below which any $\alpha \in \lbar{\operatorname{QCA}}(X)$ of $\range< \epsilon_0$ can be replaced by an equivalent $\alpha^\pr$ of arbitrarily smaller range. This ability to, at any point, pause and reduce range, prevents worrisome range build-up under composition.\footnote{We wonder, but do not know, if nonstandard analysis offers a simpler, although inevitably less concrete, resolution to the problem of range build-up under composition}.

Let us continue our informal survey of ideas with four key ingredients that make the definition of $Q(X)$ work: normality, triviality of $\alpha \otimes \alpha^{-1}$, abelianness, and wrinkling.

\begin{enumerate}
	\item \textbf{Normality of FDQC $\triangleleft$ QCA:} $H$ is normal in $G$ if we can push $h \in H$ through $g \in G$ (perhaps changing $h$ to $h^\pr$), $g h = h^\pr g$. A FDQC is a (locally finite) composition of gates. It turns out to be sufficient to show how to push a single gate $h$ through a QCA $\alpha$ to arrive at $\alpha h = h^\pr \alpha$, where $h^\pr$ will be a finite composition of gates. $\alpha$ following $h$ differs from $\alpha$ only within a certain \emph{light cone}, the $r$-neighborhood of $\supp(h)$. Any unitary on $\EuScript{N}_r(\supp(h))$ may be decomposed as a product of gates. So a single gate $h$ can be pushed through $\alpha$. The cost has been some expansion of support, so two originally disjoint gates $h_1$ and $h_2$ may push through to overlapping compositions $h_1^\pr$ and $h_2^\pr$. This appear at first to be a problem: which should be applied first? It actually does not matter: $h_1$ and $h_2$ commuted by disjointness, since $\alpha$ is a homeomorphism $h_1^\pr = \alpha h_1 \alpha^{-1}$ and $h_2^\pr = \alpha h_2 \alpha^{-1}$ still commute and can be applied in either order. Similarly with ancilla $\lbar{\operatorname{FDQC}} \triangleleft \lbar{\operatorname{QCA}}$. \qed

	\item \textbf{Triviality of $\alpha \otimes \alpha^{-1}$ (in the presence of ancilla):} First note that from Definition \ref{def:qca} $\operatorname{range}(\alpha^{-1}) = \operatorname{range}(\alpha)$, so the inverse of a QCA is a QCA. Indeed, since $\alpha$ is a homomorphism, $[\alpha(\Op),\Op^\pr] = 0 \iff [\Op,\alpha^{-1}(\Op^\pr)] = 0$ so the maximal expansion of support must be the same for $\alpha$ and $\alpha^{-1}$.
	
	Given any geometric system GS on $X$ (now perhaps noncompact) double it by now having two isomorphic (with a chosen isomorphism) Hilbert spaces mapping to each labeled site. In illustrations we think of the sites (old and new) in separate ``parallel'' copies of $X$. Actually in the illustration (Figure \ref{fig:ancilla}) of compositions, $X$ is compressed to a point but think of it, say, as a line going into the page. As a matter of algebra we may write:
	\begin{equation}\label{eq:a1}
		\alpha \otimes \alpha^{-1} = ((\alpha \otimes 1) \circ \mathrm{SWAP} \circ (\alpha^{-1} \otimes 1)) \circ \mathrm{SWAP}
	\end{equation}

	\begin{figure}[ht]
		\centering
		\begin{tikzpicture}[scale=1.4]
			\draw (-2,2) -- (-2,1);
			\node at (-2,0.75) {$\alpha$};
			\draw (-2,0.5) -- (-2,-2);
			\draw (-1,2) -- (-1,1);
			\node at (-1,0.75) {$\alpha^{-1}$};
			\draw (-1,0.5) -- (-1,-2);
			\node at (0,0) {$=$};
			\draw (1,2) -- (1,1.75);
			\node at (1,1.5) {$\alpha$};
			\draw (1,1.25) to[out=-90,in=90] (2,0.25) -- (2,-0.75) to[out=-90,in=90] (1,-1.75) -- (1,-2);
			\draw[white,line width = 1ex] (2,2) -- (2,1.25) to[out=-90,in=90] (1,0.25) -- (1,0);
			\draw (2,2) -- (2,1.25) to[out=-90,in=90] (1,0.25) -- (1,0);
			\draw[white,line width = 1ex] (1,-0.5) -- (1,-0.75) to[out=-90,in=90] (2,-1.75) -- (2,-2);
			\draw (1,-0.5) -- (1,-0.75) to[out=-90,in=90] (2,-1.75) -- (2,-2);
			\node at (1,-0.25) {$\alpha^{-1}$};
		\end{tikzpicture}
		\caption{}\label{fig:ancilla}
	\end{figure}

	But SWAP is a depth $=1$ quantum circuit, and so by normality:
	\begin{equation}\label{eq:a2}
		(\alpha \otimes 1) \circ \mathrm{SWAP} \circ (\alpha^{-1} \otimes 1) = \mathrm{SWAP}^\pr \circ (\alpha \otimes 1) \circ (\alpha^{-1} \otimes 1) = \mathrm{SWAP}^\pr
	\end{equation}

	Thus by (\ref{eq:a1}) and (\ref{eq:a2}),
	\begin{equation}
		\alpha \otimes \alpha^{-1} = \mathrm{SWAP}^\pr \circ \mathrm{SWAP}
	\end{equation}
	a FDQC.

	Running this argument backwards, starting with pair of ancilla at each site and applying the FDQC $\mathrm{SWAP}^\pr \circ \mathrm{SWAP}$ we obtain $\alpha \otimes \alpha^{-1}$, which we call ``trivial'' because it is produced by a FDQC from ancilla. FDQC are generally thought of as ``trivial'' since they cannot produce long-range entangled states such as topological ground states from an initially unentangled state.

	\item \textbf{$Q(X) \coloneqq \lbar{\mathrm{QCA}}(X) \slash \lbar{\mathrm{FDQC}}(X)$ is abelian:} The argument is strongly reminiscent of the Eckmann-Hilton theorem. Figure \ref{fig:abelian} shows (in the presence of ancilla) that composition multiplication is equivalent, modulo SWAP, to tensor product, from which abelianness follows.

	\begin{figure}[ht]
		\centering
		\begin{tikzpicture}[scale=1.1]
			\draw (-8,2) -- (-8,1);
			\node at (-8,0.75) {$\alpha$};
			\draw (-8,0.5) -- (-8,-0.5);
			\node at (-8,-0.75) {$\beta$};
			\draw (-8,-1) -- (-8,-2);
			\draw (-7,2) -- (-7,-2);
			\node at (-7,2.2) {copy};
			\node at (-7,2.6) {ancillary};
			
			\node at (-6,0) {$\equiv$};
			\draw (-5,2) -- (-5,1.5);
			\node at (-5,1.25) {$\alpha$};
			\draw (-5,1) to[out=-90,in=90] (-4,0) -- (-4,-0.75) to[out=-90,in=90] (-5,-1.75) -- (-5,-2);
			\draw[white,line width=1ex] (-4,2) -- (-4,1) to[out=-90,in=90] (-5,0);
			\draw (-4,2) -- (-4,1) to[out=-90,in=90] (-5,0);
			\node at (-5,-0.25) {$\beta$};
			\draw[white,line width=1ex] (-5,-0.5) -- (-5,-0.75) to[out=-90,in=90] (-4,-1.75) -- (-4,-2);
			\draw (-5,-0.5) -- (-5,-0.75) to[out=-90,in=90] (-4,-1.75) -- (-4,-2);
			
			\node at (-3,0) {$\equiv$};
			\draw (-2,2) -- (-2,0.25);
			\node at (-2,0) {$\alpha$};
			\draw (-2,0.-0.25) -- (-2,-2);
			\draw (-1,2) -- (-1,0.25);
			\node at (-1,0) {$\beta$};
			\draw (-1,-0.25) -- (-1,-2);
			
			\node at (0,0) {$\equiv$};
			\draw (1,2) -- (1,1.5) to[out=-90,in=90] (2,0.5);
			\draw (2,0) -- (2,-0.25) to[out=-90,in=90] (1,-1.25);
			\draw (1,-1.75) -- (1,-2);
			\node at (2,0.25) {$\beta$};
			\node at (1,-1.5) {$\alpha$};
			\draw[white,line width = 1ex] (2,2) -- (2,1.5) to[out=-90,in=90] (1,0.5) -- (1,0) -- (1,-0.25) to[out=-90,in=90] (2,-1.25) -- (2,-2);
			\draw (2,2) -- (2,1.5) to[out=-90,in=90] (1,0.5) -- (1,0) -- (1,-0.25) to[out=-90,in=90] (2,-1.25) -- (2,-2);
			
			\node at (3,0) {$\equiv$};
			\draw (4,2) -- (4,1);
			\node at (4,0.75) {$\beta$};
			\draw (4,0.5) -- (4,-0.5);
			\node at (4,-0.75) {$\alpha$};
			\draw (4,-1) -- (4,-2);
			\draw (5,2) -- (5,-2);
			\node at (5,2.2) {copy};
			\node at (5,2.6) {ancillary};
		\end{tikzpicture}
		\caption{}\label{fig:abelian}
	\end{figure}

	\item \textbf{Wrinkling:} The method described next has antecedents in the controlled $h$-cobordism theorems of Ferry \cite{ferry79} and Quinn \cites{quinn79,quinn82} and also in the $h$-principle \cite{EM02}. The context is a GS on a manifold $M$ (say a thickening of $X$) and a QCA $\alpha$ on that GS with a sufficiently small $\range r(\alpha) > 0$. The goal is to produce a deformation $\delta$ of $\alpha$ to $\alpha^\pr$ with a much smaller range. Below we will state a threshold theorem. However, strange as it sounds, is helpful to sketch the proof before giving a detailed statement: seeing the method demystifies the required quantifications.
\end{enumerate}

Begin by picturing some short $\range\alpha$ on $S^1$ (i.e.\ on a GS on $S^1$) and as explained in 2, additional copies of $\alpha \otimes \alpha^{-1}$ may be ``pulled from the vacuum,'' i.e.\ created by a small depth quantum circuit from ancilla. This is pictured in Figure \ref{fig:SDQC} for two such pairs, followed by small depth quantum circuits implementing various partial cancellations (back into the vacuum of ancillae) indicated by dotted lines. The final compression was a low depth quantum circuit to implement a homotopy of sites in $S^1$.

\begin{figure}[ht]
	\centering
	\begin{tikzpicture}[scale=1.15]
		\draw (-8,0) circle (1.5);
		\foreach \i in {0,...,11}
		{
			\draw[fill=black, rotate around = {30*\i:(-8,0)}] (-8,1.5) circle (0.25ex);
		}
		\node at (-8,1) {$\alpha$};
		\draw[->] (-7.9,1.15) -- (-7.75,1.4);
		\draw[->] (-6.3,0) -- (-5.2,0);
		\node at (-5.75,0.75) {\footnotesize{pull copies}};
		\node at (-5.75,0.5) {\footnotesize{from}};
		\node at (-5.75,0.2) {\footnotesize{vacuum}};
		
		\draw (-2.5,0) circle (2.5);
		\draw (-2.5,0) circle (2.25);
		\draw (-2.5,0) circle (2);
		\draw (-2.5,0) circle (1.75);
		\draw (-2.5,0) circle (1.5);
		\path[fill=white] (-2.5,1.5) circle (1ex);
		\node at (-2.5,1.5) {$\alpha$};
		\path[fill=white] (-2.5,2) circle (1ex);
		\node at (-2.5,2) {$\alpha$};
		\path[fill=white] (-2.5,2.5) circle (1ex);
		\node at (-2.5,2.5) {$\alpha$};
		\path[fill=white] (-2.1,1.7) circle (1.2ex);
		\node at (-2,1.7) {$\alpha^{-1}$};
		\path[fill=white] (-2.1,2.2) circle (1.2ex);
		\node at (-2,2.2) {$\alpha^{-1}$};
		\draw[->] (-0.3,-1.5) to[out=-45,in=45] (-0.3,-4);
		\node at (1.2,-2.3) {\footnotesize{partially cancel}};
		\node at (1.2,-2.6) {\footnotesize{copies back}};
		\node at (1.2,-2.9) {\footnotesize{into vacuum}};
		
		\node at (-2.5,-7.9) {$\dots$};
		\draw[xshift=-7cm,yshift=-5.5cm] (5,-2.45) to[out=5,in=-90] (7,0) to[out=90,in=0] (4.5,2.5) to[out=180,in=90] (2,0) to[out=-90,in=180] (4.05,-2.45) to[out=10,in=180] (4.2,-2.35) to[out=180,in=-10] (4.05,-2.25) to[out=180,in=-90] (2.25,0) to[out=90,in=197] (4.06,2.24) to[out=27,in=197] (4.23,2.21) to[out=197,in=7] (4.15,2.05) to[out=197,in=90] (2.5,0) to[out=-90,in=170] (4.07,-2.04) to[out=0,in=170] (4.23,-1.98) to[out=170,in=-20] (4.13,-1.84) to[out=170,in=-90] (2.75,0) to[out=90,in=205] (4.05,1.8) to[out=15,in=205] (4.23,1.77) to[out=205,in=5] (4.15,1.6) to[out=205,in=90] (3,0) to[out=-90,in=180] (4.5,-1.7) to[out=0,in=-90] (6,0) to[out=90,in=-25] (4.85,1.6) to[out=175,in=-25] (4.77,1.77) to[out=-25,in=165] (4.95,1.8) to[out=-25,in=90] (6.25,0) to[out=-90,in=10] (4.87,-1.84) to[out=200,in=10] (4.77,-1.98) to[out=10,in=180] (4.93,-2.04) to[out=10,in=-90] (6.5,0) to[out=90,in=-17] (4.85,2.05) to[out=173,in=-17] (4.77,2.21) to[out=-17,in=153] (4.94,2.24) to[out=-17,in=90] (6.75,0) to[out=-90,in=5] (4.95,-2.25) to[out=190,in=0] (4.8,-2.35) to[out=0,in=170] (4.95,-2.45) to[out=0,in=185] (5,-2.45);
		\node at (-2.5,-3.35) {$\dots$};
		\node at (-2.5,-7.45) {$\dots$};
		\node at (-2.5,-3.8) {$\dots$};
		\draw[->] (-5.2,-5) -- (-6.3,-5);
		\node at (-5.75,-4.8) {\footnotesize{compress}};
		
		\foreach \j in {0,...,7}
		{
			\draw[rotate around = {45*\j:(-8,-5.5)}] (-10.2,-5.5) to[out=0,in=-45] (-9.56,-3.94);
			\draw[rotate around = {45*\j:(-8,-5.5)}, line width=1.5pt, line cap=round, dash pattern=on 0pt off 2.78\pgflinewidth] (-10.4,-5.5) -- (-10.2,-5.5) to[out=0,in=-45] (-9.56,-3.94) -- (-9.7,-3.8);
		}

		\node at (-12,0) {\hspace{1em}};
	\end{tikzpicture}
	\caption{The result is an $\alpha^\pr$ of much smaller range.}\label{fig:SDQC}
\end{figure}

When $\operatorname{dim}(M) > 1$ a similar approach is possible by wrinkling, inductively over handles of increasing index. Some ground is lost while running the partial cancellation circuits, so the threshold epsilon $\epsilon_0$ will be exponentially small in $\operatorname{dim}(M)$, it also has a modest dependence on sectional curvatures $K$, and if there is a boundary to the eigenvalues, $e$, of its second fundamental form which control the local combinatorics of the handle decomposition. If $M$ is rescaled so that $-1 \leq K,e \leq +1$ and inj.\ rad.$(M) \geq 1$, then $\epsilon_0$ depends only on dimension. Figures \ref{fig:wrinkling1} and \ref{fig:wrinkling2} are meant to suggest how wrinkling proceeds in dimension two:

\begin{figure}[ht]
	\centering
	\begin{tikzpicture}[scale=0.97]
		\draw (-7,0) circle (1);
		\node at (-7,0) {\footnotesize{0-handle}};
		\draw (-3,0) circle (1);
		\node at (-3,0) {\footnotesize{0-handle}};
		\draw (-6.02,0.15) -- (-3.98,0.15);
		\draw (-6.02,-0.15) -- (-3.98,-0.15);
		\node at (-5,0) {\footnotesize{1-handle}};
		
		\draw (-7.15,0.99) -- (-7.15,2);
		\draw (-6.85,0.99) -- (-6.85,2);
		\draw (-7.15,-0.99) -- (-7.15,-2);
		\draw (-6.85,-0.99) -- (-6.85,-2);
		\draw (-9,0.15) -- (-7.99,0.15);
		\draw (-9,-0.15) -- (-7.99,-0.15);
	
		\draw (-2.85,0.99) -- (-2.85,2);
		\draw (-3.15,0.99) -- (-3.15,2);
		\draw (-2.85,-0.99) -- (-2.85,-2);
		\draw (-3.15,-0.99) -- (-3.15,-2);
		\draw (-1,0.15) -- (-2.01,0.15);
		\draw (-1,-0.15) -- (-2.01,-0.15);

		\draw (2,0) circle (1);
		\draw (6,0) circle (1);
		\draw (2.98,0.2) -- (5.02,0.2);
		\draw (2.98,-0.2) -- (5.02,-0.2);
		
		\draw (9-7.15,0.99) -- (9-7.15,2);
		\draw (9-6.85,0.99) -- (9-6.85,2);
		\draw (9-7.15,-0.99) -- (9-7.15,-2);
		\draw (9-6.85,-0.99) -- (9-6.85,-2);
		\draw (0,0.15) -- (9-7.99,0.15);
		\draw (0,-0.15) -- (9-7.99,-0.15);
	
		\draw (9-2.85,0.99) -- (9-2.85,2);
		\draw (9-3.15,0.99) -- (9-3.15,2);
		\draw (9-2.85,-0.99) -- (9-2.85,-2);
		\draw (9-3.15,-0.99) -- (9-3.15,-2);
		\draw (8,0.15) -- (9-2.01,0.15);
		\draw (8,-0.15) -- (9-2.01,-0.15);
		
		\draw (2.98,0) -- (3.7,0) arc (180:-180:0.3 and 0.15) -- (3.7,0.7);
		\draw[dashed] (3.7,0.7) -- (3.7,1) arc (180:-180:0.3 and 0.15);
		\draw[dashed] (4.3,0.7) -- (4.3,1);
		\draw (5.02,0) -- (4.3,0) -- (4.3,0.7);
		\draw (3,0.7) -- (4.9,0.7) to[out=80,in=-130] (5.1,1.2) to[out=120,in=-100] (5,1.7) -- (3.1,1.7) to[out=-100,in=60] (2.9,1.2) to[out=-50,in=80] (3,0.7);
		\draw (2.9,1.2) -- (5.1,1.2) -- (5.4,1.5);
		\draw (5,1.7) -- (5.3,2) -- (3.4,2) -- (3.1,1.7);
		\draw (4.9,0.7) -- (5.2,1) to[out=80,in=-130] (5.4,1.5) to[out=120,in=-100] (5.3,2);
		\draw (5.05,0.85) to[out=80,in=-130] (5.25,1.35) to[out=120,in=-100] (5.15,1.85);
		\draw (4.55,0.7) to[out=80,in=-130] (4.75,1.2) to[out=120,in=-100] (4.65,1.7);
		\draw (3.45,1.7) to[out=-100,in=60] (3.25,1.2) to[out=-50,in=80] (3.35,0.7);
		
		\node at (4,-1.25) {\footnotesize{wrinkled 1-handle}};
		\draw[->] (4,-1) -- (4,-0.3);
	\end{tikzpicture}
	\caption{}\label{fig:wrinkling1}
\end{figure}

\begin{figure}[ht]
	\centering
	\begin{tikzpicture}[scale=0.97]
		\draw (-7.5,0) circle (1);
		\draw (-2.5,0) circle (1);
		\draw[decorate, decoration={snake, segment length=20, amplitude=7}] (-6.52,0.15) -- (-3.48,0.15);
		\draw[decorate, decoration={snake, segment length=20, amplitude=7}] (-6.52,-0.15) -- (-3.48,-0.15);
		
		\draw (-7.65,0.99) -- (-7.65,1.5);
		\draw (-7.35,0.99) -- (-7.35,1.5);
		\draw (-9,0.15) -- (-8.49,0.15);
		\draw (-9,-0.15) -- (-8.49,-0.15);
		\draw (-2.35,0.99) -- (-2.35,1.5);
		\draw (-2.65,0.99) -- (-2.65,1.5);
		\draw (-1,0.15) -- (-1.51,0.15);
		\draw (-1,-0.15) -- (-1.51,-0.15);
		
		\draw (-7.5,-4) circle (1);
		\draw[decorate, decoration={snake, segment length=20, amplitude=7}] (-7.65,-0.99) -- (-7.65,-3.01);
		\draw[decorate, decoration={snake, segment length=20, amplitude=7}] (-7.35,-0.99) -- (-7.35,-3.01);
		\draw (-7.65,-4.99) -- (-7.65,-5.5);
		\draw (-7.35,-4.99) -- (-7.35,-5.5);
		\draw (-9,-3.85) -- (-8.49,-3.85);
		\draw (-9,-4.15) -- (-8.49,-4.15);
		
		\draw (-2.5,-4) circle (1);
		\draw[decorate, decoration={snake, segment length=20, amplitude=7}] (-2.65,-0.99) -- (-2.65,-3.01);
		\draw[decorate, decoration={snake, segment length=20, amplitude=7}] (-2.35,-0.99) -- (-2.35,-3.01);
		\draw[decorate, decoration={snake, segment length=20, amplitude=7}] (-6.52,-3.85) -- (-3.48,-3.85);
		\draw[decorate, decoration={snake, segment length=20, amplitude=7}] (-6.52,-4.15) -- (-3.48,-4.15);
		\draw (-1.51,-3.85) -- (-1,-3.85);
		\draw (-1.51,-4.15) -- (-1,-4.15);
		\draw (-2.65,-4.99) -- (-2.65,-5.5);
		\draw (-2.35,-4.99) -- (-2.35,-5.5);
		
		\draw (1.5,0) circle (1);
		\draw (6.5,0) circle (1);
		\draw (1.5,-4) circle (1);
		\draw (6.5,-4) circle (1);
		
		\draw (9-7.65,0.99) -- (9-7.65,1.5);
		\draw (9-7.35,0.99) -- (9-7.35,1.5);
		\draw (0,0.15) -- (9-8.49,0.15);
		\draw (0,-0.15) -- (9-8.49,-0.15);
		\draw (9-2.35,0.99) -- (9-2.35,1.5);
		\draw (9-2.65,0.99) -- (9-2.65,1.5);
		\draw (8,0.15) -- (9-1.51,0.15);
		\draw (8,-0.15) -- (9-1.51,-0.15);
		
		\draw (9-7.65,-4.99) -- (9-7.65,-5.5);
		\draw (9-7.35,-4.99) -- (9-7.35,-5.5);
		\draw (0,-3.85) -- (9-8.49,-3.85);
		\draw (0,-4.15) -- (9-8.49,-4.15);
		\draw (9-1.51,-3.85) -- (8,-3.85);
		\draw (9-1.51,-4.15) -- (8,-4.15);
		\draw (9-2.65,-4.99) -- (9-2.65,-5.5);
		\draw (9-2.35,-4.99) -- (9-2.35,-5.5);
		
		\draw (4,0.2) ellipse (1.25 and 0.6);
		\draw (2.75,0.2) to[out=-90,in=180] (4,-0.8) to[out=0,in=-90] (5.25,0.2);
		\draw (4,-1.3) arc (-90:51:1.25 and 0.4);
		\draw (4,-1.3) arc (-90:-231:1.25 and 0.4);
		\draw (2.75,-0.9) to[out=-90,in=180] (4,-1.7) to[out=0,in=-90] (5.25,-0.9);
		\draw (4,-2.2) arc (-90:46:1.25 and 0.4);
		\draw (4,-2.2) arc (-90:-226:1.25 and 0.4);
		
		\draw (2.49,0) to[out=15,in=-110] (2.75,0.15);
		\draw (5.21,-0.1) to[out=-45,in=180] (5.5,0);
		\draw[decorate, decoration={snake, segment length=20, amplitude=7}] (2.49,-4) -- (5.51,-4);
		\draw[decorate, decoration={snake, segment length=20, amplitude=7}] (1.5,-1) -- (1.5,-3);
		\draw[decorate, decoration={snake, segment length=20, amplitude=7}] (6.5,-1) -- (6.5,-3);
		
		\node at (4,0.1) {$\alpha$};
		\node at (4,-0.55) {$\alpha^{-1}$};
		\node at (4,-1.05) {$\alpha$};
		\node at (4,-1.5) {$\alpha^{-1}$};
		\node at (4,-1.95) {$\alpha$};
		
		\node at (-5,-6) {wrinkled 1-handles};
		\node at (4,-5.8) {bellows-like};
		\node at (4,-6.2) {wrinkled 2-handle};
	\end{tikzpicture}
	\caption{}\label{fig:wrinkling2}
\end{figure}

This completes our informal tour of key geometric constructions from \cite{fh20} and \cite{fhh}. A comprehensive wrinkling theorem (slightly extending \cite{fhh}) can be stated as soon as we extend the concept of range of a QCA to range of a deformation of QCAs.

Recall that a deformation $\delta$ of a QCA $\alpha$ is a sequence of QCAs $\alpha \coloneqq \alpha_1, \dots, \alpha_n$ where at each step a \emph{gate} is applied or ancillae are added/removed. (In \cite{fhh} also permitted a homotopy\footnote{We actually said ``isotopy'' in that paper.} of a site map $s$, but this is redundant as it can be simulated by adding ancillae applying SWAPs and removing ancillae.) We define $\range(\delta) = \max_i \operatorname{range}(\alpha_i)$. But the following example highlights a small risk. Consider a circle with a GS with three sites located at $\theta = 0,\frac{2\pi}{3}$, and $\frac{4\pi}{3}$, and say a qubit at each site (with specified isomorphisms to $\C^2$). Any permutation $p$ of sites induces an endomorphism $\alpha(p) \in \End(\C^2)^{\otimes 3}$. Let $p_0 = (0,\frac{2\pi}{3},\frac{4\pi}{3})$, $p_1 = (0,\frac{4\pi}{3},\frac{2\pi}{3})$ and $p_2 = (\frac{2\pi}{3},\frac{4\pi}{3})$, the SWAP gate. We might consider all three permutations to have ``small range'' since sites are moved less than the injectivity radius of $S^1$. But the deformation $\alpha_0 = \alpha(p_0)$, $\alpha_1 = \mathrm{SWAP} \circ \alpha_0 = \alpha(p_1)$ takes counterclockwise to clockwise rotation. This is a problem, deformation classes of QCA should, according to \cite{gross12}, distinguish the direction of flow. Motivated by the example, we make a blanket requirement that all gates $g$ occuring in a deformation $\delta$ must have $\operatorname{diam}(\supp(g)) \leq \operatorname{inj.\ rad.}(X,d)/100$. This will prevent a quantum circuit taking a ``small'' range QCA to another small range QCA of a different topological character; clockwise vs.\ counterclockwise rotation in the preceeding example. (For a simplicial complex inj.\ radi.\ is the supremum of radii $r$ so that all metric balls of radius $r$ in $X$ are collapsable.) This completes our sketch of the proof techniques for Theorem \ref{thm:wrinkling} below.

Let us make a comment going beyond our quoted references: define a 2-complex, $\mathbb{QCA}(X)$, to model QCA$(X)$. Build $\mathbb{QCA}(X)$ as a 2-complex, the 0-cells should be admissible QCA on $X$ and 1-cells admissible deformations. Let us propose that 2-cells be an \emph{admissible} 2-deformation $\Delta$, a deformation of deformations. Since we may implant all the ancillae we will ever need at the start (and remove them at the end), we only need to define a 2-deformation over a sequence of gates, where $\Delta$ may be defined to be any portion of a square grid of gates where each box commutes:

\begin{figure}[ht]
	\centering
	\begin{tikzpicture}[scale=1.8]
		\node at (0,0.5) {$\alpha_i^1$};
		\draw[->] (0,0.25) -- (0,-0.25);
		\node at (0,-0.5) {$\alpha_i^2$};
		\draw[->] (0.25,0.5) -- (0.75,0.5);
		\node at (1.25,0.5) {$\alpha_{i+1}^1$};
		\draw[->] (1.75,0.5) -- (2.25,0.5);
		\node at (2.75,0.5) {$\dots$};
		\draw[->] (1.25,0.25) -- (1.25,-0.25);
		\node at (1.25,-0.5) {$\alpha_{i+1}^2$};
		\draw[->] (1.75,-0.5) -- (2.25,-0.5);
		\node at (2.75,-0.5) {$\dots$};
		\draw[->] (0.25,-0.5) -- (0.75,-0.5);
		\node at (0,-1) {$\vdots$};
		\node at (1.25,-1) {$\vdots$};
		\node at (-1.25,0.5) {$\alpha_{i-1}^1$};
		\draw[->] (-0.85,0.5) -- (-0.25,0.5);
		\node at (-1.25,-0.5) {$\alpha_{i-1}^2$};
		\draw[->] (-0.85,-0.5) -- (-0.25,-0.5);
		\draw[->] (-1.25,0.25) -- (-1.25,-0.25);
		\node at (-2.5,-0.5) {$\dots$};
		\draw[->] (-2.2,-0.5) -- (-1.75,-0.5);
		\node at (-2.5,0.5) {$\dots$};
		\draw[->] (-2.2,0.5) -- (-1.75,0.5);
		\node at (-1.25,-1) {$\vdots$};
		\draw[->] (-0.75,0.25) to[out=-10,in=90] (-0.4,0) to[out=-90,in=10] (-0.75,-0.25);
		\draw[->] (0.5,0.25) to[out=-10,in=90] (0.85,0) to[out=-90,in=10] (0.5,-0.25);
	\end{tikzpicture}
	\caption{Arrows are gates and squares commute.}
\end{figure}

We define the range of $\Delta$, $r(\Delta) = \max_{\text{grid}}(\operatorname{range}(\alpha_p^q))$, and $\Delta$ is admissible if $r(\Delta) < \epsilon_0$. Thus, $\pi_0(\mathbb{QCA}(X)) \cong Q(X)$, and $\pi_1(\mathbb{QCA}(X))$ is some other abelian group (since composition makes $\mathbb{QCA}(X)$ into an $H$-space), possible the other grading of some $\Z_2$-periodic spectrum.

The 2-deformation $\Delta$ in the proof of Theorem \ref{thm:covariant}, below, will be of a very simple type: a triangular array, $q \geq p$, and all horizontal arrows identity maps.

\begin{thm}[Wrinkling Threshold \cite{fhh}]\label{thm:wrinkling}
	Let $X$ be a locally finite simplicial complex. The largest scale we consider is the majorant function $i: X \ra \R^+$, $i \coloneqq \mathrm{inj.\ rad.}/2$. We now use variously subscripted $\epsilon$ to denote smaller majorant functions to $\R^+$. Let $\epsilon \leq i$ be any majorant (in our aplications $\epsilon = i$). Then there is an $\epsilon_0$ so that any QCA on $X$ of $\range\leq \epsilon_0$ can be deformed $(\delta)$ to arbitrarily small range so that every QCA in the deformation has $\range\leq \epsilon$ (and substituting $\epsilon_0$ for $\epsilon$, there is an $\epsilon_{-1}$ so that any QCA of $\range\leq \epsilon_{-1}$ can be deformed to arbitrarily small range with the QCA in that deformation having $\range\leq \epsilon_0$).
	
	\emph{Addendum}: If $\epsilon = i$, $\epsilon_0$ and $\epsilon_{-1}$ depend only on the local combinatorics of $X$. Because the language of handlebodies is very convenient to construct wrinkling procedures, if $X$ is not a manifold we would generally thicken it to a Riemannian manifold $M$ to produce the claimed wrinklings (i.e.\ shrinking of ranges). In the case $X=M$, if $M$ is rescaled so that $-1 \leq K,e \leq 1$ for all sectional curvatures and eigenvalues of the II-fundamental form (when $\partial M \neq \varnothing$) and also inj radius $M \geq 1$, then $\epsilon_0$ depends only on dim$(M)$; $\epsilon_0 = c^{-d}$ for some constant $c>1$. \qed
\end{thm}

Recall we have defined geometric systems GS $(X,h)$, thickened $X$ to a Riemannian manifold $M$ (to conveniently use handle theory), defined a QCA $\alpha$ on a GS on $M$, defined $\range(\alpha)$, defined a deformation $\delta$ from $\alpha_1$ to $\alpha_n$, $\range(\delta)$, and the constant or majorant function $\epsilon_o(M) > 0$.

Using the wrinkling threshold theorem, one may verify that $Q$ depends only on the proper homotopy type of $X$ and indeed is a covariant proper homotopy functor. Ultimately the choice of path metric $d$ drops out. Of course, when $X$ is compact, $Q$ is a homotopy functor.

\begin{thm}\label{thm:covariant}
	$Q$ is a covariant function from the category of locally finite simplicial complexes and continuous, proper maps to the category of abelian groups.
\end{thm}

\begin{proof}[Proof sketch, focusing on the compact case]
	Wrinkling is a powerful method for shrinking size scales, but like so many ``$(\U ,\mathrm{V})$ arguments''\footnote{Physicists be warned: U and V here refer to nested open sets, not the ultraviolet.} ranges may expand a bit before shrinking. Figure \ref{fig:functor} illustrates how this makes the equivalence of admissible QCAs on $X$ subtle. The answer may depend on the existence of an inobvious admissible deformation, the dotted line in Figure \ref{fig:functor}.
\end{proof}

\begin{figure}[ht]
	\centering
	\begin{tikzpicture}[scale=1.4]
		\draw (-3,1.5) -- (2,1.5);
		\draw (-3,0) -- (2,0);
		\draw (-3,-1.3) -- (2,-1.3);
		\node at (-3.75,1.5) {$\frac{\text{inj.\ rad.}(M)}{2}$};
		\node at (-3.75,0) {$\epsilon_0$};
		\node at (-3.75,-1.3) {$\epsilon_{-1}$};
		\draw[->] (-4.5,1.5) -- (-4.5,-1.5);
		\node[rotate=90] at (-4.75,0) {smaller $\epsilon$};
		\node[rotate=90] at (-5.05,0) {range of $\alpha$ and $\delta$};
		
		\draw[<->] (-2.5,-2.5) .. controls (-2.5,0.5) and (-1.4,0) .. (-1.5,-1.5) .. controls (-1.4,1.5) and (0.4,1.5) .. (0.5,-1.5) .. controls (0.6,0) and (1.5,0.5) .. (1.5,-2.5);
		
		\draw[fill=black] (-2.49,-2) circle (0.2ex);
		\draw[fill=black] (-1.5,-1.5) circle (0.2ex);
		\draw[fill=black] (0.5,-1.5) circle (0.2ex);
		\draw[fill=black] (1.49,-2) circle (0.2ex);
		
		\node at (-2.2,-2.1) {$\alpha_1^\pr$};
		\node at (-1.2,-1.5) {$\alpha_1$};
		\node at (0.8,-1.5) {$\alpha_2$};
		\node at (1.8,-2) {$\alpha_2^\pr$};
		
		\draw[dashed] (-2.49,-2) to[out=30,in=225] (-0.7,-1.5) to[out=45,in=180] (-0.1,-0.7) to[out=0,in=180] (0.6,-2.3) to[out=0,in=225] (1.49,-2);

		\node at (4,0) {\hspace{1em}};
	\end{tikzpicture}
	\caption{}\label{fig:functor}
\end{figure}

The three solid arcs in \ref{fig:functor} represent deformations with range their largest vertical value. They imply $\alpha_1 \equiv \alpha_1^\pr$ and $\alpha_2 \equiv \alpha_2^\pr$ but by themselves do not imply $\alpha_1 \equiv \alpha_2$. If dashed arc is also a deformation then all four QCAs are equivalent.

Suppose $\alpha_1$, $\alpha_1^\pr$, $\alpha_2$, and $\alpha_2^\pr$ from \ref{fig:functor} are QCAs on $X$ and we have homotopy equivalences $X \underset{g}{\overset{f}{\rightleftharpoons}} Y$. Let us suppose we are aware of the three solid deformations on $X$ but unaware of the dotted one. So at this point we see two possibly distinct elements, $[\alpha_1],[\alpha_2] \in Q(X)$. Now let us map to $Q(Y)$. Because $f$ may stretch distance (and/or $Y$ may have a smaller $\epsilon_0(Y)$) $f(\alpha_1)$ and $f(\alpha_2)$ might fail to be admissible QCAs on $Y$. But for very small $\range\alpha_1^\pr$ and $\alpha_2^\pr$ (respectively in the equivalence classes of $\alpha_1$ and $\alpha_2$ on $X$) we will have $f(\alpha_1^\pr)$ and $f(\alpha_2^\pr)$ admissible QCA on $Y$. Now suppose that on $Y$ we are aware of a ``dotted'' admissible deformation, call it $\delta^\pr$, from $f(\alpha_1^\pr)$ to $f(\alpha_2^\pr)$. Because $\delta^\pr$ is admissible (in $Y$) we may wrinkle it as extravagantly as we like to produce an extremely small range deformation $\delta^{\pr\pr}$ from $(f(\alpha_1^\pr))^\pr$ to $(f(\alpha_2^\pr))^\pr$. If one likes, $\delta^{\pr\pr}$ is one of four ``boundary arcs'' of a 2-deformation $\Delta$ carrying $\delta^\pr$ to $\delta^{\pr\pr}$ (see Figure \ref{fig:boundary_arc}). $\Delta$ may have dangerously\footnote{Dangerous in the sense of being too large to be reduced by the Wrinkling Theorem.} large range (the maximum of its QCAs) near inj.\ rad$(Y)/2$, but this is irrelevant. The three of the four boundary arcs of $\Delta$, $\delta_1 \coloneqq$ the arc joining $f(\alpha_1^\pr)$ to $(f(\alpha_1^\pr))^\pr$, $\delta_2 \coloneqq$ the arc joining $f(\alpha_2^\pr)$ to $(f(\alpha_2^\pr))^\pr$, and $\delta^{\pr\pr}$ all have range less than a quantity $\epsilon_{-1}(Y)$ which, by the last sentence of Theorem \ref{thm:wrinkling}, we can choose as small as we like by making $\alpha_1^\pr$ and $\alpha_2^\pr$ to make arbitrarily small. Thus we can ar$\range\range(g(\delta_1 \cup \delta^{\pr\pr} \cup \delta_2)) < \epsilon_0(X)$, verifying that $[\alpha_1] = [\alpha_2]$ in $Q(X)$. \qed

\begin{figure}[ht]
	\centering
	\begin{tikzpicture}[scale=1.4]
		\draw (-3,1.5) -- (2,1.5);
		\draw (-3,0) -- (2,0);
		\draw (-3,-1.5) -- (2,-1.5);
		\node at (-3.65,1.5) {$\frac{\text{inj.\ rad.}(Y)}{2}$};
		\node at (-3.55,0) {$\epsilon_0(Y)$};
		
		\node at (-4.5,-1.2) {\footnotesize{$g$ sends QCA on $Y$ ranges}};
		\node at (-4.5,-1.5) {\footnotesize{below this line to QCA on}};
		\node at (-4.7,-1.8) {\footnotesize{$X$ with ranges $<\epsilon_0(X)$}};
		
		\draw (-2.8,-3.5) .. controls (-2.8,-1.5) and (-2.2,-1.8) .. (-2.2,-2.3) to[out=30,in=180] (-0.6,-0.7) to[out=0,in=180] (0.4,-2.3) to[out=0,in=180] (0.9,-1.8) .. controls (1.2,-1.8) and (1.4,-2) .. (1.4,-3.5) to[out=170,in=0] (-0.7,-3.1) to[out=180,in=30] (-1.9,-3.3) to[out=210,in=10] (-2.8,-3.5);
		
		\draw[fill=black] (-2.2,-2.3) circle (0.2ex);
		\draw[fill=black] (-2.8,-3.5) circle (0.2ex);
		\draw[fill=black] (1.4,-3.5) circle (0.2ex);
		\draw[fill=black] (0.6,-2.15) circle (0.2ex);
		
		\node at (-2.8,-3.8) {$f(\alpha_1^\pr)^\pr$};
		\node at (2,-3.6) {$f(\alpha_2^\pr)^\pr$};
		\node at (-0.5,-2.9) {$\delta^{\pr\pr}$};
		\node at (-3,-2.7) {$\delta_1$};
		\node at (1.6,-2.7) {$\delta_2$};
		\node at (-2.2,-2.55) {$f(\alpha_1^\pr)$};
		\node at (2,-2) {$f(\alpha_2^\pr)$};
		\draw[->] (1.5,-2) -- (0.7,-2.15);
		
		\draw[->] (-0.62,-0.7) to[out=90,in=-86] (-0.63,0.5);
		\draw[->] (-0.63,0.5) .. controls (-0.65,1.5) and (-0.85,1.5) .. (-0.9,0.8);
	    \draw[->] (-0.9,0.8) to[out=-95,in=87] (-1,-0.5);
	    \draw[color=white,line width=1.5mm] (-1,-0.5) -- (-1.03,-1);
	    \draw[->] (-1,-0.5) to[out=-93,in=89] (-1.05,-2);
	    \draw[->] (-1.05,-2) to[out=-91,in=90] (-1.07,-3.1);
	    
	    \node at (0.75,0.6) {\footnotesize{part of $\Delta$, the 2-}};
	    \node at (1.07,0.3) {\footnotesize{deformation from $\delta^\pr$ to $\delta^{\pr\pr}$}};

		\node at (4.5,0) {\hspace{1em}};
	\end{tikzpicture}
	\caption{}\label{fig:boundary_arc}
\end{figure}

Warning to physicists: While our definitions respect the concept of \emph{finite information density} FID, $\operatorname{dim}(H_i) < \infty$, the site map $s$ is finite-to-one, and sites are locally finite. We have \emph{not} enforced bounded information density BID. For example, under our definition $Q(\R^3) \cong Q(H^3)$, even though any map from hyperbolic to Euclidean space involves an exponentially increasing compression of volume, hence increasing compression of information, as one approaches infinity. Similarly for a punctured sphere $Q(S^3_-) \cong Q(\R^3) \cong Q(H^3)$, since all these spaces are proper homotopy equivalent (in fact homeomorphic), and the groups are nontrivial by \cite{fhh20}. Whereas under a BID hypothesis all three might well be distinct. In one lower dimension it was proven \cite{fh20} that $Q(S^2) = 0$ and those methods also imply\footnote{Private communication from Jeongwan Haah.} that $Q(S^2_-) \cong Q(\R^2) \cong Q(H^2)$ are trivial.

Particularly in the case where $X$ is a smooth 4-manifold $M$, it would be of interest to define a more refined functor $\wt{Q}(M)$ which would have the potential to depend on the smooth structure. To do this, modify the definition of \emph{geometric system} to place the (finite) Hilbert space dof on piecewise smooth oriented top cells within a piecewise smooth cell structure on $M$. Adding (removing) an ancilla would correponding to subdividing (fusing) top cells. Because of the orientation condition on top cells, a diffeomorphism, but not a general map, induces a homomorphism $\wt{Q}(M) \ra \wt{Q}(M^\pr)$. But notice that cells structures can be pulled back by covering maps and codimension zero immersions, so the partial order explored in the main paper may be relevant to such a $\wt{Q}(M)$. Also, pleasantly, $\wt{Q}$ would restore \emph{cells} to the concept of a \emph{quantum cellular automata}.

Returning to $Q(X)$, we pose the question: Is $Q(X)$ classified by a spectrum, and if so, which one? In this appendix we have taken a \emph{sharp} definition of range with an abrupt cut-off. Better motivated physically would be to consider QCAs with exponentially decaying tails (corresponding to partially screened electron-electron interactions). Recently, the results of \cite{gross12} have been recovered in one such context\footnote{The permissible tails in this work are actually algebraic.} \cite{rww20}. But curiously, even more recent unpublished work of Kitaev suggests that the 1D classification of QCA with tails may depend on the exact choice of operator norms under which they are measured. So perhaps several related spectra should be sought.

To build a little intuition for $Q(X)$, we suggest that it be thought of as a kind of quantized first homology, $H_1(X)$. For example, suppose we look at a QCA consisting of discrete qubit dof acted on only by \emph{small permutations} $p$, i.e.\ if $p(x) = y$ then $\operatorname{dist}_X(x,y) < \epsilon << \text{inj.\ rad.}(X)$. Permutation whose cycles all have diameter $O(\epsilon) << \text{inj.\ rad.}(X)$, are equivalent to ancilla and should be divided out. It is a pleasant exercise to verify: $\{\text{small permutations}\} \slash \{$permutations with all-small cycles$\} \cong H_1^{\text{loc.\ fin.}}(X;\Z)$. Allowing qunits complicates the coefficients slightly, see \cite{gross12}, but a similar statement holds with qunit dof.

The body of the paper plays on this theme. The forward (functional map) $Q(X) \ra Q(Y)$, given $f: X \ra Y$, is apparent (although some epsilons must be thought through, as we saw in our sketch of Theorem \ref{thm:covariant}) and was exploited in \cite{fhh20}. More interesting are the \emph{wrong direction} maps associated with coverings and, even more subtly, codimension-0 immersions. To continue the analogy with homology, these should be thought of as \emph{transfer} maps. A goal of this paper is to develop tools, such as this transfer, to put interesting QCAs on a variety of manifolds, and, perhaps, eventually use these to study or classify the manifolds themselves.

To link QCA to interesting, topological states of matter, we may consider that each local Hilbert space $H_i$ is provided with a distinguished vacuum vector\footnote{Something we did not wish to do at the outset, a choice that drove us away from Hilbert spaces and towards their endormophism algebras.} $\lvert \uparrow \rangle$. Then any QCA $\alpha$ provides a state $\psi = \alpha(\otimes_i \lvert \uparrow \rangle)$. It is proven in \cite{fedhh}, modulo accepted physical properties of central charge and edge currents, that there are nontrivial topological ground states $\psi_{\text{top}}$ on a 3-torus $T^3$ that can be disentangled via $\alpha^{-1}$, i.e.\ $\psi_{\text{top}} = \alpha \otimes_i \lvert \uparrow \rangle$. In contrast, as a matter of definition, a finite depth quantum circuit cannot disentangle a topological ground state $\psi_{\text{top}} \neq \operatorname{FDQC}(\otimes_i \lvert \uparrow \rangle)$. Thus, $Q(M)$, essentially $\operatorname{QCA} \slash \operatorname{FDQC}(M)$, may produce topological ground states on a manifold $M$. In \cite{fedhh} these topological ground states, torsion elements of $Q(T^3)$, are detected via algebraic $K$-theory. This should not be surprising since $K_1$ is all about whether automorphisms of rings and algebras can be decomposed in terms of the simplest representatives, analogous to attempting to write a QCA as a FDQC.

In this paper, the transfer map is studied with an eye towards understanding which topological states can live on which manifolds. A trivial example, well known in physics, is that topological phases with point-like Fermionic excitations (which by the spin-statistics theorem, can be regarded as a point-like excitation with a $(-1)$ twist factor under $\pi_1(\operatorname{SO}(d))$) must reside on manifolds with a spin structure.

Given a QCA $\alpha$ on some manifold $M$, we study a partial order based on Kirby's torus trick (see \cite{kirby69}, and for its appearance in physics \cites{hastings13, hastings18}) so that if $\alpha$ can be defined on $M$ then a closely related $\alpha^\pr$ can be defined on any $M^\pr < M$ in that partial order.

It is worth comparing this proposal with statistical physics. Percolation, for example (see \cite{gg}), can be studied on random Voronoi tesselation of any Riemannian manifold $M$, so there is no question of the existence of the model on $M$. However, the topology of $M$ will effect which observables are present e.g.\ which cycles may or may not percolate \cite{freedman97}. Beyond topology, the geometry will also inform the generic behavior, e.g.\ almost surely there is giant cluster on $\R^d$ but not on $H^d$. As the zoo of topology phases expands \cite{shirley20}, it seems reasonable to look for phases tightly linked to a manifold's (proper) homotopy type using $Q(M)$ and eventually its smooth structure using some version of $\wt{Q}(M)$ described above.

\subsection{QCA pullback under immersion}
In this subsection, $X$ will be a compact $d$-manifold. Manifolds are treated to exploit immersion theory. Compactness ensures that all algebras are finite-dimensional so subfactor inclusions are \emph{trivial} and Wedderburns Theorem may be applied.\footnote{To obtain our immersions we are willing (like Kirby) to remove a point; the resulting non-compactness will not interfere with the finite dimensionality.}

Suppose $i: M^d \looparrowright N^d$ is an immersion between manifolds, with $s(h) \coloneqq G \subset N^d$ a set of points labeled by finite dimensional Hilbert spaces $H_k$. (It is harmless to assume the site map $s$ is one-to-one.) Let $\alpha: \End \ra \End$, $\End \coloneqq \otimes_k \End(H_k)$ be an admissible QCA on $N$. (End is already complete as we are in a finite dimensional context.)

We turn now to Hastings's adaptation of the ``torus trick'' \cite{kirby69}: $M_-$ is the non-compact manifold obtained by deleting a point from a compact $M$. $N$ is also compact. Let $G^{\text{pt}}$ be the pullback in the category of spaces, pt indicating \emph{puncture}.
\begin{equation}
	\vcenter{\hbox{
	\tikz{
		\node at (-1,0.7) {$G^\text{pt}$};
		\draw[->] (-1,0.4) -- (-1,-0.4);
		\node at (-1,-0.7) {$M_-$};
		\draw[->] (-0.6,0.7) -- (0.7,0.7);
		\node at (1,0.7) {$G$};
		\draw[->] (1,0.4) -- (1,-0.4);
		\node at (1,-0.7) {$N$};
		\draw[->] (-0.5,-0.8) to[out=90,in=-45] (-0.55,-0.6) arc (45:225:0.055) to[out=-45,in=180] (-0.6,-0.7) -- (0.7,-0.7);
		\node at (0,-0.5) {$i$};
	}}}
\end{equation}

In \cite{hastings13} a \emph{pullback} $\alpha^\text{pt}: \End_-^{\text{pt}} \ra \End^{\text{pt}}$ over $G^\text{pt}$ is defined provided $r \coloneqq \operatorname{range}(\alpha)$ is small w.r.t.\ the geometry of the immersion $i$, $\End^{\text{pt}} = \otimes_{k^\pr \in G^{\text{pt}}} \End(H_k)$, and $\End_-^{\text{pt}}$ is specified below. The relevant scalar measure of the geometry of $i$ is \emph{loop diameter}, $ld(i) = \inf_\gamma \operatorname{diam}_N(\gamma)$, where $\gamma$ is a double point loop in $N$, that is $\gamma = i(\beta)$, $\beta$ an arc in $M_-$ with $i(\partial_+ \beta) = i(\partial_-\beta)$. Let us explain how $\alpha^{\text{pt}}$ is defined, and then how the puncture is \emph{healed} to result in a ``pulled back'' QCA $\alpha^\pr$ on the compact $M$.

$\alpha^\text{pt}$ will actually just be a $\ast$-homomorphism, not a $\ast$-isomorphism, but the healed\footnote{We use a distinct symbol $G^\pr$ for the sites of the healed $\alpha^\pr$ because, in the end, we think of $G^\pr$ as sites of $M$ not $M_-$.} $\alpha^\pr$ over $G^\pr \coloneqq G^{\text{pt}}$ will be a $\ast$-isomorphism. $\alpha^\text{pt}$ is only defined to act on operators $\Op$ supported on unions of sites $v_{i^\pr} \in G^\text{pt}$, $v_{i^\pr}$ covering $v_i \in G$, which are ``far away'' from the puncture. We have denoted this $\ast$-algebra by $\End_-^\text{pt}$. ``Far away'' means that if $v_{i^\pr} \in \supp(\Op)$, then any path $\gamma$ in $N$ of length $\leq r$ starting at $v_i$ has a lift to a path $\gamma^\pr$ in $M$ starting from $v_i$, $r = \operatorname{range}(\alpha)$. $\alpha^\text{pt}$ is a $\ast$-homomorphism from the $\ast$-algebra generated by such operators into the $\ast$-algebra $\otimes_{k^\pr \in G^\text{pt}} \End(H_{k^\pr})$.

To define $\alpha^\text{pt}$, begin with what it does to operators $\Op$ supported at a single site $v_{i^\pr} \in G^\text{pt}$. The site $v_{i^\pr}$ maps (under the immersion $i$) to $v_i$ and there is an obvious way to pull back the action $\alpha$ to $\alpha^\text{pt}$ acting on $\Op$. For example, if $\alpha$ takes Pauli $X$ at $v_i$ to, say, (Pauli $Y$ at $v_i$) $\otimes$ (Pauli $Z$ at some nearby $v_j$) then $\alpha^\text{pt}$ should be defined to take Pauli $X$ at $v_{i^\pr}$ to (Pauli $Y$ at $v_{i^\pr}$) $\otimes$ (Pauli $Z$ at $v_{j^\pr}$). The condition on paths guarantees that exactly on one $v_{j^\pr}$, at radius $\leq r$ in $G^\text{pt}$ will cover $v_j$, and we use Pauli $Z$ at that site, and similarly for site $v_{i^\pr}$. Now extend to products of such local $\Op$ to keep $\alpha^\text{pt}$ a $\ast$-homomorphism:
\begin{equation}
	\alpha^\text{pt}(\Op_1 \Op_2 \cdots \Op_n) \coloneqq \alpha^\text{pt}(\Op_1) \alpha^\text{pt}(\Op_2) \cdots \alpha^\text{pt}(\Op_n)
\end{equation}

Hastings checks that the order of the products is irrelevant, making $\alpha^\text{pt}$ well-defined. The interesting case is when $\Op_1$ and $\Op_2$ are supported at nearby sites of $G^\text{pt}$. In this case, the commutator $[\alpha^\text{pt} \Op_{1^\pr}, \alpha^\text{pt}\Op_{2^\pr}]$ vanishes since it factors through $[\Op_1, \Op_2] = 0$ due to the disjointness of $v_1$ and $v_2$:
\begin{equation}
	0 = \alpha [\Op_1, \Op_2] = [\alpha\Op_1, \alpha \Op_2] \implies [\alpha^\text{pt} \Op_{1^\pr}, \alpha^\text{pt}\Op_{2^\pr}] = 0
\end{equation}

The final step is to truncate the range of $\alpha^\text{pt}$ so it becomes an isomorphism $\alpha^\text{pt}_-$ and then to complete  $\alpha^\text{pt}_-$ to the desired $\ast$-automorphism $\alpha^\pr$ over all of $M$. This is done using Wedderburn's Theorem: A simple subalgebra $S$ of a finite dimensional complex matrix $\ast$-algebra $\End(V)$ is itself a complex matrix algebra and in the vector representation, the sub-matrix $\ast$-algebra $S$ is realized as $S = \End(V^\pr)$ for some $V^\pr$ a tensor factor of $V$. Clearly, $\End^\text{pt}_-$ has trivial center, so its (injective) image $S$ in $\End^\text{pt}$ is a (simple) factor of a tensor decomposition: $S \otimes T \cong \End^\text{pt}$ (corresponding to, say, $H^\pr \otimes H^{\pr\pr} \cong H$). Define $\alpha^\pr \coloneqq \alpha^\text{pt}_- \otimes \theta$, where $\theta$ is induced from an arbitrary $\ast$-isometry on the $H^{\pr\pr}$ factor, $\theta: H^{\pr\pr} \ra H^{\pr\pr}$. This \emph{heals} the puncture and constitutes an algebraic analog of Kirby's torus trick.\footnote{Kirby used the ``push-pull'' technology of Morton Brown developed in the early 1960's to ``heal'' his puncture. Hastings devised a remarkable algebraic substitute, how it might play out in less trivial von Neumann algebras is an intriguing problem.} $\alpha^\pr$ is arbitrary near the puncture but retains the global features of $\alpha$ and how these features interact with the underlying manifold topology.

Notice that the truncation to $H^\pr$ and hence the tensor decomposition
\begin{equation}
	H^\pr \otimes H^{\pr\pr} \cong H
\end{equation}
is \emph{not} necessarily geometric. $H^\pr$ may fail to have the form $H^\pr = \otimes_{i^\pr \in Z} H_{i^\pr}$, for $Z$ some subset of vertices of $G^\text{pt}$. Thus, the algebraic torus trick has this additional subtlety: The splitting of the algebra over $M$ into tensor product of algebras over a compact portion of $M_-$ and over a neighborhood of the puncture $\EuScript{N}(\text{puncture})$ is not always possible on the spatial level.

\subsection{The pullback partial order}
The construction $\alpha \mapsto \alpha^\pr$ motivates the partial orderings studied in the body of this paper. For $d$-dimensional manifolds, we say $M^d \leq N^d$ if there is a smooth immersion of punctured $M$, $i: M_- \looparrowright N$. What we have just seen is that if $M \leq N$ and $N$ admits a QCA $\alpha$ then $M$ admits a pulled back QCA $\alpha^\pr$. The pullback, though arbitrary near one point of $M$, is functorial and acts, in particular, on the GNVW-index \cite{gross12} $\in H^1(;\Q)$ and also the torsion elements of $Q(T^3)$ and other 3-manifolds, constructed in \cite{fedhh} and \cite{fhh20}.

If $\alpha$ may be regarded as an entangling tool, capable in one shot of creating long range entanglement, then if $M \leq N$ this tool can port QCA and thus certain entangled states from $N$ back to $M$. Among such entangled states are ground states of nontrivial TQFTs.

Our partial order is designed merely to illuminate an existence question: given a QCA $\alpha$ on $N$, can a similar $\alpha^\pr$ be constructed on $M$? Finer questions regarding the \emph{different} ways $\alpha^\pr$ might be situated on $M$, for example, pulled back by distinct immersions $i_1,i_2: M_- \looparrowright N$, are also of physical interest but will not be addressed here. In the future, one might also investigate immersions with additional properties e.g.\ restrict to those that pull back certain tangent bundle reductions.

\subsection{$X$ noncompact}
As we have seen, the noncompact case permits a sharp distinction of scales: infinite vs.\ finite making wrinnkling (in some applications) redundant. This would allow us to define a quotient $Q_{\text{noncompact}}(X)$ of $Q(X)$ as:

\begin{defi}
	$Q_{\text{noncompact}}(X) = \{\text{finite $\range\alpha$ on $X$}\} \slash \{\operatorname{FDQC}(X)$, addition/removal of ancilla$\}$.
\end{defi}

Under this definition, compact components of $X$ behave as if crushed to points, much as in the study of coarse geometry. But, in some cases nothing is lost. For example, by \cite{gross12}, $Q_{\text{noncompact}}(\R) \cong Q(\R)$.

\bibliography{immersions}
\newpage

\section{Appendix on 3-manifolds and integral lifts} 
\label{lift}

\[
\text{by Alan W. Reid}
\]

In \cite{FKKT}, the authors introduce the notion of {\em immersion equivalence} of closed connected smooth manifolds of the same dimension. Although their focus is mainly in dimension $4$, they also discuss the situation in dimensions $2$ and $3$
and \cite[Propositions 2.1 \& 2.3]{FKKT} describes what is known in these settings.  In particular, in dimension $3$, all closed connected orientable $3$-manifolds  are immersion equivalent to $S^3$, with the case of closed connected non-orientable $3$-manifolds admitting at least two equivalence classes: these are represented by $S^1\tilde{\times}S^2$ (the twisted product of $S^1$ and $S^2$) and $\mathbb{RP}^2\times S^1$, and are distinguished by 
whether $w_1(M) \in H^1(M,\Z/2\Z)$ (the first Stiefel-Whitney class of $M$) has an {\em  integral lift} (i.e. whether there is a class in $H^1(M,\Z)$ whose reduction modulo $2$ is $w_1(M)$) or not (see Lemma \ref{no-int-lift}).

As part of their investigation into immersion equivalence of non-orientable closed connected $3$-manifolds, the authors of \cite{FKKT} asked the following question. To state this we recall that a finitely generated group $G$ has {\em rank $d$}, if the minimal cardinality of a generating set for $G$ is $d$. When $G=\pi_1(M)$ where $M$ is a closed connected $3$-manifold, then we abuse notation and also say that $M$ has rank $d$.\\[\baselineskip]
\noindent {\bf Question 1:}~{\em Does there exist a closed connected non-orientable hyperbolic $3$-manifold $M$ of rank $2$ for which $w_1(M)$ does not admit an integral lift?}\\[\baselineskip]
Note that $\mathbb{RP}^2\times S^1$ is $2$-generator. In this Appendix we answer Question 1, namely we show.
%Geometrically, the integral lifting condition is equivalent to the $\Z/2\Z$-Poincar\'e dual surface to $w_1(M)$ being orientable, like the case of $S^1\tilde{\times}S^2$ (the twisted product of $S^1$ and $S^2$).
%In the case of $\mathbb{RP}^2\times S^1$ the $\Z/2\Z$-Poincar\'e dual surface is $\mathbb{RP}^1\times S^1$ which contains an orientation-reversing loop, thereby obstructing an integral lift.  
%Having an integral lift is also equivalent to the $\Z/2\Z$-Poincar\'e dual surface to $w_1(M)$ having odd intersection with orientation-reversing loops, and even intersection with orientation-preserving ones.
%

\begin{thm}
\label{rank2}
Question 1 has an affirmative answer.\end{thm}

More details on the manifold $M$ used to answer Question 1 are given below.  For now we simply note that $H_1(M,\Z) = \Z\times\Z/4\Z$, and so 
from \cite[Proposition 2.3]{FKKT}, and in the notation of \cite{FKKT}, $M \leq \mathbb{RP}^2\times S^1$ but $M$ is not immersion equivalent to $\mathbb{RP}^2\times S^1$. 
In addition, \cite[Proposition 2.3(4)]{FKKT} $S^1\tilde{\times}S^2\leq M$, however using Theorem \ref{rank2} and  \cite[Proposition 2.3(5)]{FKKT} it follows that $M$ is not immersion equivalent to $S^1\tilde{\times}S^2$. Summarizing
we have the following.

\begin{cor}
\label{notequiv}
The manifold $M$ constructed above is not immersion equivalent to $S^1\tilde{\times}S^2$ or $\mathbb{RP}^2\times S^1$, and hence represents a distinct immersion equivalence class of closed connected non-orientable 3-manifolds.\end{cor}

Further motivation for Question 1 is given in \S \ref{compatible}.\\[\baselineskip]
\noindent {\bf Acknowledgement:}~{\em The author gratefully acknowledges the financial support of the N.S.F.  and the Max-Planck-Institut f\"ur Mathematik, Bonn, for its financial support and hospitality during the preparation of this work.
He also wishes to thank M. Freedman and P. Teichner for helpful correspondence and conversations on topics related to this Appendix, and a referee for several helpful comments.}

\subsection{Proof of Theorem \ref{rank2}} 
\label{p:rank2}
We begin by proving a lemma that provides an easy way (in principle) to construct closed connected non-orientable 3-manifolds $Y$ for which $w_1(Y)$ does not admit an integral lift (note this applies to $\mathbb{RP}^2\times S^1$).

\begin{lem}
\label{no-int-lift}
Let  $Y$ be a closed non-orientable 3-manifold satisfying the following properties:

\begin{enumerate}
\item $Y$ is fibered over the circle with fiber a closed non-orientable surface;
\item $H_1(Y,\Z) = \Z\times T$, where $T$ is finite.
\end{enumerate}
Then $w_1(Y)$ does not admit an integral lift.
\end{lem}

\begin{proof} Since $Y$ is non-orientable $w_1(Y)\in H^1(Y,\Z/2\Z)$ is non-trivial. Furthermore, since $Y$ is fibered over the circle with fiber $F$, the normal bundle of $F\subset Y$ is trivial and so standard properties of Stiefel-Whitney classes shows that $w_1(Y)|_F=w_1(F)\neq 0$.

On the other hand, $b_1(Y)=1$ and so the fiber surface $F$ is dual to the unique epimorphism $\phi: \pi_1(Y)\to \Z$  and in particular $\phi(\pi_1(F))=0$. Hence the induced epimorphism $\overline{\phi}\in H^1(Y,\Z/2\Z)$ satisfies 
$\overline{\phi}(\pi_1(F))=0$.  Putting this together with the previous paragraph, we conclude that $\phi$ (which is unique) cannot be an integral lift of $w_1(Y)$, since $w_1(Y)|_F\neq 0$.\end{proof}

\noindent Theorem \ref{rank2} will follow from Lemma \ref{no-int-lift} once we exhibit a closed non-orientable hyperbolic $3$-manifold $M$ of rank $2$ that satisfies the hypothesis of Lemma \ref{no-int-lift}.  There should be many such examples obtained by taking the mapping tori of ``suitable" pseudo-Anosov homeomorphisms on closed non-orientable hyperbolic surfaces.  However, our approach here is more concrete and computational, and makes use of a closed non-orientable hyperbolic $3$-manifold arising in the SnapPy census \cite{CDGW}.\\[\baselineskip]
\noindent{\bf The example:}~{\em Let $M$ denote the manifold ${\tt{m313(1,0)}}$ of the SnapPy census. Then $M$ satisfies the hypothesis of Lemma \ref{no-int-lift}.}\\[\baselineskip]
We will make considerable use of SnapPy \cite{CDGW} in our analysis, and using SnapPy we find that $H_1(M,\Z) = \Z\times\Z/4\Z$, and a presentation for $\pi_1(M)$ is given by: 

\medskip

\centerline{$\pi_1(M)=${\tt{<a,b|aaaBAAAbbaaababb, aaabAAbbaaaBABAAAB>}},}

\medskip

\noindent where $X$ denotes $x^{-1}$. Note that to find $M$ we used the SnapPy command: 

\medskip

\noindent {\tt{for M in NonorientableClosedCensus[:20]: print(M, M.volume())}} 

\medskip

\noindent to list small volume closed non-orientable hyperbolic 3-manifolds. Our manifold $M$ is eighth on this list and has volume
approximately $3.17729327860\ldots$.  
\begin{rem} Note that out of the first eight manifolds in the list of small volume closed non-orientable hyperbolic 3-manifolds we generated from SnapPy,  six have first homology group $\Z$ and so could not be fibered 
over the circle with fiber a non-orientable surface, the other example had first homology group $\Z\times \Z/2\Z$. Using the methods we describe below, we were able to check that this example is fibered with fiber an orientable surface. This is how we were led to $M$ as an appropriate candidate.\end{rem}

\noindent That $M$ has the desired properties follows from Claim 1:\\[\baselineskip]
\noindent{\bf Claim 1:} {\em $M$ is fibered with fiber a closed non-orientable surface.}\\[\baselineskip]
\noindent To prove Claim 1 we first prove:\\[\baselineskip]
\noindent{\bf Claim 2:}~{\em The non-orientable cusped hyperbolic $3$-manifold $N={\tt{m313}}$ has one cusp which has a Klein Bottle cross-section and is fibered with fiber a non-orientable punctured surface.}\\[\baselineskip]
Before commencing with the proof of Claim 2, we include some relevant pre-amble. To show that $N$ is fibered we will use a criterion of Brown \cite{Br} (see also \cite{DT} for a discussion aimed at low-dimensional topologists)
which provides an algorithm to decide if a $2$-generator, $1$-relator group
$G$ equipped with an epimorphism $\phi : G \rightarrow \Z$ has $\ker(\phi)$ finitely generated (which by a result of Stallings \cite{St} in the $3$-manifold setting is equivalent to the manifold being fibered).
Below we state a version of Brown's theorem (see \cite[Theorem 4.3]{Br}) that we shall use. We refer the reader to \cite{Br} and \cite{DT} for further details.

We begin by fixing some notation. Suppose that $G = <a, b~|~R >$ is a $1$-relator group, with $R$ a nontrivial cyclically reduced word in the free group on $\{a, b\}$. Let 
$R_i$ denote the initial subword consisting of the first $i$ letters of $R$, and assume that all the initial subwords of $R$ are given by $R_1, R_2, \ldots , R_n$ with $R_n = R$. Following \cite[Theorem 5.1]{DT} we have:

\begin{thm}[Brown]
\label{t:brown}
In the notation above, let $\phi: G \rightarrow \Z$ be an epimorphism and assume that
$\phi(a)$ and $\phi(b)$ are both non-zero, then $\ker(\phi)$ is finitely generated if and only if the sequence $\phi(R_1), \phi(R_2), \ldots , \phi(R_n)$ has a unique minimum and maximum. \end{thm}

\noindent{\bf Proof of Claim 2:}~Using the SnapPy command ${\tt{M.cusp\_info()}}$ shows that $N$ has a single cusp with cusp cross-section a Klein Bottle (which we denote by $\mathcal{K}$). Also using SnapPy 
we get $H_1(N,\Z) = \Z\times\Z/4\Z$ and a presentation:

\medskip

\centerline{$\pi_1(N)=${\tt{<a,b|aaaBAAAbbaaababb>}}.}

\medskip

\noindent We will apply Brown's algorithm to the unique epimorphism $\phi:\pi_1(N)\rightarrow\Z$. One can check that $\phi(a)=-1$ and $\phi(b)=1$, and so to apply Theorem \ref{t:brown} we need to check the values on the initial subwords of the relator $R:=$ {\tt{aaaBAAAbbaaababb}}.

We identified $16$ initial subwords: 
$$R_1=a, R_2=a^2, R_3=a^3, R_4=a^3b^{-1},\ldots, R_9=a^3b^{-1}a^{-3}b^2, \ldots, R_{16}=R.$$
The values of $\phi(R_i)\in \{-4, -3, -2, -1, 0, 1\}$ with $\phi(R_4)=-4$ a unique minimum and 
$\phi(R_9)=1$ a unique maximum. Hence Theorem \ref{t:brown} applies to show that $\ker(\phi)$ is finitely generated, and so we conclude that $N$ is fibered using \cite{St}.

To show that the fiber is non-orientable we also make use of SnapPy to get a description of the faithful discrete representation into $\rm{Isom}(\mathbb{H}^3)$ which we identify with $\rm{O}^+(3,1)$. To that end we 
use the SnapPy commands ${\tt{G=M.fundamental\_group()}}$ together with {\tt{G.O31(`a')}} and {\tt{G.O31(`b')}} to exhibit matrices in  $\rm{O}^+(3,1)$ representing $a$ and $b$ whose determinants can then be checked to be (approximately) $+1$ and $-1$ respectively. In particular the element $b$ is orientation-reversing, as is the element $ab\in\ker(\phi)$ (from above). Hence the fiber is non-orientable and the proof of Claim 2 is complete.\qed\\[\baselineskip]
\noindent{\bf Completing the proof of Claim 1:}~Since $N$ is non-orientable with cusp cross-section $\mathcal{K}$, there is a unique way to compactify $N$ by gluing a solid Klein Bottle to $\mathcal{K}\times [0,1)$, which by definition gives $M$.
Indeed, there is a unique essential simple closed curve (slope) $\alpha \subset \mathcal{K}$ that is the core of the solid Klein Bottle that can be attached to $\mathcal{K}\times [0,1)$ which
compactifes $N$ via Dehn filling and for which $\alpha$ is killed. Now the slope $\alpha$ is the boundary slope of the fiber surface of $N$ (this is made explicit in Remark \ref{slope} below), and so in particular the fibering of $N$ extends after Dehn filling along $\alpha$ to get $M$.  Hence 
$M$ is fibered with fiber a closed non-orientable surface as required.\qed

\begin{rem} \label{slope} Using SnapPy the filling curve $\alpha$ can be identified as the peripheral element \hfill\break
{\tt{BAAABaaabAAbbaaaBA}}$~\in\pi_1(N)$ which can easily be seen to be a re-ordering of the second relation in $\pi_1(M)$. It can also be readily checked that
this element of $\pi_1(N)$  maps trivially under the map to $\Z$.\end{rem}

\subsection{Compatibility}
\label{compatible}
One motivation for Question 1 is the following result %of Freedman 
which needs some additional terminology.  Let $M$ and $N$ be closed connected non-orientable $3$-manifolds, and assume that $w_1(M)$ does not admit an integral lift.
%with orientation double covers $M^+$ and $N^+$, 
Say that a homomorphism $\theta:\pi_1(M)\rightarrow \pi_1(N)$ is {\em $w_1$-compatible} if and only if the following diagram commutes:

\medskip

\begin{tikzcd}[column sep=small]
\pi_1(M) \arrow[rr, "\theta"] \arrow[rd, "w_1"'] &   & \pi_1(N) \arrow[ld, "w_1"']  & \\
 & \Z/2\Z & 
\end{tikzcd}

\noindent Note that compatibility implies that $w_1(N)$ also does not admit an integral lift. %%AR: think thats correct!

\begin{thm} %%AR: reworded
\label{mike}
Assume that $M$ and $N$  are closed non-orientable hyperbolic $3$-manifolds for which $w_1(M)$ (and hence $w_1(N)$) does not admit an integral lift. Assume that $M$ has rank $2$ and $b_1(N)\geq 3$. Then there is no $w_1$-compatible homomorphism $\theta:\pi_1(M)\rightarrow \pi_1(N)$.\end{thm}

\begin{rem} The manifold $M$ constructed to answer Question 1 provides an example of a closed non-orientable hyperbolic $3$-manifold for which $w_1(M)$ does not admit an integral lift, and so an example to which
Theorem \ref{mike} applies.  This manifold also provides an example of a manifold that is not immersion equivalent to $S^1\tilde{\times}S^2$ or $\mathbb{RP}^2\times S^1$, and hence represents a distinct immersion equivalence class of closed connected non-orientable 3-manifolds.\end{rem}

\begin{proof} The proof given below is essentially that described to the author by M. Freedman.  We begin by proving a lemma which is well known in the orientable case (see \cite[Theorem VI.4.1]{JS}). As in the proof of Theorem \ref{rank2}, $\mathcal{K}$ denotes the Klein Bottle.

\begin{lem}
\label{rank2_structure}
Let $X$ be a closed hyperbolic $3$-manifold and $H<\pi_1(X)$ a subgroup of rank $\leq 2$. Then, either 
$H$ is free of rank $\leq 2$, or $H$ has finite index in $\pi_1(X)$.
\end{lem}

\begin{proof} As stated above, the case when $X$ is orientable follows from \cite[Theorem VI.4.1]{JS}.
We now handle the case when $X$ is non-orientable, and to that end let $X^+$ denote the orientation double cover of $X$ and $H^+=H\cap\pi_1(X^+)$.

If $H$ is not free then by \cite[Corollary 4]{Rat}, $\chi(H)=0$. If $H$ does not have finite index in $\pi_1(X)$, then $H^+$ is an infinite index subgroup of $\pi_1(X^+)$ with $\chi(H^+)=0$. If $C^+$ denotes a compact core for $\H^3/H^+$, then  a standard argument shows that:
$$\chi(\partial C^+)/2 = \chi(C^+) = \chi(H^+)=0.$$
Hence $\partial C^+$ consists of a union of tori ($\mathcal{K}$ is excluded since $C^+$ is orientable), and since $H^+$ has infinite index in $\pi_1(X^+)$, and $X^+$ is a closed orientable hyperbolic 3-manifold, it follows that the only possibility for $H$ is that it is virtually abelian. 
Since $X$ is hyperbolic, $H$ cannot be
free abelian of rank $2$, or $\pi_1(\mathcal{K})$, hence $H\cong \Z$ and this completes the proof. \end{proof}

With this lemma in hand, we complete the proof of Theorem \ref{mike}. To that end assume that there is a $w_1$-compatible homomorphism $\theta:\pi_1(M)\rightarrow \pi_1(N)$, and let $H=\theta(\pi_1(M))$. 
Compatibility and non-orientability of $M$ imply that $H\neq 1$, and so $H$ is a non-trivial subgroup of $\pi_1(N)$ of rank $\leq 2$.

We now consider $H$ in the context of Lemma \ref{rank2_structure}.
If $H$ is a free group of rank $2$, we obtain a sequence of epimorphisms $F_2\rightarrow \pi_1(M)\rightarrow H\cong F_2$. However free groups are Hopfian, and so these epimorphisms are all isomorphisms.
However, the fundamental group of a closed hyperbolic $3$-manifold is not free (since they are irreducible).  The case of finite index is ruled out by the hypothesis that $b_1(N)\geq 3$, so that every subgroup $G$ of finite index in $\pi_1(N)$ also has $b_1(G)\geq 3$ (by the transfer homomorphism), which $H$ clearly cannot.

Thus we are reduced to $H=\Z$ and $\theta$ factors through a homomorphism onto $\Z$. But this is a contradiction to $M$ not having an integral lift since $w_1$-compatibility ensures that any orientation-reversing element of $\pi_1(M)$ must map to an odd integer.\end{proof}

%%%%%%%%%%%%%%%%%%%%%%%%%%%%%%%%%%%%%%%%

%%%%%%%%%%%%%%%%%%%%%%%%%%%%%%%%%%%%%%%%

%%%%%%%%%%%%%%%%%%%%%%%%%%%%%%%%%%%%%%%%

%\bibliography{references}
%
\end{document}